\documentclass[11pt, reqno]{article}

\usepackage{fullpage}
\usepackage{enumerate}
\usepackage{setspace}

\usepackage[colorinlistoftodos]{todonotes}

\usepackage{upgreek, mathtools,bm}

\usepackage{amsmath,amsfonts,amssymb,graphics,amsthm, comment}
\usepackage{hyperref}
\hypersetup{
    colorlinks=true,
    linkcolor=blue,
    citecolor=red,
    urlcolor=blue,
    pdfborder={0 0 0}
}
\usepackage{yfonts}

\usepackage{graphicx}
\usepackage{xcolor}
\usepackage{bbm}
\usepackage{wrapfig} % for figures inserted on the side of the text

\usepackage[font=sf, labelfont={sf,bf}, margin=1cm]{caption}

\usepackage{cleveref}
  \crefname{theorem}{Theorem}{Theorems}
  \crefname{thm}{Theorem}{Theorems}
  \crefname{thm*}{Theorem*}{Theorems}
  \crefname{lemma}{Lemma}{Lemmas}
  \crefname{lem}{Lemma}{Lemmas}
  \crefname{remark}{Remark}{Remarks}
  \crefname{prop}{Proposition}{Propositions}
\crefname{notation}{Notation}{Notations}
\crefname{claim}{Claim}{Claims}
  \crefname{defn}{Definition}{Definitions}
  \crefname{corollary}{Corollary}{Corollaries}
  \crefname{section}{Section}{Sections}
  \crefname{figure}{Figure}{Figures}
    \crefname{assumption}{Assumption}{Assumptions}

\newtheorem{thm}{Theorem}[section]
\newtheorem{thm*}{Theorem*}[section]

\newtheorem{lemma}[thm]{Lemma}
\newtheorem{corollary}[thm]{Corollary}
\newtheorem{prop}[thm]{Proposition}
\newtheorem{defn}[thm]{Definition}

\numberwithin{equation}{section}

\theoremstyle{definition}
\newtheorem{remark}[thm]{Remark}

\def\cZ{\mathcal{Z}}
\def\cY{\mathcal{Y}}

\def\cV{\mathcal{V}}

\def\cT{\mathcal{T}}
\def\cS{\mathcal{S}}
\def\cR{\mathcal{R}}
\def\cQ{\mathcal{Q}}
\def\cP{\mathcal{P}}

\def\cM{M}
\def\cL{\mathcal{L}}

\def\cI{\mathcal{I}}
\def\cH{\mathcal{H}}
\def\cG{\mathcal{G}}

\def\cE{\mathcal{E}}
\def\cD{\mathcal{D}}
\def\cC{\mathcal{C}}

\def\cA{\mathcal{A}}
\def \ve {\varepsilon}

\def \Gd {G^{\#\delta}}

\def\P{\mathbb{P}}
\def\Q{\mathbb{Q}}

\def\E{\mathbb{E}}
\def\C{\mathbb{C}}
\def\R{\mathbb{R}}
\def\Z{\mathbb{Z}}

\def\D{\mathbb{D}}

\def\H{\mathbb{H}}

\def  \p- {p\textunderscore}

\def \bs {\boldsymbol}

\def\eps{\varepsilon}

\def \d {{\# \delta}}

\def\Euc{\textsf{euc}}

\DeclareMathOperator{\Ptemp}{\mathbb{P}_{Temp}}

\DeclareMathOperator{\Pwils}{\mathbb{P}_{Wils}}
\DeclareMathOperator{\Ewils}{\mathbb{E}_{Wils}}

\DeclareMathOperator{\Pwwils}{{\mathbb{P}}_{CRSF}}

\DeclareMathOperator{\Ztemp}{Z_{Temp}}

\DeclareMathOperator{\Zwils}{Z_{Wils}}
\DeclareMathOperator{\Zwwils}{ {Z}_{CRSF}}

\DeclareMathOperator{\dist}{dist}

\newcommand{\red}[1]{{\color{red}{#1}}}

\DeclareMathOperator{\diam}{Diam}

\DeclareMathOperator{\arcch}{arccosh}

\newcommand{\note}[1]{{\color{red}{[note: #1]}}}

\newcommand{\sep}{\mathtt{Sep}}

\def\ul{\underline{\lambda}}
\def\mass{\Lambda}
\def\masss{\mass^{{\not \circlearrowright}}}

\begin{document}
\title{Dimers on Riemann surfaces II:\\
 conformal invariance and scaling limit}

\author{Nathana\"el Berestycki\thanks{University of Vienna. Email: nathanael.berestycki@univie.ac.at} \and Benoit Laslier\thanks{Université Paris Cité and Sorbonne Université, CNRS, Laboratoire de Probabilités, Statistique et Modélisation. Email: laslier@lpsm.paris} \and Gourab Ray\thanks{University of Victoria.  Email:gourabray@uvic.ca }}

\maketitle

\begin{abstract}
Given a bounded Riemann surface $M$ of finite topological type, we show the existence of a universal and conformally invariant scaling limit for the Temperleyan cycle-rooted spanning forest on any sequence of graphs which approximate $M$ in a reasonable sense (essentially, the invariance principle holds and the walks satisfy a crossing assumption). In combination with the companion paper \cite{BLR_Riemann1}, this proves the existence of a universal, conformally invariant scaling limit for the height function of the Temperleyan dimer model on such graphs. Along the way, we describe the relationship between Temperleyan CRSFs and loop measures, and develop tools of independent interest to study the latter using only rough control on the random walk.
\end{abstract}

\tableofcontents

\section{Introduction}

\subsection{Motivations and main results}

This paper is the second of a series of two about the Temperleyan dimer model on Riemann surfaces. As discussed in the first paper of this series \cite{BLR_Riemann1}, the combined result of these two papers is the existence of a universal and conformally invariant scaling limit for the height function associated to the dimer model subject to Temperleyan conditions. We refer to \cite{BLR_Riemann1} for some preliminary discussions and introduction to the problem. 

In \cite{BLR_Riemann1}, the Temperleyan dimer model on a Riemann surface was shown to be equivalent via an extension of Temperley's bijection to a notion which we called Temperleyan cycle-rooted spanning forests, or Temperleyan CRSF for short. These are locally tree-like spanning subgraphs, such that every cycle is noncontractible, and subject to a further topological condition. Furthermore, in \cite{BLR_Riemann1} we reduced the question of convergence of the height function for the dimer model to showing the existence of a universal and conformally invariant scaling limit for this Temperleyan CRSF. This scaling limit result is the heart of the proof, and is in our view the more challenging and also the more novel part of this programme. Its proof is contained in the present paper. Although this result is motivated by our work on the dimer model (as explained in \cite{BLR_Riemann1}), all statements and proofs can be read independently of \cite{BLR_Riemann1}.

Let $M$ be a connected Riemann surface of genus $g$ \footnote{The genus of a connected, orientable real two-dimensional manifold -- which includes Riemann surfaces -- $M$ is the maximum number of simple closed curves that can be drawn on $M$ so that cutting along those curves does not disconnect $M$.} and with $b$ boundary components; we assume that $M$ is bounded and that locally around every point of $M$, there is a local chart homeomorphic to a domain of the complex plane or of the upper half plane (i.e., the boundary components are macroscopic: they are not reduced to a point). Let $\chi = 2- 2g -b $ denote the Euler characteristic of $M$. We suppose that $M$ is neither the sphere nor a simply connected domain (so $\chi \le 0$) and let $ \mathsf{k} = |\chi|$. Fix $\mathsf{k}$ distinct points $x_1, \ldots, x_{\mathsf k}\in M$ called the \textbf{punctures} (if $\mathsf{k} =0$, i.e. if $M$ is a torus or an annulus, there are no punctures). This choice will be made once and for all throughout the paper.  

\begin{figure}[h]
    \centering
    \includegraphics[scale = 0.7]{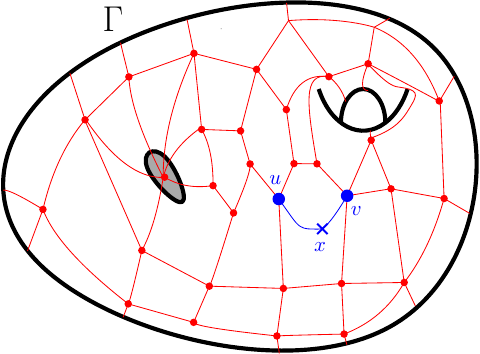}
    \caption{The graph $\Gamma$ in red. Here the surface is a torus with one hole so $g = 1, b =1$ and $\chi = -1$. A single puncture $x$ has been removed from the surface, and a path $e$ connecting $x$ with two adjacent vertices $u$ and $v$ in the face containing $x$ has been drawn in blue.}
    \label{fig:Gamma}
\end{figure}
Fix a sequence of (possibly weighted and directed, simple) graphs $(\Gamma^\d)_{\delta >0}$ embedded in $M$, which conformally approximate $M$ in a suitable sense described more precisely in Section \ref{sec:setup}; for each boundary component add a boundary vertex, and connect it to all the vertices that are adjacent to it in the planar sense (these are the vertices which are adjacent to a face containing the boundary, see Figure \ref{fig:Gamma}). 

For every puncture $x_i \in M$ ($1\le i \le \mathsf{k}$), fix a pair of distinct vertices $(u_{i}, v_i)= (u_i^\d, v_i^\d)$ closest to $x_i$ (i.e., they are adjacent to the face containing $x_i$ -- we assume that no puncture lies on an edge without loss of generality)\footnote{For context, we point out that the puncture corresponds to removing a white vertex in the Temperleyan version of $\Gamma$ which is formed by superimposing $\Gamma$ and its dual $\Gamma^\dagger$. Without referring to $\Gamma^\dagger$, one may think of the puncture being on an edge of $\Gamma$ which is then removed, which is equivalent to the above condition on the location of the puncture. 
See \Cref{rmk:puncture} for further details and references.}.
We recall the notion of wired cycle-rooted spanning forest (introduced in \cite{K_laplacian}, see also \cite{F_laplacian}) and of Temperleyan forests, as introduced in \cite{BLR_Riemann1}. 

\begin{defn} \label{def:temperleyan}
A \textbf{wired, oriented cycle-rooted spanning forest} (or wired CRSF for short) $\mathbf{t}$ on $\Gamma$ is an oriented subgraph of $\Gamma = \Gamma^\d$ such that: 
\begin{itemize}
\item out of every vertex (except boundary vertices) there is a unique outgoing edge. 

    \item all the cycles of $\mathbf{t}$ are noncontractible.
    
\end{itemize}
For any fixed non-boundary vertex $v$, let $\gamma_v (\mathbf{t})$ denote the branch containing $v$, that is, the path obtained by following the forward edges starting from $v$ in $\mathbf{t}$. A \textbf{Temperleyan CRSF} is a wired, oriented CRSF which satisfies the following further topological condition. Consider the branches $\gamma_{u_i} (\mathbf{t}), \gamma_{v_i} (\mathbf{t}), 1\le i \le \mathsf{k}$ emanating out of the $k$ punctures, to which we add a fixed smooth simple path called $e_i$ connecting $u_i, x_i, v_i$ in the face containing the puncture. Let 
\begin{equation}\label{E:skeleton}
\mathfrak{s} = \mathfrak{s}(\mathbf{t}) = \bigcup_{i=1}^{\mathsf{k}} \gamma_{u_i} (\mathbf{t}) \cup e_i \cup \gamma_{v_i} (\mathbf{t}),
\end{equation}
which we call the \textbf{skeleton} of $\mathbf{t}$.
The Temperleyan condition for a CRSF is:
\begin{itemize}
\item $M\setminus \mathfrak{s}$ has the topology of a finite union of disjoint annuli.
\end{itemize}
\end{defn}
The last assertion about $M \setminus \mathfrak{s}$ means that if we cut the surface $M$ along $\mathfrak s$, the surface $M$ decomposes into finite number of components, each of which is  homeomorphic to an annulus. See \Cref{F:pants} for an example.

We stress that $\mathbf{t}$ is oriented even if $\Gamma$ is not. Note that, if $M$ is homeomorphic either to a torus or to an annulus, every CRSF is Temperleyan, since $\mathsf{k}=0$. The notion of CRSF was introduced by Kassel and Kenyon \cite{KK12} who provided a Wilson-type algorithm and considered the scaling limit of its non trivial cycles under rather strong assumptions on the graph. They also related the partition function of CRSFs in which every noncontractible cycle receives a certain weight to determinants of ``twisted'' Laplacians (see \cite{KK12} for details).

There are two measures which we consider in this paper. The first more tractable law (which is the main law we work with in this paper) is the law $\Pwils$ defined as follows:
\begin{equation}
\Pwils ( \cT = \mathbf t) = \frac{1}{\Zwils} 1_{\{\text{$\mathbf t$ Temperleyan}\}} \prod_{e \in \mathbf  t} w(e), \label{eq:pwils}
\end{equation}
where $w(e)$ denotes the weight of the directed edge $e$. This law is tractable because it can be sampled using the Wilson-type algorithm described in \cite{KK12}, although the restriction to Temperleyan forests induces an asymptotically degenerate conditioning. From the point of view of the dimer model however, the most relevant law (see \cite{BLR_Riemann1} for an explanation) is a variant which we call $\Ptemp$ and is defined as follows:
\begin{equation}
\Ptemp ( \cT =  \mathbf t) = \frac{1}{\Ztemp} 1_{\{\text{$\mathbf t$ Temperleyan}\}} 2^{K^\dagger}\prod_{e \in \mathbf t} w(e), \label{eq:ptemp}
\end{equation}
where $K^\dagger$ denotes the number of cycles in the planar dual of $\mathbf{t}$. The reason for introducing $\Pwils$ is technical: it is amenable to Wilson's algorithm. However, as we will see, these measures are very similar for all practical purposes, as their Radon-Nikodym derivative is bounded from above and below. With these definitions, our main result for this paper (stated somewhat informally), is the following.

 \begin{thm}\label{thm:CRSF_universal_intro}
 	Let $M$ be a connected Riemann surface of finite topological type with Euler characteristic $\chi \le 0$, and let $(\Gamma^\d)_{\delta >0}$ be a sequence of graphs embedded in $M$ as above. If the random walk on $\Gamma^\d$ converges to Brownian motion on $M$ and satisfies the assumptions of \cref{sec:setup}, then the Temperleyan CRSF (under either $\Pwils$ or $\Ptemp$) on $\Gamma^\d$ have scaling limits as $\delta \to 0$ in the Schramm topology. Furthermore, the laws of the limits do not depend on the sequence of graphs chosen and are also invariant under conformal transformations.
% 	 \note{Why do we want the moreover part in this version of the theorem ?}\textcolor{blue}{I think we can delete it here. I have grayed it out. But we need to add it somewhere.}
%	 \textcolor{gray}{Moreover, if $K$ denotes the number of noncontractible cycles in either the CRSF or the Temperleyan CRSF, then the tail of $K$ is superexponential: for any $q>1$, there exists a constant $C_q>0$ independent of $\delta$ such that
% 	$$
% 	\E(q^K) \le C_q < \infty.
% 	$$
%	}
 \end{thm}

See \cref{sec:main_result} for a discussion of the Schramm topology. The conformal invariance statement in the above theorem should be interpreted as the limiting law being invariant under conformal bijection between Riemann surfaces with $\sf k$ marked points.
 As mentioned in \cite{BLR_Riemann1}, in combination with the results in that paper \Cref{thm:CRSF_universal_intro} proves the following result:
 
 \begin{thm}\label{T:main_intro}
	Let $M$ denote a bounded, connected Riemann surface, possibly with a boundary, which is not the sphere nor a simply-connected domain. Let $G^\d$ be a sequence of Temperleyan discretisations of $M$ satisfying the invariance principle and a crossing estimate\footnote{The most important of these assumptions are also described in \cref{sec:setup} here; compared to those, the assumptions in \cite{BLR_Riemann1} also includes assumptions on the winding of edges in the embedding of the graph and assumptions on the embedding of the dual graph.} for random walk described in Section 2.4 of \cite{BLR_Riemann1}. Let $h^\d$ denote the height function of the dimer model on $G^\d$. Then $h^\d - \E(h^\d)$ has a limit in law as $\delta \to 0$, which is independent of the sequence $G^\d$ and independent of the sequence $G^\d$ and is invariant under conformal transformations.
\end{thm}
Using Proposition 5.10 in the companion paper \cite{BLR_Riemann1}, the limiting field restricted to a small enough neighbourhood of $M$ away from the punctures can be coupled to the restriction of a Gaussian free field in a slightly bigger neighbourhood, with positive probability; in particular this field is nontrivial. We conjecture that it is a \textbf{compactified Gaussian free field} with appropriate parameters depending only on the surface $M$ and the positions of the punctures $(x_1, \ldots, x_k)$ on $M$.

The precise assumptions about $M$ and the graph embedding can be found in \cref{sec:surface_embedding,sec:setup} and a precise version of \cref{thm:CRSF_universal_intro} can be found in  \cref{sec:main_result}. 
We refer to \cite{BLR_Riemann1} for the meaning of various words in \cref{T:main_intro}, such as Temperleyan discretisation, height function (which is really a closed one-form on $M$ rather than a function) and the topology with respect to which the convergence holds, as well as a precise version of the statement (see in particular Theorem 6.1 in \cite{BLR_Riemann1}).

\begin{comment}
%%%NB removed this on 9 Jan 2024
\red{Let us briefly discuss the \emph{uniform crossing assumption}, as this is one of the key assumptions used widely throughout the article; more precisely let us recall from \cite{BLR16} what this notion means in the simply connected setup. Suppose $\Gamma^\d$ is a sequence of graphs approximating a domain $D \subset \C$. The uniform crossing assumption is that for any rectangle of fixed aspect ratio and length at least $\delta/\delta_0$ in $D$, a random walk crosses it from left to right without exiting the rectangle with probability at least a constant which is uniform in $\delta$ and depends only on the aspect ratio of the rectangle. In our analysis, this assumption (or rather a version of this in the Riemann surface setup) essentially replaces precise random walk estimates on regular graphs like isoradial graphs. See \Cref{sec:setup} for details.}
\end{comment}

\medskip Furthermore, as part of our proof, we obtain the following result which may be of independent interest, and which gives the existence of a universal, conformally invariant scaling limit  for the oriented wired CRSF, under mild assumptions on the graph, thereby generalising the result of Kassel and Kenyon \cite{KK12}. Let $\Pwwils$ denote the law on CRSFs as in \eqref{eq:pwils} but without the indicator that $\mathbf{t}$ is Temperleyan:
\begin{equation}
    \label{eq:law_CRSF}
    \Pwwils ( \cT =  \mathbf t) = \frac{1}{\Zwwils}   \prod_{e \in \mathbf t} w(e), 
\end{equation}

\begin{thm}\label{thm:CRSF_nopuncture_intro}
In the same setup as above, the wired, oriented CRSF under $\Pwwils$ on $\Gamma^\d$ has a scaling limit as $\delta \to 0$ in the Schramm topology. Furthermore, the law of the limit does not depend on the sequence of graphs chosen. Consequently, this limit law is also conformally invariant.
\end{thm}

See \cref{thm:CRSF_universal} for a precise statement.
%We mention here quickly that while we will provide a fully self contained definition of Temperleyan CRSF, we will \emph{not} motivate that definition further by proving that they are in bijection with the dimer model in this paper and refer to the companion paper for this part. 

Our assumptions on the graph are much weaker than those in \cite{KK12} {(in particular, see the definition of \emph{Conformal approximation} in Section 4.1 of that article)}. This is closely related to the fact that the approach in \cite{KK12} is algebraic (in particular, the authors rely on a deep theorem of Fock and Goncharov \cite{FockGoncharov}, see \cite[Lemma 16]{KK12}).  Our approach in contrast is purely probabilistic, and uses a novel connection to loop measures, for which estimates of independent interest are also obtained. An example is the following result (stated informally in this introduction):
\begin{thm}\label{T:loopsoupRWintro}
Suppose $\Gamma = \Gamma^\d$ is a sequence of planar graphs approximating a domain $D \subset \C$ and suppose the random walk on $\Gamma^\d$ satisfy the crossing assumption of \cite{BLR16}. Let $\cA$ be a set of macroscopic loops in a planar graph $\Gamma$ satisfying a crossing assumption. We write $\Lambda$ for the loop measure in $\Gamma$, $X_y$ for the random walk started from a point $y \in \Gamma$ and $\ell( X_y)$ for the loops erased in the loop erasure of $X_y$. Then the following holds uniformly in $\delta$.
{	\begin{enumerate}
		\item 
  For all $\eps, r > 0$, there exists $C < \infty$ such that, if $y$ is a point at distance at least $r$ such that $\P( \ell( X_y) \cap \cA = \emptyset) \geq \eps$, then $\Lambda( \cA) \leq C$.
		\item For all $\eta>0, r>0$, there exists $\eps>0$ such that, if for all $y$ at distance at least $r$ from the boundary, $\P( \ell( X_y) \cap \cA \neq \emptyset) \leq \eps$ then $\Lambda( \cA)\leq \eta$. 
	\end{enumerate}
 }
\end{thm}
In other words, it is enough to understand the loops coming from a single random walk to bound the loop measure. See \cref{sec:loop_measure} and more precisely \cref{T:loopsoupRW} for a precise statement. %In our opinion, this result and our approach to loop measure is interesting because it bypasses any need for the precise green function estimates used typically to study the loop-measure on $\Z^2$. This is what allows us to work with very weak assumptions on the underlying graphs.
In particular, unlike in previous works such as \cite{Lawler2sided}, we require no fine estimates on the Green function, which are unavailable on general Riemann surfaces to the best of our knowledge. 

\medskip It is well-known that loop measures are closely connected to uniform spanning trees, and it is therefore natural to try and develop an analogous connection for CRSFs and Temperleyan CRSFs. This is however far from straightforward, and we will see that this will require a new notion for loops which we call \textbf{$\bs \eta$-contractibility}. See \cref{sec:loop_measure} and in particular \cref{density_non_contractible} for more details.

\medskip The price to pay for this generality is that the theorems do not provide any concrete description of the scaling limits. It is unfortunately a feature of our method that we have very little in terms of usable description of the limit, either through a variant of the Loewner equation suited to a surface or through an SLE partition function and Radon-Nikodym derivative with respect to a reference measure. However we think that it is also a strength since our approach separates the proof of the existence of a scaling limit from its identification, and let us note that it uses neither any observable nor any exact combinatorics. We also use only weak assumptions on the underlying sequence of graphs. See for example \cite{ray2021quantitative} where a similar quenched version of {\Cref{T:main_intro}} was proved for dimers in a random environment in the plane, a notable example being high-density Poisson--Voronoi triangulations. (We believe a similar result will hold on surfaces too). In fact, as highlighted in the statement of \cref{thm:CRSF_universal_intro}, conformal invariance appears as a corollary of this quite strong form of \emph{universality} which is in agreement with the general intuition in theoretical physics.

% \begin{figure}
%     \centering
%     \includegraphics[width = .7\textwidth]{Diagram1.jpg}
%     \caption{A schematic view of the generalised Temperley bijection in \cite{BLR_Riemann1}.}
%     \label{fig:schematic}
% \end{figure}

\begin{figure}
    \centering
    \includegraphics[width = 1\textwidth]{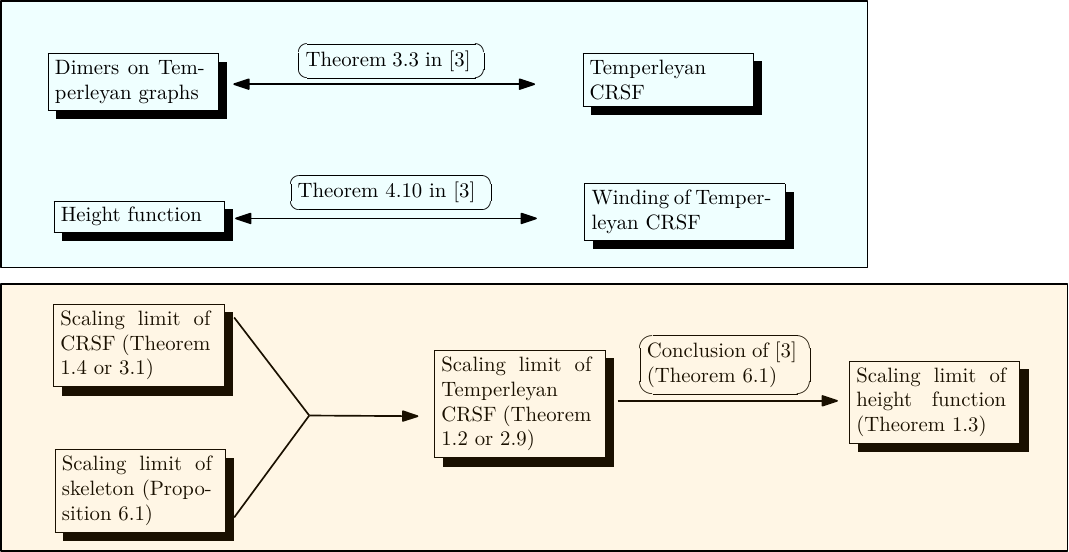}
    \caption{\textbf{Top} (in cyan): A schematic view of the generalised Temperley bijection in \cite{BLR_Riemann1}. \textbf{Bottom} (in orange): A schematic view of the proof of Theorem \ref{T:main_intro} and other main results in this paper.}
    \label{fig:schematic}
\end{figure}

\medskip We conclude by mentioning that a notion of SLE$_2$ loop measure on general Riemann surfaces  is also present in the article \cite{BD16}, but it is unclear to us how this relates to the noncontractible loops in the CRSF or in the Temperleyan CRSF.

%\note{Add a few words about $\eta$-contractible loops.}\note{Refer/ read Stephen Benoist paper....}

%\note{Should we also discuss in this introduction the fact that in a surface, there does not seem to exists a canonical definition of SLE because of the non-trivial global information ? This can also be added as a remark when we introduce the Wilson and Temperleyan measures.}\noteb{G: I think we already did it above, but still keeping the note in case someone wants to add more}

\subsection{Sketch of the proof and organisation of the paper}

%\note{Come back at the end...}
%\note{I think that we could maybe merge a few sections into a big background and setup section. Something like the current ``background'', ``Temperley's bijection on Riemann surfaces'' (shortened to only contain the definitions and no motivation of these definitions) ``setup'' and possibly ``Wilson's algorithm to generate CRSF''}

\cref{sec:background} presents the precise setup for the rest of the paper as well as a quick reminder of the facts about Riemann surfaces that will be needed later. In the notation of \cref{thm:CRSF_universal_intro}, the full assumptions on the sequence of graphs $\Gamma^\d$ appear in \cref{sec:setup}. %Note that while the link to the dimer model is the main motivation to study the latter, \cref{sec:definitionCRSF} will provide fully self-contained definition that do not refer to the dimer model.

 The first step of the proof is to establish the scaling limit result for the CRSF with no conditioning {\Cref{thm:CRSF_nopuncture_intro}}). The main idea for this step is to bootstrap from the simply connected case by giving a description of Wilson's algorithm where all small loops are already erased without affecting the large scale behaviour. As already mentioned before, our assumptions are considerably weaker than the \emph{conformal approximation} assumption in \cite[Section 4.1]{KK12}, where convergence of the discrete derivatives of the Green's function in the graph to that of the continuous Green's function was assumed. Furthermore, the results in \cite{KK12} only talks about the convergence of the loops, not the whole CRSF in Schramm topology.
The rest of the paper is devoted to the same result for a \emph{Temperleyan} CRSF (i.e., sampled from either $\Pwils$ or $\Ptemp$). Recall that {from Definition \ref{def:temperleyan}} we may see it as a version of the {uniform CRSF, conditioned on an (asymptotically singular) event}. {As it turns out}, roughly speaking this singular conditioning essentially boils down to the requirement that {for each $1\le i \le \sf k$, a pair of paths emanating out of either side of the puncture $x_i$ are disjoint. In other words, the conditioning requires that the branches of the CRSF from two adjacent vertices do not merge for a macroscopic  distance (by `macroscopic' we mean distance at least $c>0$ with respect to a natural metric on the surface, see e.g. \Cref{sec:surface_embedding}) from their starting points. It is easy to see that this makes the laws $\Ptemp,\Pwils$ singular with respect to $\P_{\text{CRSF}}$ in the $\delta \to 0$ limit (however they remain mutually absolutely continuous with respect to one another in the limit).} {To prove the existence of the scaling limit of the Temperleyan CRSF,} we will follow an approach from Lawler \cite{Lawler2sided}, initially used to establish the existence of an infinite, two-sided LERW on $\Z^d$ (this is equivalent to the existence of a local limit for a pair of infinite loop-erasures of random walks under the conditioning that the second walk avoids the loop-erasure of the first one). {Lawler's idea was to write the Radon--Nikodym derivative of the conditioned path with respect to independent paths in term of the \emph{random walk loop measure}. We generalise this approach to the CRSF setup where this derivative can be written in terms of the $\eta$-contractible loops {mentioned earlier and which will be introduced in Section \ref{sec:loop_measure}}. We study this in \cref{sec:loop_measure0}. More precisely, we generalise to surfaces and (Temperleyan) CRSF the expression of the Radon-Nikodym derivative (obtained by Lawler \cite{Lawler2sided}) in \cref{sec:law_branch,sec:law_marginal} and we establish the general bounds on the mass of loops such as \cref{T:loopsoupRWintro} in \cref{sec:finiteness}.}

Our next step (\cref{S:Lawler}) is to obtain the existence of a scaling limit in the simply connected setting already considered by Lawler \cite{Lawler2sided} (the difference being that Lawler's result is a local limit whereas we obtain a scaling limit, and more importantly we obtain these results without fine control on e.g. the Green function). 

Finally we conclude in \cref{sec:localglobal} with the ``local to global argument'', where we compare the true law of a Temperleyan CRSF with the law obtained in the previous step. Again this uses exact expressions for the Radon--Nikodym derivatives and the control of the loop measure from \cref{sec:loop_measure0}. {The proof of the main theorem (Theorem \ref{thm:main_Temp_CRSF}) is complete at the end of Section \ref{S:local_global}. Section \ref{sec:quant_special} contains some technical estimate which follow from the work in this paper and are also needed in \cite{BLR_Riemann1}.}

\subsection*{Acknowledgements}

The authors are grateful to Greg Lawler for some useful discussions regarding \cite{Lawler2sided}. The authors are also grateful to the anonymous referees for making many useful comments which improved the exposition of the paper greatly.

NB's research is supported by FWF grant P33083, ``Scaling limits in random conformal geometry''. GR's research is supported by NSERC 50311-57400. BL's research is supported by ANR-18-CE40-0033 “Dimers”.
The paper was finished while NB was in residence at the Mathematical Sciences Research Institute in Berkeley, California, during
the Spring 2022 semester on \emph{Analysis and Geometry of Random Spaces}, which was supported by the National Science Foundation under Grant No. DMS-1928930.
\section{Background and setup}\label{sec:background}
\subsection{Riemann surfaces and embedding}\label{sec:surface_embedding}
In this article, we work with a {connected} {Riemann surface $M$ (a connected, one dimensional complex manifold)}, %equipped with a Riemannian metric $d_M$ {\color{blue} compatible with the complex structure} 
satisfying the following properties:
\begin{itemize}
\item {The surface $M$ is of finite topological type, meaning that it has finite genus (i.e., finitely many handles) and finitely many boundary components. This means that
  the fundamental group $\pi_1(M)$ is finitely generated.} 
\item  The surface $M$ can be compactified by
  specifying a \textbf{boundary} $\partial M$. We denote by $\overline
  M$  the compactified Riemann surface with the boundary. More
  precisely, every point is either in the interior and hence has a
  local chart homeomorphic to $\C$, or is on the boundary and has a
  local chart homeomorphic to the closed upper half plane $\bar \H$. Also there are finitely many such charts which cover the boundary. Note that this condition implies that $M$ has no punctures. (However for future reference, we note here that we will later introduce punctures on $M$, corresponding with the removed vertices of Temperleyan graphs. The resulting punctured surface will be denoted by $M'$.)
\end{itemize}

We say $M$ is \textbf{nice} if $M $ satisfies the above properties. 

\medskip
{ Given such a Riemann surface $M$, it is possible to equip it with a Riemannian metric $g$ (with associated distance function denoted by $d_M$) which extends continuously to the boundary, and turns $M$ into a smooth Riemannian real manifold of dimension two; with an abuse of notation we still write $d_M$ for the distance function on $\bar M$ induced by this metric. (This Riemannian metric can be constructed by considering the hyperbolic metric on the universal cover of the surface if it has no boundary, see \cite[Section 2.4]{Jost} in combination with the uniformisation theorem (\cite[Theorem 4.4.1]{Jost}) also recalled and discussed below.
If the surface has boundary components, we may apply the same result to the surface obtained by gluing it to itself along the boundary -- also called the double.) 
%\note{Talk about extension to $\bar M$.}
%%In other words, recall that this metric can be represented locally in isothermal coordinates as $e^{\rho}|dz|^2$ for a smooth function $\rho$ which extends continuously to the boundary.

Note that the assumption about the boundary means we exclude surfaces such as the hyperbolic plane. This will simplify certain topological issues later when we deal with the Schramm topology.}

\medskip \paragraph{Classification of surfaces.}

Riemann surfaces can be classified into the following classes depending on their conformal type (see e.g. \cite{donaldson2011riemann} for an account of the classical theory):

\begin{itemize}

\item Elliptic: this class consists only of the Riemann sphere, i.e., $M \equiv \hat{  \mathbb{C}}$

\item Parabolic: this class includes the torus, i.e., $ M \equiv \mathbb T =  \C / (\Z + \tau \Z)$ where $\Im(\tau >0)$, the cylinder $M \equiv \mathbb{C} \setminus \{0\}$, and the complex plane itself (or the Riemann sphere minus a point), $M \equiv \C$.

\item Hyperbolic: this class contains everything else. This includes examples such as the two-torus, the annulus, as well as proper simply connected domains in the complex plane, etc.
\end{itemize}
%\note{This sentence makes no grammatical sense, we say is conformally equivalent to elliptic.}
The proofs in this paper (and its companion \cite{BLR_Riemann1}) are concerned with the hyperbolic case (subject to the above conditions) as well as the case of the torus. So from now on we always assume that $M$ is such a surface. We note that the case of simply connected proper domain in $\C$ is covered in our previous work \cite{BLR16}. %These are representative of the main difficulties that arise. We also remark that, in the case of simply connected domains in $\C$, this result is a  special case of our previous work \cite{BLR16}.

% The other cases in the elliptic and parabolic classes (i.e. the whole complex plane, the punctured plane and the Riemann sphere) can all be treated with identical arguments as in our previous work \cite{BLR16}. We will explain in \cref{sec:remainder} how the proofs carry over for these cases.

\subsection{Universal cover}
\label{SS:universalcover}

{ The universal cover $\tilde M$ of the Riemann surface $M$ will play an important role in our analysis.
Recall that any Riemann surface admits a (regular) covering map by a covering space which is simply connected. This covering space may furthermore be endowed with a conformal structure, with respect to which the covering map is then analytic. Being simply connected, this covering space is (by the uniformisation theorem, see Theorem 4.4.1 and Theorem 2.4.3 in \cite{Jost}) conformally equivalent to a Riemann surface $\tilde M$, which is either the unit disc $\D \subset \C$ (hyperbolic case), the whole plane (parabolic case), or the Riemann sphere $\hat \C$ (elliptic case). Altogether we obtain a map\footnote{For concreteness one can fix a point $x_0$ in $M$ and consider $p$ such that $p(0) = x_0$. In the hyperbolic case fix an arbitrary choice of rotation as well.} $p:\tilde M \to M$ which is both analytic and a regular covering map of $M$. By definition of a covering map, $p$ is a local homeomorphism: for every $z \in M$, there exists a neighbourhood $N$ containing $z$ so that $p$ is injective in every component of $p^{-1}(N)$.}
        
%We further recall here that the covering map acts discretely, in the sense that for every $z \in M$, there exists a neighbourhood $N$ containing $z$ so that $p$ is injective in every component of $p^{-1}(N)$.

{
 Recall further 
 that 
 %the classification of %Riemann surfaces is a corollary of the  Poincar\'{e}--Koebe Uniformisation theorem, which states that every simply connected Riemann surface is conformally equivalent to either the Riemann sphere $\hat \C$ (elliptic case), the complex plane $\C$ (parabolic case), or the unit disc $\D$ in the complex plane (hyperbolic case). Recall further that these three spaces 
 in each of the three cases (elliptic, parabolic, hyperbolic), $\tilde M$
 can be endowed with a metric (compatible with the complex structure) of constant Gaussian curvature equal to $1,0$ and $-1$ respectively. 
 Furthermore, there is a subgroup $F$ of conformal automorphisms on $\tilde M$ acting  properly discontinuously and freely -- i.e., without fixed points --  (see, e.g., \cite[Section 2.4]{Jost} for definitions),
 such that $p$ descends to a conformal equivalence between $\tilde M/F$ and $M$. Put it more simply, $F$ is a discrete subgroup of the corresponding set of M\"obius transformations describing the conformal automorphisms of $\tilde M$, and $M$ is conformally equivalent to  $\tilde M /F$ (i.e., to either $\hat \C / F$, $\C/F$ or  $\D/F$).

%We can apply the P\'{o}incare Koebe Uniformisation theorem to the universal cover $\tilde M$, and using the fact that $\tilde M$ is conformally equivalent to 

%as a corollary, we obtain the  classical \textbf{Riemann uniformisation theorem}: a Riemann surface $M$ is conformally equivalent to $\tilde M / H$ where $H$ is a subgroup of the automorphism group of $\tilde M$ which acts freely and properly discontinuously. 

%Recalling the fact that $\tilde M$ admits a conformally equivalent metric of constant curvature whose group of conformal automorphisms are simply the Mobius maps
%That is, in the hyperbolic case, we write $M = \D/F$ and $p:\D \to M$ is the canonical projection which is then conformal. In the case of the torus, $p: \C \to M$ is the standard projection from the plane onto the torus and is also conformal.
%$\tilde M/F$ where $\tilde M$ is endowed with a metric of constant curvature (0 if $\tilde M  =\mathbb C$ and $-1$ if $\tilde M = \D$) and $F$ is a discrete subgroup of the M\"obius group.

In case of the torus, this discrete subgroup is isomorphic to $\Z^2$ and the generators specify translations in the two directions of the torus. In the hyperbolic case, this class of subgroups is much bigger and are known as \textbf{Fuchsian groups} (see e.g. \cite{donaldson2011riemann} for a general account). This particular representation of a hyperbolic Riemann surface, sometimes called a Fuschian model, will be particularly convenient because it allows us to describe the scaling limits of Temperleryan or cycle rooted spanning forests \emph{in the universal cover}, rather than on the surface itself. This will allow us to import directly a number of the ideas and result from \cite{BLR16} on the simply connected case. The Fuchsian structure itself plays a more technical role, see for instance \cref{app:RW}.}  %{\color{cyan}Alternate: Concretely, since in the Fuchsian representation, the dimer and spanning tree lift to just a planar graph embedded in a disc, we can look at this disc with the Euclidean metric in which the sum of angles around a point is $2\pi$ and re-use all the relations between winding and height function proved in the simply connected case in \cite{BLR16} (see also Lemma \ref{lem:int_to_top} below).}

\subsection{Graph discretization and assumptions}
\label{sec:setup}

%\note{Add a remark to emphasize that most assumptions are about a sequence of graph and therefore do not matter whenever we give an exact statement on a fixed graph (say in the loop-measure section).}\notet{G: where is this remark? }

We say that a graph $\Gamma$ is \textbf{faithfully} embedded in a nice Riemann surface $M$ if it satisfies the following (see \cref{fig:Gamma}):
\begin{itemize}
	\item the embedding is proper, i.e. edges do not cross; %the vertices of $\Gamma^\dagger$ lie inside the faces of $\Gamma$ and vice-versa; 
	\item $\Gamma$ is connected;
	\item $\Gamma$ has a marked vertex (called a boundary vertex) for each component of $\partial M$, {we think of this vertex as located `inside' the hole formed by the boundary (see \Cref{fig:Gamma});}
	\item Every face of $\Gamma$ not containing a boundary vertex has the topology of a disc.
\end{itemize}
Note that with this convention, we can think of each boundary vertex as being `delocalised' along the whole boundary component and indeed later we will consider CRSF with a wired condition on each boundary component.

%\begin{figure}
%	\begin{center}
%		\includegraphics[width = 0.3\textwidth]{temp_example1}
%		\caption{An exemple of a faithfully embedded graph $\Gamma$ in red and the associated dual $\Gamma^\dagger$ in black.}\label{fig:temp}
%	\end{center}
%\end{figure}

We will also consider $\Gamma$ to be endowed with a (continuous time) Markov chain $X$ respecting the graph structure and let us emphasise that we allow nonreversible chains. In other words each oriented edge has a weight $w_{(x,y)}$ with no particular relation between these weights. With a slight abuse of notation, we call this chain the \textbf{random walk on $\Gamma$} from now on. Also with a slight abuse of notation, we will in general think of the trajectory of a random walk (or of its loop-erasure) as the continuous path generated by the edges used in order and we will, depending on the context, identify the trajectory either as a set or a continuous path. We write $X[s,t]$ for the path followed between times $s$ and $t$.
 As mentioned in the previous paragraph, we will analyse the random walk until it hits the boundary, so we will not be interested by the behaviour of the walk after the first time it hits a boundary. We can therefore without loss of generality assume that the outgoing weights from any boundary vertex is $0$. Let us also note that we think of $X$ as a continuous time process for simplicity only since all properties of interest to us concern the geometry of its path up to time change, and so the precise time-parametrisation is completely irrelevant.

The embedded graph $\Gamma$ can naturally be lifted to the universal cover {$\tilde M$ of $M$} and we will call this lift $\tilde\Gamma$. More precisely by the uniformisation theorem (\Cref{SS:universalcover}) {we can always find a regular and analytic covering map $p$ from  $\tilde M$ to $M$ and we will always assume that $\tilde \Gamma$ is obtained by such a lift}. %\footnote{For concreteness one can fix a point $x_0$ in $M$ and consider $p$ such that $p(0) = x_0$ and fix an arbitrary choice of rotation}.
Recalling that $M$ will always be either the torus or a non simply-connected hyperbolic surface, $\tilde \Gamma$ will always be a planar graph embedded either in the plane or the disc. Furthermore $\tilde \Gamma$ is ``periodic'' in the sense that it is invariant under the action of the group $F$. {We lift the weights of $\Gamma$ to turn $\tilde \Gamma$ into a weighted graph, and call 
 the {corresponding} walk $\tilde X$.}

When considering scaling limits, we will take a sequence of graphs $(\Gamma^\d)_{\delta>0}$ embedded faithfully on a fixed nice Riemann surface $M$. Because of this, our assumptions relate to the sequence rather than an individual graph $\Gamma^\d$ for a fixed $\delta>0$, even though we will sometimes omit the dependence on $\delta$ and write $\Gamma$ for $\Gamma^\d$. 
We assume the following about $\Gamma^\d$ apart from it being embedded faithfully (see \cref{sec:surface_embedding}). 

{
Let $p: \tilde M \to M$ be a map which is both analytic and a regular covering of $M$  by its universal cover $\tilde M$, which is either the unit disc $\mathbb D \subset \C$ or the whole plane $\C$ (recall the discussion in  \Cref{SS:universalcover} for the existence of this map.)
 In the end the choice of this covering map does not affect the following assumptions, see \cref{lem:conf_inv_assumption} for a precise discussion about this covering map.}

%, which is not the sphere, the punctured plane or the full plane. Recall also that $\tilde \Gamma^\d$ denote the lift of $\Gamma^\d$ to the universal cover $\tilde M$ (which is either the unit disc $\D$ or the complex plane $\C$). We
%assume the following about $\Gamma^\d$ apart from it being embedded faithfully (see \cref{sec:surface_embedding}). Let $d_M$ denote the Riemannian metric in $M$ and let $d_{\bar M}$ be this metric continuously extended to $\bar M$. We assume that we are given a Markov chain on $\Gamma^\d$ respecting the graph structure which can be non-uniform and nonreversible: in other words, we have some weights $w_{(x,y)}$ for each oriented edge $(x,y)$ in $\Gamma^\d$ and the Markov chain moves from $x$ to $y$ at rate $w_{(x,y)}$ in continuous time. With an abuse of terminology we will call this Markov chain \textbf{random walk} on $\Gamma^\d$. (We think of it in continuous time for simplicity, but all properties of interest to us concern the geometry of its path up to time change, and so the precise time-parametrisation is completely irrelevant.)
%We let $p:\tilde M \mapsto M$ be a \emph{conformal} lift to the universal cover $\tilde M$, which is either the unit disc or the whole plane. In the end the choice of this lift does not affect the following assumptions, see \cref{rem:lift} for a precise discussion about this lift.

%\note{I removed the assumption about the winding and the smoothness of edges since I don't think it plays any role in this paper.} \textcolor{blue}{Added a remark}.

{Recall from \Cref{sec:surface_embedding} that we equipped $M$ with a distance function $d_M$ deriving from a Riemannian metric, which extends to the closure $\bar M$; we still denote this extended distance function by $d_M$.}

\begin{enumerate}[{(}i{)}]
	\item \label{boundeddensity} \textbf{(Bounded density)} { We assume that there exists a constant
	$C$ independent of $\delta$ such that for any $x \in M$, the number
	of vertices of $\Gamma^\d$ in the ball $\{z \in M: d_M(x,z) <\delta\}$ is smaller
	than $C$.}
%	\item \label{embedding} \textbf{(Good embedding)} The edges of the graph are embedded
%	as smooth curves and for every compact set  $K \subset \tilde M$, the intrinsic winding of every edge in the lift $\tilde \Gamma^\d$ intersecting $K$ is bounded by a constant $C=C_K$ depending only on $K$. (Note that this allows edges to wind quite a bit near holes.)
	
	\item \label{InvP} \textbf{(Invariance principle)} As $\delta
	\to 0$, the continuous time random walk $\{\tilde X_t\}_{t \ge 0}$ on
	$\tilde \Gamma^\d$
	started from the nearest vertex to $0$ converges to Brownian motion in the following sense. For every compact set $K \subset \D$ containing $0$ in its interior, let $\tau_K$ denote the exit time of $K$, we assume that for all $K$
	$$
	{( \tilde X_{t})_{0 \leq t \le \tau_K}} \xrightarrow[\delta \to 0 ]{(d) }
	(B_{t})_{0 \leq t \le \tau_K}
	$$
	where $(B_t, t \ge 0)$ is a two dimensional standard Brownian motion in $\tilde M$ (killed when it leaves $\tilde M$, if $\tilde M = \D$) started from
	$0$. The convergence above is in law and using the uniform topology on curves up to parametrisation.
	
	We remark that the above condition is equivalent to asserting that simple random walk from some fixed vertex converges to Brownian motion on the Riemann surface itself up to time parametrisation (see e.g. \cite{hsu2002}). Note that since we are only concerned with the trajectory up to its first exit from a compact set, there is no assumption on the asymptotic behaviour of the random walk on $\tilde \Gamma^\d$.

%	\begin{figure}
%		\centering
%		\includegraphics[scale = 0.5]{crossing}
%		\caption{An illustration of the crossing condition.}\label{fig:crossing}
%	\end{figure}
	\item \label{crossingestimate} \textbf{(Uniform crossing estimate).}
	Let $R$
	be the horizontal rectangle $[0,.3]\times
	[0,.1]$ and $R'$ be the vertical
	rectangle $[0,.1]\times [0,.3]$.
	Let $B_1 :=
	B((.05,.05),0.025)$ be the
	\emph{starting ball} and $B_2:=
	B((.25,.05),.025) $ be the \emph{target ball}.

	{The uniform crossing condition is the following. There exist universal constants $\delta_0,\alpha_0>0$ such that for every compact set $K  \subset  M$, there exists a $\delta_K$ such that for all $\delta \in (0,\delta_K)$  the following is true. Let $\tilde K = p^{-1}(K)$ be the lifts of $K$. Let $R'' \subset K$ be a set of the form $cR+z$, where $c \ge \delta/\delta_0$ and $z \in \R^2$ (i.e. a scaling and translate of $R$). }
 Let $B_1'' = cB_1+z$ and  $B_2'' = cB_2+z$. For all $v \in \tilde \Gamma^\d \cap B_1''$,
	\begin{equation}
	\P_{v}(\tilde X \text{ hits }B''_2 \text{ before exiting } R'')
	>\alpha_0.\label{eq:cross_left_right}
	\end{equation}
	{We emphasise that this crossing condition is defined in the Euclidean metric in the disc, and not the more standard hyperbolic metric when dealing with the universal cover of a hyperbolic surface.

	In what follows, sometimes for a compact set $S \subset \tilde M$, we will write $\delta_S$ to mean $\delta_{p(S)}$ as defined above.}

 { \item \textbf{(punctures)} \label{punctures} Recall that we remove $\sf k = |\chi|$ many points, denoted $x_1,\ldots,x_{\sf k}$, from (the interior of) $M$ where $\chi$ is the Euler characteristic of $M$. Recall the vertices ${(u_i, v_i) =} (u_i^\d,v_i^\d)$ which are fixed closest to $x_i$.
 and assume that both $u_i^\d,v_i^\d$ converge to $x_i$ as $\delta \to 0$. We also assume without loss of generality that 
for all $\delta$, the punctures are sufficiently far (say at graph distance at least some large fixed constant) from the boundary and from each other.

 Finally, we also suppose that the following holds. There exist paths $\gamma_{u_i}, \gamma_{v_i}$ (viewed as a ordered collection of adjacent edges $((x_1,x_2), (x_2,x_3),\ldots, (x_{r-1},x_{r}))$ or as a continuous path depending on context) satisfying the following. The paths start with $x_1 = u_i$ or $x_1=v_i$ respectively for $\gamma_{u_i}, \gamma_{v_i}$, and we assume
\begin{equation}
M \setminus \bigcup_{1 \le i \le \sf k} (\gamma_{u_i} \cup \gamma_{v_i } \cup (u_i,v_i))\label{eq:necessary_cond}\end{equation}
is topologically a disjoint union of annuli.}

\begin{figure}[h]
		\centering
		\includegraphics[scale = 0.5]{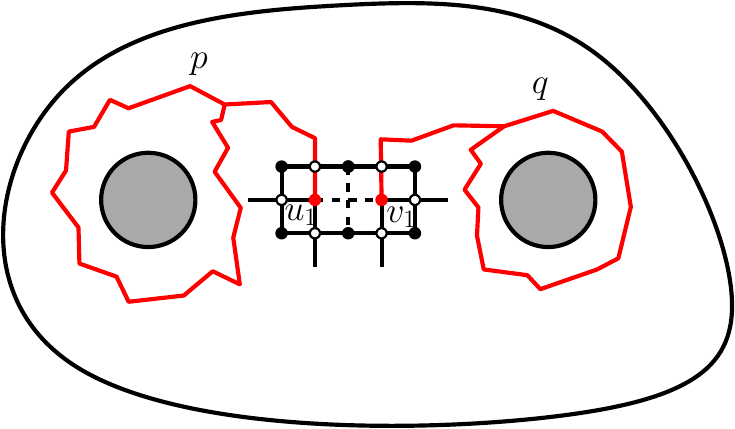}\hspace{.5cm}
		\includegraphics[scale = 0.5]{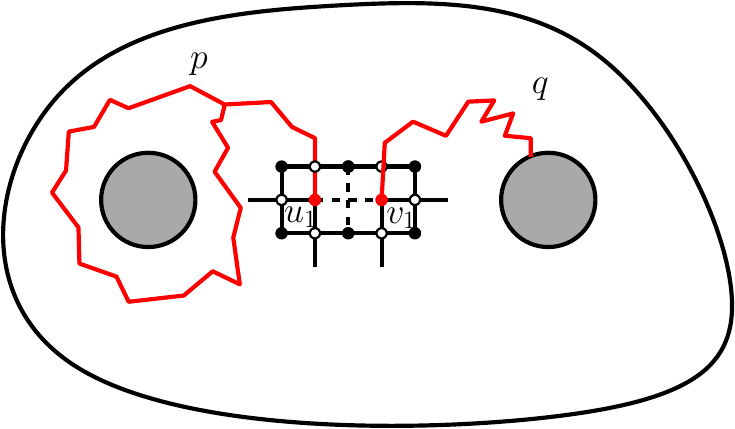}
		\caption{ {An illustration of the assumption (\ref{punctures}). The paths $\gamma_{u_1}, \gamma_{v_1}$ consists of $p$ and $q$, which decompose the surface into a number of disconnected annuli: three in the first case, two in the second.
			The dotted edges are the ones removed from $G$ to get $G'$.}} \label{F:pants}
\end{figure}

\end{enumerate}
%In case $\partial M \neq \emptyset$, recall that the set of boundary cycles $(\partial \Gamma^\dagger)^\d$ corresponds to the connected components of $\partial M$.
{\begin{remark}
Typically the existence of the paths $\gamma_{u_i}$ and $\gamma_{v_i}$ for $\delta $ small enough is guaranteed by the crossing condition and the \emph{pants decomposition} of $M$. However, it is possible to construct pathological examples (for large $\delta$) where there is no paths such that \eqref{eq:necessary_cond} holds, so this needs to be made as a separate assumption.
\end{remark}}
{\begin{remark}\label{rmk:start_point}
		We also remark here that the invariance principle assumption (\ref{InvP}) actually implies something stronger: for any point $x$ in $\tilde M$, the random walk started from a vertex $x^\d$ nearest to $x$ converges to a Brownian motion started from $x$ up to a time change as above. This is a consequence of the fact that random walk from 0 comes close to $x$ with uniformly positive probability (using the crossing estimate) and the strong Markov property of Brownian motion. More precisely, by the strong Markov property, the random walk from 0 cannot visit points $y$ for which the harmonic measure in a given domain $D$ from $y$ deviates by a positive amount from that of Brownian motion (i.e., the probability that it does so tends to zero). Since the walk from zero can make a loop around $x$ with uniformly positive probability,  it follows that there are loops arbitrarily close to $x$ consisting only of points $y$ for which the harmonic measure is close to that of Brownian motion. This implies that the random walk from $x^\d$ has a harmonic measure in $D$ which is also close to that of Brownian motion. Since $D$ is arbitrary, the walk converges to a Brownian motion up to a time-change. (A similar claim is implicit in \cite{YY}.)   
	\end{remark}}
{One consequence of the Invariance principle assumption (item (\ref{InvP})) is that the vertices
adjacent to a boundary vertex converge in the Hausdorff metric (induced by $d_{ M}$) to the associated component of $\partial M$.
}%Sometimes, as already mentioned, we drop the superscript $\delta$ from $\Gamma^\d, (\Gamma^\dagger)^\d$ for clarity which should not cause any confusion.
{\begin{remark}\label{rmk:puncture}
  In \cite{BLR_Riemann1} we needed to consider both a graph $\Gamma^\d$ and its dual $(\Gamma^\dagger)^\d$ faithfully embedded in the surface. The Temperleyan forest that we needed to consider was actually defined on a punctured graph $(\Gamma')^\d$, which is obtained from $\Gamma^\d$ by removing an edge from it (where the puncture or white monomer is located) for each $1\le i \le \mathsf{k}$: see Figure 2 in \cite{BLR_Riemann1}. This is only necessary in order to make the connection with the dimer model and in particular for Temperley's bijection. These minor complications are however irrelevant for this paper, so we dispense {with} them. Still, it may be worth noting for the application to \cite{BLR_Riemann1} that the graph $\Gamma^\d$ here corresponds to $(\Gamma')^\d$ in \cite{BLR_Riemann1}. In the companion article \cite{BLR_Riemann1}, we also need to consider windings of paths and their scaling limits, and to that end we need to assume piecewise smoothness of the embedding  of the graph. However this is not required for the proofs in this paper, so we do not mention it here.
\end{remark}
}

{\begin{lemma}[Conformal invariance of assumptions]\label{lem:conf_inv_assumption}
The assumptions (\ref{InvP}),(\ref{crossingestimate}) on the sequence $(\Gamma^\d)_{\delta>0}$ are invariant with respect to the choice of the covering map $p$, in the sense that if $p'$ is any other analytic and regular covering of $M$ then both the assumptions are satisfied for some (possibly different) choices of $\delta_0,\alpha_0,(\delta_K)_{K \subset M}$.

Furthermore, if $(\Gamma^\d)_{\delta}$ satisfies the assumptions (\ref{boundeddensity}),(\ref{InvP}),(\ref{crossingestimate}),(\ref{punctures}) for a choice of covering map $p$ and constants $C,\delta_0,\alpha_0, (\delta_K)_{K \subset M}$ then these assumptions are invariant with respect to conformal transformations in the sense that if $\psi: (M,x_1,x_2,\ldots, x_{\sf k}) \to (N,x'_1,x_2',\ldots, x'_{\sf k})$ is a conformal bijection mapping $M$ to $N$ and $x_i$ to $x'_i$ for $1 \le i \le \sf k$, then
$(\psi(\Gamma^\d))_{\delta>0}$ also satisfies assumptions (\ref{boundeddensity}),(\ref{InvP}),(\ref{crossingestimate}),(\ref{punctures}) for the choice of the covering map $p\circ \psi$ and constants $C,\delta_0,\alpha_0, (\delta_{\psi^{-1}(K)})_{K \subset M'}$.
\end{lemma}
\begin{proof}
Let us start with the first assertion. Note that by the uniformisation theorem, there exists a conformal bijection (M\"obius map) $\phi: \tilde M \to \tilde M$ such that $p\circ \phi = p'$. Note $\phi$ is a M\"obius map from the unit disc to itself in the hyperbolic case and from the complex plane to itself in the torus case.  In assumption (\ref{InvP}), we only require convergence of the random walk up to time change which is preserved under M\"obius maps.  Assumption (\ref{crossingestimate}) is easily seen to be preserved in the  torus case as a M\"obius map maps a rectangle to another rotated and translated rectangle, which can be crossed by concatenating bounded number of vertical and horizontal rectangles of smaller scales. In the hyperbolic case, a rectangle is mapped to a domain bounded by four circular arcs, and the starting and target discs are mapped to discs inside this domain. Furthermore, compact sets are mapped to compact sets. Thus concatenating domains of this type, it is easy to see that uniform crossing estimate holds for $M$, perhaps for different positive constants $\delta_0, \alpha_0, (\delta_K)_{K \subset M}$.

We now prove the second assertion, for which items (\ref{InvP}),(\ref{crossingestimate}),(\ref{punctures}) are trivial. For (\ref{boundeddensity}), if $M$ has no boundary then $d_M (x,y) = d_N (\psi(x), \psi(y))$ (essentially because M\"obius maps are isometries with respect to the hyperbolic metric on the unit disc). Thus (i) holds. 
If $M$ has a boundary, then note that the conformal bijection $\psi$ between $M$ and $N$ extends to a conformal bijection $\bar \psi$ between the doubles $\bar M,\bar N$; and we are back to the previous case.
\end{proof}
}

\begin{remark}\label{rem:lift} 
 Most of this paper is devoted to the proof of Theorem \ref{thm:CRSF_universal_intro}, for which it is not necessary to consider the lift of the punctured surface $M'$. For the application to Theorem \ref{T:main_intro} it becomes however relevant to consider this lift. While we have stated the assumptions (\ref{InvP}) and (\ref{crossingestimate}) on the universal cover of $M$, {these assumptions are also valid for the sequence $(\Gamma^\d)_{\delta>0}$ on the universal cover of the punctured surface $M'$ as well}. This can be checked using the fact that there is a map from the universal cover of $M'$ to $\tilde M \setminus p^{-1} (\{x_1, \ldots, x_{\sf k}\})$ (where $x_1, \ldots, x_{\sf k}$ is the set of punctures and $p:\tilde M \to M$ is a covering map of $M$) which is locally a conformal bijection. { Also note that (\ref{boundeddensity}) is also satisfied in $M'$ if we endow it with the metric it inherits from $M$ via inclusion.} 
 \end{remark}

{\color{gray} 
%\begin{remark}\label{rem:lift} 
% By the uniformisation theorem of Riemann surfaces, we know that there exists a conformal map from the Riemann surface $\tilde M/ F$ to $M$ where $F$ is a Fuchsian group which is a discrete subgroup of the group of M\"obius transforms on $\tilde M$\footnote{$F$ is discrete if and only if for every $x \in M$, $\exists$ a neighbourhood $V $ of $x$ so that $fV \cap V \neq \emptyset$ for finitely many $f \in F$.}. In the case of the torus, this subgroup is simply a group of translations \blue{of the Euclidean plane} isomorphic to $\Z^2$. In the hyperbolic case $F$ is a subgroup of the group of M\"obius transforms of the unit disc $\D$. Such a conformal map is unique up to conformal automorphisms (i.e. M\"obius transforms) of the unit disc. In other words, if $F,F'$ are two Fuchsian groups such that $M$ is conformally equivalent to both $\D / F$ and $\D/ F'$ then there exists a M\"obius map $\phi: \D \mapsto \D$ such that $F'  = \phi^{-1}\circ F \circ \phi$. Since we have fixed a lift $p$, we have defined $F$ uniquely.
}

 %{\color{red} In summary, we could work with any choice of (regular, analytic) covering map both for the punctured and the non-punctured manifold and we fix a particular choice of this map and call it $p$ (with a small abuse of notation we will use $p$ throughout in both the punctured and unpunctured cases).}
	
	%In summary, we could work with any lift (i.e. any choice of $F$) and we fix a particular choice of this lift and call it $p$ for concreteness.
%\end{remark}
%\note{A weaker assumption would be to assume the same but with $M'$ rather than $M$. Still the above ones are pretty natural and we can just work with that.}

%\medskip

We finish this section with some estimates which are immediate consequences of our assumptions. {This type of estimate is pretty standard in the literature, but we still include a quick proof here for completeness.}

\begin{lemma}[Beurling type estimate]\label{lem:Beurling}
	For all $r,\ve>0$ there exists $\eta>0$ such that for any $\delta <\delta(\eta)$ and for any vertex $v \in \Gamma^\d$ such that $\eta/2 < d_M(v,\partial M) <\eta$, the probability that a simple random walk exits $B_M(v,r ) := \{z\in M: d_M(z,v) <r\}$ before hitting $\partial \Gamma^\d$ is at most $\ve$.
\end{lemma}
\begin{proof}
	Choose an open cover $(U_i)_{i\in I}$ of $\partial M$ so that the elements of the cover are disjoint if they belong to different components of $\partial M$ {(that is, $U_i \cap U_j =\emptyset$ if $U_i \cap \partial M$ and $U_j \cap \partial M$ belong to different connected components of $\partial M$)}. By compactness, choose a finite subcover $\{U_i,\phi_i\}_{1 \le i \le k}$ where $\phi_i$ are charts. Note that $\phi_i$ maps $U_i$ to a subset of the upper half plane $\H$. In the upper half plane, by gambler's ruin estimates, for all $r',\ve$, we can choose an $\eta'$ small enough so that the diameter of the trace of a Brownian motion in $\H$ starting from a point within distance $\eta'$ of the boundary until it hits the boundary is less than $r'$ with probability at least $1- \ve$.
	% Thus by union bound and conformal invariance of Brownian motion, the diameter of a standard Brownian motion starting from any point in $\phi_i(U_i)$ within distance $\eta$ of the boundary, until it hits the boundary is less than $r$ with probability at least $1- \ve$.
	We now use continuity of $\phi_i$ in both directions, all the way to the boundary. Using this, we can choose $r'$ and then $\eta'$ small enough depending on $\phi_i$ so that if $d( \phi_i(x) , \partial \H) \le \eta')$ for all $i$ such that $x \in U_i$ then $d_M(x, \partial M) \le \eta$ and if a set $X \subset \phi_i(U_i) \subset \H $ has (Euclidean) diameter less than $r'$ then also $\diam_M( \phi_i^{-1}(X)) \le r$.
	The proof is now complete by the assumption of invariance principle (\ref{InvP}). \end{proof}

%\begin{remark}
%The reader might also wonder why we did not assume the uniform crossing estimate directly in the manifold. This is because we plan to use the
%loop erased random walk estimates in the plane proved in \cite{BLR16} using the uniform crossing estimate, and we do not want to repeat the exercise in a general manifold.
% Also given a uniform crossing estimate in the universal cover, it is not clear whether uniform avoidance holds in general as for any $\delta$, the behaviour of the random walk outside a large neighbourhood around 0 is unpredictable although $K$ might have infinitely many pre-image in the universal cover. However, uniform avoidance is true for any reasonable embedding of the graph in the manifold.
%\end{remark}
%\item \label{loop_assumption}(\textbf{Uniform avoidance}). Let $K \subset M$ such that $M \setminus K$ remain connected. Let $K' \subset K$ be another compact set such that $K'$ and $K$ are at positive Hausdorff distance from each other. Then there exists a $\delta_{K',K}$ depending only on $K,K'$ such that for all $\delta<\delta_{K,K'}$, random walk starting from a vertex outside $K$ has a positive probability $\alpha_{K,K'}$ to complete a non-contractible loop or hit $\partial G^\d$ before intersecting $K'$.
%\note{This needs careful revisit}

The next lemma says that random walk on $\Gamma^\d$ has uniformly positive probability in $\delta$ of creating a noncontractible loop while staying within a bounded set. The basic point is that the random walk can follow any continuous noncontractible loop in the manifold by crossing finitely many rectangles.

\begin{lemma}\label{lem:uniform_avoidance}
	Let $K_0 \subset K'_0 \subset M$ be open connected sets and let $K,K'$ be compact sets which are the closures of $K_0,K_0'$ respectively. Further assume $K \subset K_0'$. Also assume $K$ contains a loop which is noncontractible in $M$. Then there exists a $\delta_{K,K'}>0$ and $\alpha_{K,K'}>0$ depending only on $K,K'$ such that for all $\delta<\delta_{K,K'}$ and all $v \in \Gamma^\d$ such that $v \in K$, simple random walk started from $v$ has probability at least $\alpha_{K,K'}$ of forming a noncontractible loop before exiting $K'$.
\end{lemma}
\begin{proof}
	In this proof we will use terminology from the crossing assumption (assumption (\ref{crossingestimate})). Consider a curve formed by going twice around the noncontractible loop $\ell$ in $M$
	%\note{We don't know such a loop exists a priori? should it be just in $M$?}
	and let $\tilde \ell$ be the lift of this curve. Using compactness, for every $v \in \tilde \ell$, let $B_v$ be a translation of $cB_1$, where $B_1$ is as in the crossing estimate (assumption (\ref{crossingestimate})) and $c >\delta/\delta_0$ where $\delta < \delta_{K} \wedge \delta_{K'}$ as in assumption (\ref{crossingestimate}). Also pick $B_v$ such that $p(B_v ) \subset K'$.  Let $R_v$ be a rectangle which is a suitable scaling and translation of $R$ and containing $B_v$ so that $B_v$ is the starting ball of $R_v$ and $p(R_v)$ is in $K'$. Clearly $\{B_v\}_{v \in \tilde \ell}$ forms a cover of $\tilde \ell$ which has a finite sub-cover by compactness. From this, it is easy to see that we can move by crossing rectangles from the neighbourhood containing any point of $\tilde \ell$ to a neighbourhood containing its copy, and then form a noncontractible loop. Furthermore, since $p(R_v)$ is in $K'$, we can ensure that this walk will never leave $K'$ with a probability which is uniform over the starting vertex.
\end{proof}

\subsection{Wilson's algorithm to generate wired oriented CRSF}\label{sec:Wilson}

We now describe Wilson's algorithm to generate a wired (but not necessarily Temperleyan)
oriented CRSF on $\Gamma$. We
prescribe an ordering of the vertices $(v_0,v_1, \ldots )$ of $\Gamma$.
\begin{itemize}
	\item We start from $v_0$ and perform a loop-erased random walk until
	a noncontractible cycle is created or a boundary vertex (i.e., a vertex in $\partial \Gamma$) is hit.
	\item We start from the next vertex in the ordering which is not
	included in what we sampled so far and start a loop-erased random
	walk from it. We stop if we create a noncontractible cycle or hit
	the part of vertices we have sampled before.
\end{itemize}
There is a natural orientation of the subgraph created since from every non-boundary vertex there is exactly one outgoing edge through which the loop erased walk exits a vertex after visiting it. 
\begin{prop}\label{prop:unoriented_to_oriented}
	%The subgraph created above has the same law as a wired oriented CRSF with a law given by \eqref{eq:oriented CRSF} in $G$ with boundary $\partial$.
The resulting law is $\Pwwils$ defined in \eqref{eq:law_CRSF}.
%	In particular, $\Pwwils$ generates a wired oriented CRSF of $\Gamma$ as described by \cref{def:CRSF} with law given by \eqref{eq:law_CRSF}. Furthermore, conditional on being Temperleyan, $\Pwwils$ coincides with $\Pwils$.
\end{prop}
\begin{proof}
	This follows from Theorem 1 and Remark 2 of Kassel--Kenyon \cite{KK12}.
\end{proof}

As a corollary, since the order of the vertices can be chosen arbitrarily we can sample a Temperleyan CRSF (under $\Pwils$) by the following procedure. First sample the skeleton $\mathfrak{s}$, i.e a finite number of LERW conditioned to satisfy the condition of \cref{def:temperleyan}. Then complete the rest with the ``standard'' algorithm (i.e., with loop-erasures of random walks without any conditioning: in other words, given the skeleton, the rest of the Temperleyan CRSF is an (ordinary) CRSF in each of the component of $M\setminus \mathfrak{s}$. This fact naturally 
allows us to consider separately the scaling limit of $\mathfrak{s}$ and of the rest of the Temperleyan CRSF. We start with the easier step, which is the second one: we prove in \cref{sec:universal} that CRSFs have a universal and conformally invariant scaling limit.  Then in  \cref{S:Lawler,sec:localglobal} we deal with the skeleton, which is the harder part.

\subsection{Precise statement of the main result}\label{sec:main_result}
We are ready to state more formally the main result of this article. Since the forests under consideration becomes space-filling as $\delta \to 0$, we need to define a suitable topology for this convergence ({{in other words, if we simply use Hausdorff topology induced by $d_M$  as described in \Cref{sec:surface_embedding}, the CRSF converges to $\bar M$ itself}}). Such a topology was already proposed by Schramm in his original paper \cite{SLE}. We call this the \textbf{Schramm topology}; it is defined as follows. {Let $\cP(z,w,\bar M)$ be the space of all continuous paths in $\bar M$
from a point $z\in \bar M$ to $w\in  \bar M$ oriented from $z$ to $w$. 
We consider the Schramm space
$\cS = \bar M \times \bar M \times \cup_{z,w \in \overline{M}} \cP(z,w, \bar M)$. For any metric space
$X$, let $\mathcal H (X)$ denote the space of compact subsets of $X$ equipped with the
Hausdorff metric. We view the Schramm space $\cS$ as a subset of the
compact space $  \overline M \times \overline M \times \cH
(\overline M)$ equipped with its product metric. The Schramm topology is simply the topology generated by the Hausdorff metric on the space of compact subsets of $\cS$, which is  simply denoted by  $\cH(\cS)$.

To any (discrete) CRSF we can associate an element of $\cS$ which encodes all the relevant information of the CRSF as follows:
namely, we consider $\cup_{z,w}\{(z,w,p(z,w))\}$ where the union is over vertices $z,w$ of the graph, and $p(z,w)$ is the oriented path in the CRSF between them, if it exists (recall that from any given vertex there is by definition a single arrow pointing outward, so we include $p(z,w)$ if following the arrows from $z$ eventually leads us to $w$ via the path $p(z,w)$).

 For example if $z,w$ are in a loop, then either there is a path from $z$ to $w$ or from $w$ to $z$, but not both.}
{\begin{remark}\label{rmk:schramm}
    It is standard to check that the Schramm topology is compatible with continuous maps (and consequently, for our applications, with conformal maps which extend continuously to the boundary of $M$). Indeed suppose that $M,N$ are two Riemann surfaces satisfying our assumptions. If $f:\bar M \to \bar N$ is a continuous map and $\cS_M$ and $\cS_N$ are the associated Schramm spaces, then $f$ induces a continuous map between $\cH(\cS_M)$ and $\cH(\cS_N)$. This follows from two observations. Firstly, if $f:(X,d_X) \to (Y,d_Y)$ is a continuous map between two metric spaces, then it induces a continuous map between $\cH(X)$ and $\cH(Y)$. Secondly, if $f_i:(X_i,d_{X_i}) \to (Y_i, d_{Y_i})$ are two continuous maps for $i \in \{1,2\}$ then $(f_1,f_2)$ is a continuous map between the product metric spaces $(X_1 \times X_2 , d_{X_1} + d_{X_2})$ and $(Y_1 \times Y_2, d_{Y_1} + d_{Y_2})$. {Finally if $f$ is a conformal isomorphism between two nice Riemann surfaces $M$ and $ N$
    which extends continuously to the boundary, then $f$ is continuous with respect to $d_M$ and $d_N$.}
    \end{remark}
}
{
\begin{thm}[Universality of the Temperleyan forest scaling limit]\label{thm:main_Temp_CRSF}
Let $M$ be a Riemann surface satisfying the assumptions of \Cref{sec:surface_embedding}, which we equip with the distance $d_M$ described in \Cref{sec:surface_embedding}.
 Let $(\Gamma^\d)_{\delta >0}$ be a sequence of graphs with boundary $\partial \Gamma^\d$ faithfully embedded in $M$ satisfying the assumptions of \cref{sec:setup}.  Then the
limit in law as $\delta \to 0$ of the Temperleyan CRSF sampled using $\Pwils$ or $\Ptemp$ exists
in the Schramm topology relative to $d_M$. Both limits are independent of the sequence $\Gamma^\d$ subject to the assumptions in \cref{sec:setup}. These limits are also conformally invariant.

Furthermore, let $K$ (resp. $K^\dagger$) denote the number of noncontractible loops in the CRSF (resp. its dual). Then for any $q>1$ there exists a constant $C_q>0$ independent of $\delta$ such that
$$\Ewils(q^{K}) \le C_q,$$
 where $\Ewils$ denotes the expectation under $\Pwils$. The same bound holds also with $K$ replaced by $K^\dagger$.
\end{thm}
In \Cref{thm:main_Temp_CRSF}, we emphasise that the set of non-contractible loops can be viewed as a measurable function of the limiting CRSF viewed as an element in the Schramm space, so that $K$ and $K^\dagger$ are indeed random variables under $\Pwils$ and $\Ptemp$. 
Finally, we note that while it is elementary to associate to a discrete CRSF an elemant of the Schramm space (described earlier), for the limiting continuum CRSF our description is rather implicit.  
}
% can always choose to sample first the set of special branches, possibly under the conditioning from \cref{def:temperleyan} if we want a Temperleyan CRSF, and then complete the rest with no remaining conditioning.

%In \cref{lem:CRSF_good}, we show that under the assumptions as in \cref{sec:setup} on our graph $\Gamma'$ with a small mesh size, the probability that a uniformly picked wired oriented CRSF is good is uniformly positive in the mesh size. We finish with a lemma which will be useful later to prove that the probability of a uniformly picked wired oriented CRSF is good converges as the mesh size goes to 0.

\begin{comment}
\begin{lemma}\label{lem:good_measurable}
The event that a wired oriented CRSF of $(\Gamma^\dagger)'$ is Temperleyan is a measurable function of its non-contractible cycles.
\end{lemma}
\begin{proof}
Given all the noncontractible cycles, we can cut open $M$ along these cycles and each component of $\Gamma',M'$ is a nice embedding on a nice manifold with at least two boundary components. The lemma now follows from  \cref{lem:annulus}.
\end{proof}
\end{comment}
%%%%%WAS WRONG!

\section{Universality of cycle rooted spanning forests}\label{sec:universal}

In this section, we prove that the wired oriented CRSF on  a graph satisfying an invariance principle and an RSW-type crossing condition converges to a universal limit in the Schramm sense. This section can be read independently of the rest of the paper.  We only need to recall \cref{def:temperleyan} and the assumptions in \Cref{sec:setup}.

%\note{If instead of removing a vertex, we remove a microscopic chunk with one extra white vertex missing, then we can conclude since then the singularity disappears. But the graph wont approximate $M$. Do we want to remark about this?}

%\medskip

\subsection{Scaling limit and universality of cycle-rooted spanning forest}

 In this section, we give a precise statement and begin the proof of one of the main results of the paper, which shows the existence of a scaling limit for a uniform (oriented) cycle-rooted spanning forest on a nice Riemann surface. The main part of the proof consists in showing the convergence of a finite number of branches, after which a version of Schramm's ``finiteness lemma" (\cref{lem:Schramm_finiteness}) concludes the proof. In this subsection and the next, we deal with the main part of the argument, which is the convergence of a finite number of branches. The conclusion of the proof, based on Schramm's finiteness lemma, will be provided in \cref{finiteness}.

%Note that if we define a metric in $\tilde M$ which is a pushforward of the metric in $M$ by the covering map, then the Schramm topology on $M$ also induces an analogous topology in $\tilde M$. Furthermore, this induced Schramm topology is equivalent to the Schramm topology induced by the Euclidean metric in $\overline \D$ in the hyperbolic case and in $\C$ in the parabolic case. Indeed, this is true because metric can be written as $e^{\rho}|dz|^2$ and by assumption, the metric extends continuously to the boundary in the hyperbolic case.

\begin{thm}[Universality of the CRSF scaling limit]\label{thm:CRSF_universal}
Let $\Gamma^\d$ be a sequence of graphs, all of which are  faithfully embedded in a Riemann surface $M$, and satisfying the assumptions of \cref{sec:setup}. Then the
limit in law as $\delta \to 0$ of the wired oriented CRSF sampled using \eqref{eq:law_CRSF} on $\Gamma^\d$ exists
in the Schramm topology and is independent of the sequence $\Gamma^\d$ subject to the assumptions in \cref{sec:setup}. This limit is also conformally invariant.

Furthermore, let $K^\d$ be the number of noncontractible loops in the CRSF of $\Gamma^\d$. Then for any $q>1$ there exists a constant $C_q>0$ independent of $\delta$ such that
$$\E(q^{K^\d}) \le C_q.$$
\end{thm}

%\note{NB: add exponential bounds on $\mathsf{K}$.}

We explain briefly the ideas behind the proof of this theorem. As explained in \cref{sec:Wilson}, branches of the oriented wired CRSF can be sampled according to a version of Wilson's algorithm: namely, we do successive loop-erased random walks until we hit the boundary or create noncontractible loops. It is not hard to convince oneself that these should have a scaling limit. Indeed, the most difficult aspect of the scaling limit of LERW is to deal with small loops. However, locally, such a loop-erased random walk will behave as if on a portion of the plane where the scaling limit is known. Indeed in this situation, the assumptions on $\Gamma^\d$ in \cref{sec:setup} and a result of Yadin and Yehudayoff \cite{YY} as well as Uchiyama \cite{uchiyama} guarantee the convergence of a small portion of the path towards an SLE$_2$-type curve (we need Uchiyama's result to deal with rough boundaries induced by the past of loop-erased random walk itself). The remaining task is to glue these pieces together.

{Once we have convergence of one branch, by iterating Wilson's algorithm, one obtains convergence of a large but finite number of branches. After this, to obtain weak convergence in the Schramm topology, the idea is similar to Schramm's original article \cite{SLE}: we apply a version of Schramm's finiteness lemma (\Cref{lem:Schramm_finiteness}), which proves that once branches are sampled from finite but dense set of branches, all the remaining branches are of small diameter with high probability (even though there are technically `infinitely' many branches left to sample as $\delta \to 0$.)}
%it simply remains to describe how to erase of keep macroscopic loops.
%These are either contractible in which case they do not matter (i.e., they are erased), or noncontractible in which case keep them. The point is that noncontractible loops are only macroscopic.

Kassel and Kenyon \cite{KK12} considered measures on loops of cycle-rooted spanning forests arising from a generalization of Wilson's algorithm (we use a special case of that algorithm in \cref{sec:Wilson}) on a Riemann surface. They consider scaling limits of the loops associated with this measure. However, their measures are more general as one can assign a certain holonomy to every loop. On the other hand, they assume that the embedded graphs on the surface must \emph{conformally approximate} the surface, in the sense that the derivative of the discrete Green's function converges to the derivative of the continuum Green's function. This assumption is much stronger than our assumption of invariance principle and the crossing assumption.

We also mention here a robust approach to the convergence of loop-erased random walk by Kozma \cite{kozma2002}, which could probably be used as an alternative approach to proving \cref{thm:CRSF_universal}, and could also be useful for instance in extending this result to the case where instead of all boundaries on the surface being wired, we have mixed wired and free boundary conditions.

%\medskip
\subsection{A discrete Markov chain}\label{sec:discreteMC}
Let $\tilde \Gamma^\d$ denote the natural lift of $\Gamma^\d$ to the universal cover $\tilde M$.
To start the proof of \cref{thm:CRSF_universal}, we now consider the {continuous time} simple random walk $\{X_t\}_{t \ge 0}$ on $\Gamma^\d$ starting from some vertex. We are
going to define certain stopping times $\{\tau_k\}_{k \ge 1}$ for this random walk. If $t \ge 0$, we denote by $\cY^t$ the loop
erasure of $X[0,t]$. In other words, if we chronologically erase the loops
of $X[0,t]$ then we obtain an ordered collection of vertices, which we denote by
$\cY^t$. Using compactness and the definition of a cover, we can find a finite cover $\cup_i N_i$ of $\bar M$ so that $p$ is injective in every component of $p^{-1}(N_i)$ for all $i$ {(recall $\bar M$ is the closure of $M$)}. {For $v \in \bar M$, pick an $N_i$ from the finite cover containing $v$ (choose arbitrarily) and call it $N_v$. For $v \in \tilde M$, let $N_v$ be one of the preimages of $N_{p(v)}$ containing $v$.
} For $v,w \in \tilde M$, define $N_{v \setminus w}$ to be $N_v \setminus p^{-1}(B(p(w), \frac12 d_M(p(v),p(w))))$ (this definition is intended to find a large enough neighbourhood of $v$ not containing $w$). The following lemma is immediate now.
\begin{figure}
\centering
\includegraphics[scale = 1.4]{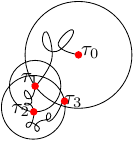}
\caption{An illustration of the stopping times $\tau_k$.}\label{stopping_times}
\end{figure}
\begin{lemma}\label{lem:BM_hit}
Suppose $\partial M \neq \emptyset$. Let $B(t)$ be a standard Brownian motion on $M$ started from some point. Assume $\tau_0=0$ and inductively define $\tau_k$ to be the exit time from $N_{B(\tau_{k-1})}$. Let $\sf N$ be the smallest $k$ such that $B(\tau_k)$ is in the boundary of $M$. Then $\sf N$ is finite a.s. and has exponential tail.
\end{lemma}
\begin{proof}
Since {the closure $\bar M$ of $M$} can be covered by a finite cover $\{N_i\}_i$ there is a uniformly positive probability to go from any of $N_i$ to $N_j$ where $N_j$ is a neighbourhood containing a boundary point (uniformity is over all the neighbourhoods $\{N_i\}_i$). Once the Brownian motion is in a neighbourhood containing a boundary point, it has a uniformly positive probability to hit the boundary before leaving the neighbourhood (for example, using \cref{lem:Beurling} in a weak sense). The geometric tail of $\sf N$ is immediate and the lemma follows.
\end{proof}

%
% For every vertex $z ,v\in \tilde M$, we define
%$$N_v = \cup \{S   :v \in S, p \text{ is injective in } S\}  \subset \tilde M \quad; \quad N_{v \setminus z} = \cup \{S   :v \in S, p \text{ is injective in } S, z \notin S \}.$$
%\note{Actually we need $N_v$ to be simply connected. What's the best choice?}
Let $\partial $ be a specific subset of edges of $\Gamma^\d$. Let $\tilde \partial $ be its lift to $\tilde M$ {(i.e. $\tilde \partial = p^{-1}(\partial)$)}, and assume that the connected components of $\tilde M \setminus \tilde \partial $ are simply connected.
We define a sequence of stopping times $\tau_1, \ldots, $ as follows. We start from an arbitrary vertex $v_0 \in \Gamma^\d \setminus \partial $ and fix one pre-image $\tilde v_0=p^{-1}(v_0)$ in the covering map (in the end this choice is going to be irrelevant). Let $\tilde X$ the unique lift of $X$ starting from $\tilde v_0$. Observe that loops of $\tilde X$ correspond to contractible loops of $X$. On the other hand, noncontractible loops of $X$ do not give rise to loops for $\tilde X$ (we will stop when the first noncontractible loop is formed in the description below.) See \Cref{stopping_times} as a reference.

\begin{enumerate}[{(}i{)}]
 \item Define $\tau_1$ to be the first
time the walk $\tilde X$ leaves $N_{\tilde v_0}$ or intersects $\tilde \partial$.
% \note{not defined yet, I think fine to take $\tilde \partial = p^{-1} (\partial)$? Later we take $\tilde \partial = p^{-1} (\partial \cap M) \cup \partial \D$}.

\item Having defined $\tau_1, \ldots, \tau_k$, we define $\tau_{k+1}$ as
follows. Let $N = N_{\tilde X_{\tau_k} \setminus \tilde v_0}$ and let $\tilde A_k$ be the portion of $p^{-1}(\cY^{\tau_k})$ in $N$
 such that if $\tilde X$ intersects $\tilde A_k$ after $\tau_k$ but before exiting $N$,  a noncontractible loop will have been formed: in other words, $\tilde A_k$ consists of all other preimages of $\cY^{\tau_k}$ intersected with $N$, except the one started from $\tilde v_0$. Note that by assumption of injectivity on $N$, this is the only way a non-contractible loop can be formed.

 We
continue the simple random walk from $\tilde X_{\tau_k} $ until we intersect $\tilde A_k
\cup \tilde \partial $ or exit $N$ and  call that time
$\tau_{k+1}$.\label{main_step}

\item We stop if $\tilde X$ intersects $\tilde A_k \cup \tilde \partial$, otherwise we continue and perform step (ii) again with $k$ changed into $k+1$.
\end{enumerate}

We actually want to see $k \mapsto \cY^{\tau_k}$ as a Markov chain on
curves in $M$ because later we will provide a continuum description of this chain. It is easy to
see that the transitions of that chain are given by the following.
\begin{enumerate}[{(}i{)}]
 \item In the first step, $\cY^{\tau_1}$ is a loop-erased walk from $v_0$
stopped at a time it either exits $p(N_{v_0})$ or intersects $\partial$.

\item Define $A_k = p(\tilde A_k)$. Given $\cY^{\tau_{k}}$, we start a random walk from its endpoint until
we either intersect $A_k \cup \partial$ or exit $p(N)$ where $N = N_{\tilde X_{\tau_k} \setminus \tilde v_0}$ as above. Let $V_k$
be the last vertex of $\cY^{\tau_{k}}$ which is not erased and let $\theta_k$
be the random walk time of the last visit to $V_k$. Let $\gamma$ be
the loop-erasure of $X[\theta_k, \tau_{k+1}]$ and $\cZ^{k}$ be the portion
of $\cY^{\tau_{k}}$ which was not erased by $X[\tau_{k}, \theta_k)$. We have clearly
$$
\cY^{\tau_{k+1}} = \cZ^{k} \cup \gamma.
$$
\end{enumerate}

This transition law can be simplified using the following lemmas. Let
$(Y_i^{\tau_k})_{i \ge 1}$ be the ordering of the vertices of $\cY^{\tau_k}$
which it inherits from the random walk.

\begin{lemma}\label{lem:Vk}
In the above construction, for any realisation of $X$ and $\cY$, we have $V_k
= Y^{\tau_{k}}_m$ with
$$
m = \inf \{ i : Y_i^{\tau_{k}} \in X[\tau_{k},
\tau_{k+1}] \}.
$$
\end{lemma}
\begin{proof}
 This follows from the definition.
\end{proof}
\begin{lemma}\label{lem:independent_RW}
 Conditioned on $V_k, X_{\tau_{k+1}}$ and $\cY^{\tau_{k}}$, the law of
$X[\theta_k,\tau_{k+1}]$ is a random walk from $V_k$
conditioned to first hit $p(N)^c \cup A_k \cup \partial \cup \cZ_k$ at the point $X_{\tau_{k+1}}$.
%to $X_{\tau_k}$ conditioned
%not to intersect $A_k \cup \partial \cup \cZ_k$ or to leave $p(N_{\tilde X_{\tau_k} \setminus \tilde v_0})$ except possibly at the endpoint.
\end{lemma}
\begin{proof}
 This lemma follows from the definition since conditioned on $V_k, X_{\tau_{k+1}}$, the law of
 $X[\theta_k,\tau_{k+1}]$ is uniform over all random walk trajectories starting at $V_k$ and ending at $X_{\tau_{k+1}}$
 which do not intersect $A_k \cup \partial \cup \cZ_k$ or leave $p(N)$
 except at $X_{\tau_{k+1}}$.
 \end{proof}
%\note{change the horrible notation $\eta$}
Clearly, if in the algorithm above we have $\partial = \partial \Gamma^\d$, the resulting curve $ \cY = \cup_{k\ge 1} \cY^{\tau_k}$, either intersects $\partial \Gamma^\d$ or forms a noncontractible loop (this terminates with probability one). We can now repeat the above algorithm this time taking
$\partial$ to be $\partial \Gamma^\d$ together with the curve $\cY$ discovered in the previous step. Repeating this algorithm until all vertices of $\Gamma^\d$ are in $\partial$, this generates a random oriented subgraph of $\Gamma^\d$ which by the generalisation of Wilson's algorithm (Section \ref{sec:Wilson}) has the same law as an oriented CRSF of $\Gamma^\d$.
The above procedure then gives a convenient breakdown of each step of Wilson's algorithm into a Markov chain that is going to have a nice continuum description.

There is one small caveat. We need to show the following:

\begin{lemma}\label{L:simply_connected}
After every step of the above algorithm (i.e. after a full running of the Markov chain until termination), {every connected component of} $\tilde M \setminus \tilde \partial$ is simply connected.
\end{lemma}
\begin{proof}
At the first step, recall that every component of $(\partial \Gamma^\dagger)^\d$ is a noncontractible loop by assumption. It is shown in the appendix of \cite{BLR_Riemann1} that every component of the lift of a noncontractible loop is a bi-infinite simple path (this is a deterministic fact coming from Riemannian geometry). Furthermore these paths connect two boundary points of $\partial \D$ in the hyperbolic case, and go to infinity in a particular direction in the torus case. {Also note that in one step of the algorithm (say starting from $v\in M\setminus \partial$) we add either a simple path from $v$ to $\partial$ or a non-contractible simple loop disjoint from $\partial$, together with a simple path from $v$ to it. Call $\cY$ the added path. In the former case any component of $p^{-1}(\cY)$ has to be a simple path from some $\tilde v \in \tilde M \setminus \tilde \partial$ to $\tilde\partial$ so this cannot create any non-simply connected component. In the latter case, a component of $p^{-1}(\cY)$ will be the union of a simple path connecting two boundary points of $\partial \D$, together with simple paths attached to it. This will divide the connected component of $\tilde v$ in exactly two new components, both of which are simply connected.}
%Since after sampling a branch of the CRSF, either a path connecting to $\tilde \partial $ is added or a noncontractible loop is formed, simple connectedness of $\tilde M \setminus \tilde \partial$ is maintained (one can think of a path connecting \blue{a vertex} to $\partial \Gamma^\d$ to be a path connecting \blue{a vertex} to $\partial M$, i.e. stopped when $\partial M$ is hit, to justify this topological argument). \red{This sentence has clearly not been re-read!}
\end{proof}

\subsection{Continuum version of Wilson's algorithm to generate CRSF.}
\label{SS:wilsoncontinuous}\medskip We now describe the continuum process. The main technical input is a result of Uchiyama \cite[Theorem 5.7]{uchiyama}, which itself is a generalisation of a result of Suzuki \cite{Suzuki}. Suzuki proved convergence of a loop-erased random walk excursion to chordal SLE$_2$ subject to an assumption that the boundary is piecewise analytic (while Yadin and Yehudayoff \cite{YY} dealt with the radial case).

We define a \textbf{discrete domain} to be a union of faces of $\Gamma^\d$ along with the edges and vertices incident to them. We specify certain edges and vertices of the domain to be the boundary of the domain. We say a discrete domain is \textbf{simply connected} if the union of its faces and non-boundary edges and vertices (called its \textbf{interior}) form an open, connected and simply connected domain in $\C$.

%We use a superscript
%$\d$ on discrete walk quantities to emphasise the dependence on the mesh size.
\begin{thm}[Uchiyama \cite{uchiyama}]\label{lem:scaling_limit_main}
Let $D$ be a simply connected domain such that $\bar D \subset \D$. Let $\bar D^\d$ be a sequence of simply connected discrete domains with $D^\d$ being its interior. Let $p_0 \in D$
and suppose that $D^\d$ converges in the Carath\'eodory sense to $D$: if $\phi$ (resp. $\phi_\d$) is the unique conformal map sending $D$ (resp. $D^\d$) to the unit disc $\D$ such that $\phi(p_0) = 0$ and $\phi'(p_0)>0$ (resp. $\phi_\d(p_0) = 0$ and $\phi_\d'(p_0)>0$), then $\phi_\d$ converges to $\phi$ uniformly over compact subsets of $D$.  Suppose that $a^\d, b^\d$ are two boundary points on $\bar D^\d$ (understood as prime ends) such that $\phi_\d(a^\d) \to \tilde a \in \partial \D$ and $\phi_\d(b^\d) \to \tilde b \in \partial \D$ with $\tilde a \neq \tilde b$.

Let $\tilde X^\d$ be a random walk subject to the assumptions in \cref{sec:setup} from $a^\d$ conditioned to take its first step in $D^\d$ and to leave $D^\d$ at $b^\d$. Then the loop erasure of $\phi_\d(X^\d)$ converges to
chordal SLE$_2$ from $\tilde a$ to $\tilde b$ in $\D$ for the Hausdorff topology on
compact sets (and in fact in the stronger, uniform sense modulo reparametrisation).
\end{thm}

\begin{proof}
 This follows from the work of Uchiyama. We emphasise here that Uchiyama uses a locally uniform invariance principle (called hypothesis {\bf(H)} in \cite{uchiyama}, Section 2), which says that for any compact set, a random walk run up to the time it leaves the compact set is uniformly close to a Brownian motion (up to reparametrised curves) and the uniformity is over any starting point in the compact set. This is clearly satisfied in our case, see \cref{rmk:start_point}.
\end{proof}
%\noteb{In \cite{Uchiyama}, a locally uniform invariance principle was used which is clearly verified in our situation using the crossing estimate and the invariance principle assumption.}
%\note{Is this what we wanted to say? Cant remember ... :(}
Note that the above theorem is for the conformal image of the loop erasure of the walk in the unit disc. To transfer the results to the domains of interest, we employ the following corollary. If $D^\d$ is such that $\C \setminus D^\d$ is uniformly locally connected then the conformal map $\phi_\d^{-1}$ extends continuously to $\bar \D$. Furthermore, by assumption on the Carath\'eodory convergence, we know that $\phi_\d^{-1} (z) \to \phi^{-1} (z)$ for all $ z \in \D$. Hence we deduce from the Carath\'eodory kernel theorem (see Corollary 2.4 in \cite{Pommerenke}) that $\phi_\d^{-1}$ converges in fact \emph{uniformly} to $\phi^{-1}$ over $\bar \D$; furthermore it is easy to see that also $\C\setminus D$ is locally connected. Hence we obtain the following corollary:

\begin{corollary}\label{cor:Uchiyama}
Under the assumptions of \cref{lem:scaling_limit_main}, and if $\C \setminus D^\d$ is uniformly locally connected, we have that the loop-erasure of $\tilde X^\d$ converges to chordal SLE$_2$ from $a$ to $b$, where $a = \phi^{-1} (\tilde a)$ and $b = \phi^{-1} (\tilde b)$.
\end{corollary}

% \begin{itemize}
%  Let $D$ be a domain with Holder continuous \note{how much can we relax this?}
% boundary
% % \item We now view the random walk and the loop erasure as continuous
% paths in $\C$ by drawing the edges in the graph which they cross. Let
% $\cY_{\tau_k} = (Y_1, \ldots Y_{m_k})$ and let $\beta^\d_k$ be the first time
% $Y$ enters $B(X_{\tau_k}, 0.5)$ such that it is not part of $A^\d_k$ (and hence
% all its subsequent entries are also not part of $A_k^\d$). Then
% $((Y_{\beta_k^\d}, \ldots , \eta^\d), A_k^\d)$ converges to $(Y_{\beta
% \eta},A_k)$ in law in the Hausdorff topology.
% \item  Suppose we
% condition on $(Y_{\beta_k^\d}, \ldots , \eta^\d), A_k^\d$ and $
% X_{\tau_{k+1}}$. The conditional law of the loop erasure of $\{\eta^\d, \ldots,
% X_{\tau_{k+1}}\}$  converges to a chordal SLE$_2$ in $B(x,0.5) \setminus (A \cup
% Y_{\beta \eta})$ from $\eta$ to $\eta'$.
%\end{itemize}

We now describe a continuum version of the discrete Markov chain described in the previous section.
Suppose we start with $\partial \subset \bar M$ and let $\tilde \partial$ be $p^{-1} (\partial \cap M) \cup \partial \D$ in the hyperbolic case, otherwise in the parabolic case we simply take $\tilde \partial = p^{-1} (\partial)$.
%(note that $p$ extends to a map from the closure of $\tilde M$ to $\bar M$).
We will assume that $\partial$ is such that every connected component of $\tilde M \setminus \tilde \partial$ is simply connected. (This will be almost surely satisfied after every step of the continuum Wilson's algorithm; we verify this in \cref{lem:sc} later though this is essentially the same as the discrete \cref{L:simply_connected} proved above).
%\note{What is the difference between \cref{lem:sc} and ?}
We now define a random curve starting from a point $z \in M$ and ending in $\partial$ or in a noncontractible loop. Fix a pre-image $\tilde z = p^{-1}(z)$ (again, in the end, the choice of pre-image is irrelevant). {See Figure \ref{F:CW} for an illustration. }

\begin{figure}
\centering
\includegraphics[scale = 1]{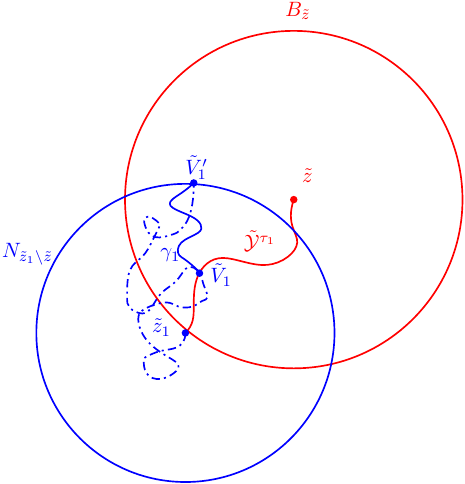}
\caption{{An illustration of the first two steps of the continuum Wilson algorithm. $\tilde \cY^{\tau_1}$ is a radial SLE$_2$ between $\tilde z$ and $\tilde z_1$ (in red). conditional on this, we start a Brownian motion (dashed blue) from $\tilde z_1$ until it leaves $N_{\tilde z_1 \setminus z}$ (or it hits the boundary $\tilde \partial$, not illustrated).  It hits $\tilde \cY^{\tau_1} $ at many points, the earliest of which is called $\tilde V_1$, and leaves $N_{\tilde z_1 \setminus z}$ at $\tilde V'_1$. We then let $\gamma_1$ be an independent chordal SLE$_2$ between $\tilde V_1$ and $\tilde V'_1$ (in solid blue). The next state of the algorithm, $\tilde Y^{\tau_2}$ is obtained by concatenating the portion of $\tilde \cY^{\tau_1}$ between $\tilde z$ and $\tilde V_1$ and $\gamma_1$.}}\label{F:CW}
\end{figure}

\begin{enumerate}[{(}i{)}]
 \item Let $B_{\tilde z}$ be the connected component of $N_{\tilde z} \setminus \tilde \partial$ containing $\tilde z$. Then note that $B_{\tilde z}$ is simply connected since $N_{\tilde z}$ is simply connected, and it is an easy exercise to check that if $A, B$ are bounded, simply connected sets in $\C$ then every connected component of $A \cap B$ is also simply connected (e.g. using Jordan's theorem).
     In the first step, we define $\cY^{\tau_1}$ to be the image under $p$ of a radial SLE$_2$ in $B_{\tilde z}$ targeted to $\tilde z$
from a point chosen from $\partial B_{\tilde z} $ according to the
harmonic measure seen from $\tilde z$.

\item Suppose we have defined the continuum curves up to step $k$ and call
it $\cY^{\tau_{k}}$. Let $\tilde z_k$ be the end point of $\tilde \cY^{\tau_{k}}$ where $\tilde \cY^{\tau_{k}}$ is the unique lift of $ \cY^{\tau_{k}}$ starting from $\tilde z$. Let $\tilde A_k = p^{-1} (\cY^{\tau_{k}}) \setminus \tilde \cY^{\tau_{k}}$, so $\tilde A_k$ consists of all the other preimages of $\cY^{\tau_{k}}$ except for $\tilde \cY^{\tau_{k}}$.
%Let $A_{k}$ be
%the subset of $\cY^{\tau_{k-1}}  \cap p( N_{\tilde z_k \setminus \tilde z})$, so that if we draw a simple curve in
%$M \setminus (\cY^{\tau_{k-1}})$ joining $z_k:=p(\tilde z_{k})$ and $A_k$ then this
%curve along with $\cY^{\tau_{k-1}}$ will create a non-contractible loop.

We start a Brownian
motion independent of everything else from $\tilde z_k$ until it either intersects $\tilde A_k \cup \tilde \partial$ or exits $N_{\tilde z_k \setminus \tilde z}$. Call the point where we stop
the Brownian motion $\tilde V_k'$. Let us parametrise  $\cY^{\tau_{k}}$ using some choice of continuous parametrisation and lift it to $\tilde \cY^{\tau_{k}}$. Let $\tilde V_k$ be the infimum of the set of  points in  $ \tilde \cY^{\tau_{k}}$ which
the Brownian motion intersects (this makes sense for any choice of
parametrisation of $\cY^{\tau_{k}}$ starting from $z$ and ending at $z_k$ and $\tilde V_k$ is independent of
the choice of parametrisation).\label{step2}

\item   Let $\tilde \cZ^{k}$ be the portion of
$\tilde \cY^{\tau_{k}}$ from its starting point to $\tilde V_k$ and let $\cZ^k = p(\tilde \cZ_k)$. Let $D_k$ be the connected component of $N_{\tilde z_k \setminus \tilde z} \setminus (\tilde A_k \cup \tilde \cZ^{k} \cup \tilde \partial)$ containing $\tilde V_k'$.
An argument similar to the above can be used to show that $D_k$ is simply connected since $\tilde z \notin N_{\tilde z_k \setminus \tilde z}$ and $p$ is injective in $ N_{\tilde z_k \setminus \tilde z}$.

%\note{This is tedious. One way is to note one end point of each path in $\tilde A_k$ is not in $N$ and use that endpoint to continue path to infinity, avoiding everything else. Something similar for $\tilde \cZ_k$. Removing these paths give a simply connected set, so we are back to the intersection argument. We propose to comment this out.}
Now define a chordal SLE$_2$ curve
$\gamma_k$ in $D_k$ from
from $\tilde V_k$ to $\tilde V_k'$ independent of everything else. Define $$\cY^{\tau_{k+1}}:=
\cZ^k \cup p(\gamma_k).$$
\item We stop if $\cY^{\tau_{k+1}}$ contains a noncontractible
loop or touches $\tilde \partial$. Otherwise, we return to step \ref{step2}.
\end{enumerate}
We now prove the following lemma which says that the paths $\cY^{\tau_k}$ converge to a limiting path $\cY^\infty := \lim_{k \to \infty} \cY^{\tau_k}$.
This will correspond to a branch of a continuum CRSF. In fact we will prove that $\cY^\infty$ is a union of finitely many elements in the union. Recall that we are not considering the whole complex plane case and the sphere case in this article.
\begin{lemma}\label{lem:MC_branch}
There exists a random variable $N$ with exponential tail such that $\cY^\infty := \cY^{\tau_N}$.
\end{lemma}
\begin{proof}
We claim that we can couple a standard planar Brownian motion $B(t)$ from $\tilde z$ with the above Markov chain such that the endpoint of $\tilde \cY^{\tau_k}$ (i.e., $\tilde z_k$) is $ B(t_k)$ for some increasing sequence $(t_k)_{k \ge 1}$. {Indeed,  we can sample a Brownian motion   from $\tilde z_k$ (independently of everything else), and run it until it either intersects $\tilde A_k \cup \tilde \partial$ or exits $N_{\tilde z_k \setminus \tilde z}$}. By the strong Markov property, if we concatenate these Brownian paths, the whole trajectory has the law of a Brownian motion with the required property.

%Since we are dealing with the non-simply connected case, every component after each step of the Wilson's algorithm is not simply connected, and therefore contains a non-contractible loop $\ell$. The proof now follows in the same way as in \cref{lem:uniform_avoidance}: we can enforce the Brownian motion to follow the loop with probability uniform over the starting vertex.

Let us consider the case where $\partial M \neq \emptyset$, which means we are in the hyperbolic setting. The fact that $N$ has exponential tail is now an easy consequence of  \cref{lem:BM_hit}.

Otherwise, if there is no boundary, we can cover $M$ by finitely many neighbourhoods, take a noncontractible loop $\ell$ and join a point from every neighbourhood to $\ell$. Using an argument similar to \cref{lem:uniform_avoidance}, we see that there is a positive probability for the Brownian motion to follow a path to $\ell$ and then move along $\ell$ to form a noncontractible loop for the Markov chain, and this probability is uniformly bounded below over the starting point. This event can happen in every step of the Markov chain with uniform positive probability, thereby concluding the proof.
\end{proof}

The above algorithm gives a recipe to sample one branch. To sample finitely many branches $\mathsf B_1,\mathsf B_2,\ldots$ from points $z_1,z_2,\ldots$, we continue sampling branches by updating $\tilde \partial$ to include the portion of the CRSF sampled in the previous step and applying the previous algorithm. Recall that we require $\tilde M \setminus \cup_{j=1}^k \mathsf B_j$ to be simply connected at every step to make sense of the algorithm. The proof is exactly the same as in \cref{L:simply_connected} which we now record.
\begin{lemma}\label{lem:sc}
For every $k$, every component of $\tilde M \setminus \cup_{j=1}^k \mathsf B_j$ is a.s. simply connected.
\end{lemma}
%\begin{proof}
%%The proof follows from induction. Indeed in every step, in every connected component, we either add a simple path joining a point  to the boundary, or in case a noncontractible loop is created, we add a bi-infinite path which is unbounded in the parabolic case and it has both endpoints in $\partial \D$ in the hyperbolic case. In both cases, every component of the complement remains simply connected. Since the algorithm is completed in a.s. finite number of steps by \cref{lem:MC_branch}, the lemma follows.
%\end{proof}
%

%Clearly if we start from $\partial = \emptyset$ and perform the above algorithm, either the endpoint of the curve $(\cY^{\tau_k})_{k \ge 1}$ converges to $\partial M$ or a non-contractible loop is created as $k \to \infty$. This is justified since we have chosen $N_v$ so that the covering map on it is injective. We can now continue the procedure by
%defining $\partial$ to be the curve we have discovered in this step. This gives
%us a step by step procedure to sample a portion of the continuum CRSF.

%\note{Remark: We have to prove that the portion of the simple random walk  is a random walk bridge and is independent of
%everything conditional on $V_k,V_k'$, $\cY^{\tau_{k-1}}$}

\begin{thm}\label{thm:main_tech}
Let $\partial^\d$ be a set of edges in $\Gamma^\d$ and assume that $\partial^\d$ converges in the Hausdorff sense to some set $\partial \subset M$. Let $\tilde \partial $ be the lift of $\partial$ to $\tilde M$. Assume $\tilde M \setminus \tilde \partial$ is locally connected and all the connected components of $\tilde M \setminus \tilde \partial$ are simply connected.
Then the Markov chain $(\cY^{\tau_k})^\d$ converges to the Markov chain $\cY^{\tau_k}$ as $\delta \to 0$.
More precisely this means that for any $k \ge 1$, the joint law of
$((\cY^{\tau_j})^\d:1 \le j \le k)$ converges to the joint law of $(\cY^j: 1
\le j \le k)$ (the convergence is in product of Hausdorff topology).
\end{thm}
 We are going to use induction and prove at the same time that $\cY^{\tau_k}$ is a simple curve a.s. at every step $k$. We use the notations used in the description of the continuous and the discrete Markov chains. In the first step, the proof is just an application of the result of Yadin and Yehudayoff \cite{YY} and the fact that radial SLE$_2$ is a.s. a simple curve. Suppose we condition on $(\cY^{\tau_k})^\d$ for $k\ge 1$.

 Let $D_k^\d$ be the connected component containing $\tilde X^\d_{\tau_k}$ of $\tilde N_{\tilde X_{\tau_k} \setminus \tilde v_0} \setminus (\tilde \partial \cup \tilde \cZ^k \cup \tilde A_k )^\d$. Recall that $D_k^\d$ is simply connected.
 %This is clear because if $D$ contains a non-contractible loop $\ell$ then there must be a component of $(\tilde \partial \cup \tilde \cZ^k)^\d$ strictly inside $\ell$. But this is impossible. Indeed, clearly this component cannot be $\tilde Z^k$ as $\tilde N_{\tilde X_{\tau_k} \setminus \tilde v_0}$ does not contain $v_0$. Hence this is a component of $\tilde \partial$. But the $p$-projection of every component of $(\tilde \partial) $ either contains a non-contractible cycle or contains a component of $\partial G^\d$. Suppose for the moment that $M$ is not simply connected and hence by assumption, every component of $\partial G^\d$ is non-contractible. This gives us a contradiction since the $p$-projection of every loop in $N_{\tilde X_{\tau_k} \setminus \tilde v_0}$ must be contractible in $M$. Finally if $M$ is simply connected, one can easily see that we can enforce none of $(N_{z \setminus \tilde v_0})_{z \in \tilde M}$ to contain all of $\partial G^\d$ for small enough $\delta$. Thus we conclude $D$ is simply connected.
Hence we can apply Uchiyama's result (\cref{lem:scaling_limit_main}) and conclude via induction.

Let us elaborate this application of Uchiyama's result.
Assume that $(\cY^{\tau_k})^\d$ is within distance $\ve$ of the law of $\cY^{\tau_k}$ in L\'evy--Prokhorov metric (with an underlying topology generated by the Hausdorff metric) for small enough $\delta$. We first need the following lemma.

\begin{lemma}\label{L:exit_convergence}
$\tilde X^\d_{\tau_{k+1}}$ converges in law to $\tilde V'_k$.
\end{lemma}

\begin{proof} From the invariance principle we know that the random walk $X^\d$ can be coupled with a reparametrised Brownian motion $B_{\phi^{-1}(t)}$ so that they are uniformly close on compact time intervals. Thus let $\tau$ be the time at which the Brownian motion hits $\tilde A_k \cup \tilde \partial \cup (N_{\tilde z_k \setminus \tilde z})^c$ (call this set $\cH_k$). Therefore, at time $\phi^{-1}(\tau)$, the random walk $\tilde X^\d$ is uniformly close to $\cH_k^\d : = \tilde A_k^\d \cup \tilde \partial^\d \cup ((N_{\tilde z_k \setminus \tilde z})^c)^\d$. Applying the Beurling-type estimate \cref{lem:Beurling} shows that with high probability the random walk will next intersect $\cH_k^\d$ after time $\tau$ at a position close to $B_{\phi^{-1}(\tau)} = \tilde V'_k$. Conversely, applying \cref{lem:Beurling} to Brownian motion (which follows e.g. by letting $\delta \to 0$ in this lemma) we see that when the random walk hits $\cH_k^\d$ at time $\tau_{k+1}$, the Brownian motion will also next hit $\cH_k$ after that time at a nearby position. Altogether this shows that $X^\d_{\tau_{k+1}}$ converges to $V'_k$.
\end{proof}

\begin{proof}[Proof of Theorem \ref{thm:main_tech}.]
From the invariance principle in our assumption, \cref{lem:Vk} and \cref{L:exit_convergence} above, it is not hard to see that $\tilde V_{k}^\d$ converges to $\tilde V_{k}$ in law since the point $\tilde V_{k}$ is an a.s. continuous function of the Brownian path $X[\tau_{k},\tau_{k+1}]$ and $\cY^{\tau_k}$.
Using \cref{lem:independent_RW} we deduce that conditioned on $X[0,\theta_{k}]^\d, (\cY^{\tau_k})^\d, X_{\tau_{k+1}}^\d$, %and $V_{k+1}^\d$
 the law of $\tilde X[\theta_{k},\tau_{k+1}]^\d$ is the same as a simple random walk starting from $\tilde V_{k}^\d$ conditioned to exit $D_k^\d$
%(or equivalently $\tilde N_{\tilde X_{\tau_k} \setminus \tilde v_0}^\d \setminus (\tilde \partial^\d \cup \tilde A_k^\d \cup (\tilde \cZ^k)^\d)$)
 at $\tilde X_{\tau_{k+1}}^\d$. Since $D_k^\d$ is simply connected, now apply Uchiyama's result (\cref{cor:Uchiyama}) to conclude that the law of the loop erasure of $X[\theta_{k},\tau_{k+1}]^\d$ conditioned on $X[0,\theta_{k}]^\d, (\cY^{\tau_k})^\d,X^\d_{\tau_{k+1}}$ converges as $\delta \to 0$ to an independent chordal SLE$_2$ in $D_k$ from $\tilde V_{k}$ to $\tilde V_{k}'$. Here $D_k$ is as in Step 3 of the continuum Wilson algorithm for generating CRSF: that is, let $\cY^{\tau_k}$ be the limiting simple curve of $(\cY^{\tau_k})^\d$ which is at most $\eps$ away in the Hausdorff sense from it (which exists by assumption). Then $D_k$ is the connected component containing $\tilde X^{\tau_k}$ in $\tilde N_{\tilde X^{\tau_k} \setminus v_0} \setminus(\tilde \partial  \cup \tilde \cZ_k \cup \tilde A_k)$.

 To see that we can apply \cref{cor:Uchiyama}, we need to verify that $D_k^\d$ converges to $D_k$ in the Carath\'eodory sense. To see that the loop-erasure converges, suppose without loss of generality that the convergence of $\tilde \partial^\d$ and $(\cY^\tau_k)^\d$ holds almost surely, by the induction hypothesis. Hence $\partial D_k^\d$ converges in the Hausdorff sense to $\partial D_k$, almost surely. Moreover, for any point $p_0$ in $D_k$ we have $p_0 \in D_k^\d$ for $\delta$ small enough and furthermore we can find an open neighbourhood of $p_0$ which is contained in $D_k^\d$  for small enough $\delta$. In other words, $D_k^\d$ converges in the sense of kernel convergence (\cite[Section 1.4]{Pommerenke}). Consequently, applying the Carath\'eodory kernel theorem (Theorem 1.8 in \cite{Pommerenke}), we deduce that for some fixed $p_0 \in D_k$, the Riemann map $\phi_\d$ in the assumptions of Theorem \ref{lem:scaling_limit_main} converges uniformly to $\phi$ on compact subsets of $D_k$.  Also note that since $\tilde \cY^{\tau_k}$ is a simple curve by the induction hypothesis, $\C \setminus D_k^\d$ is locally connected, uniformly in $\delta$. Hence the application of Corollary \ref{cor:Uchiyama} is justified and the proof {of \Cref{thm:main_tech}} is complete.
 %comment: \partial D_k consists of a Jordan domain and a slit by a simple curve. For Jordan domains, local connectedness follows from uniform continuity of the Jordan curve. Moreover, if f_n converge to f uniformly, then they are uniformly continuous, uniformly in n.
 %
 %
% Thus we get that the joint law of $((\cY^1)^\d,\ldots, (\cY^{k+1})^\d)$ is within $2\ve$ of the joint law of $(\cY^1,\ldots, \cY^{k+1})$ in the Levy-Prokhorov metric for small enough $\delta$. This completes the proof.
\end{proof}

Using \cref{thm:main_tech} and \cref{lem:MC_branch}, we know that the portion of the discrete CRSF sampled in any finite number of Wilson algorithm steps (i.e., a finite number of macroscopic branches in the CRSF) converges in law to a subset of $M$ sampled using the continuum algorithm described above. {In fact, since in {the discrete version of} Wilson's algorithm, the order in which the branches are sampled does not alter their joint law, the same is true for the continuum Wilson's algorithm using \Cref{thm:main_tech}.} To complete the proof of \cref{thm:CRSF_universal}, we need a version of Schramm's finiteness lemma (originally proved in \cite{SLE}). This is achieved in \cref{lem:Schramm_finiteness}, which introduces ideas (especially the ``good algorithm") which will be important for the local coupling argument later. We start to discuss this below.

%\begin{proof}
%
% Follows from the convergence in Schramm topology of the CRSF and the
%superexponential tail of $N_c^\d$. \note{B: I guess how obvious it will be depend on our definition of Shramm's topology and the proof would be how to read the cycles from the information we keep in the topology.}
%\end{proof}

\subsection{Schramm's finiteness lemma}\label{finiteness}

We start with a lemma on hitting probabilities which was proved {in} the simply connected case in
\cite{BLR16} using only the uniform crossing assumption; this is in a similar spirit to Lemma 2.1 in Schramm \cite{SLE}. Essentially, this provides a non-quantitative Beurling estimate and the proof is exactly the same as in \cite{BLR16}.

\begin{lemma}[Lemma 4.15 in \cite{BLR16}]
  \label{lem:SRW_hit}
There exist constants $c_0,c_1,C$ depending only on the constants in the crossing assumptions of \cref{sec:setup}
such that the following holds. Let $K \subset K' \subset \tilde M$ be a connected set such that the Euclidean diameter of $K$ is at least $R$. {Assume that there exists a vertex }$v \in \tilde \Gamma^\d$ be such that $B_{\Euc}(v,R) \subset K'$ where $B_{\Euc}$ denotes the Euclidean ball. Let $\dist(v,K)$ be the Euclidean distance between $v$ and $K$. Then
for all $\delta \in (0,\delta_{K'} \wedge C\delta_0 \dist(v,K))$,
$$
\P(\text{simple random walk from $v$ exits $B(v,R)^\d$ before hitting $K^\d$})
\le c_0 \left(\frac{ \dist(v,K)}{R}\right)^{c_1}
$$
\end{lemma}

%\begin{proof}
%This is exactly the same proof in \cite{BLR16}, so we provide a rough sketch. Let $C_i$ be the circle of radius $2^{-i}$ around $v$. Using the condition on $\delta$, we see that we can apply our uniform crossing assumption for all rectangles in annuli bounded by circles $C_i$ for all $i \in (C \log(\dist(v,K)),C\log R)$. However, the uniform crossing condition entails that the random walk has a uniformly positive probability $\alpha$ of forming a loop in each such annulus before exiting it and thereby has a uniform positive probability at least $\alpha$ of hitting $K$ before leaving it. Thus our required probability is at most $(1-\alpha)^{c \log_2(\frac{ \dist(v,K)}{R})}$, as desired.
%\end{proof}

{We will now need a version of  Schramm's finiteness lemma in
our setting of Riemann surfaces. The algorithm is almost identical to the one in \cite{BLR16}, however later in the proof of \Cref{thm:CRSF_universal}, we need to combine this with an additional argument to rule out accumulation of cycles near the boundary.}
\begin{description}\item[The good algorithm.]
Suppose we have specified a (possibly empty) set $\partial$ of boundary vertices. (In applications later, $\partial$ will consist of the natural boundary of the manifold $\partial \Gamma^\d$, which may be empty, and possibly a finite set of branches already discovered in the CRSF, including some noncontractible cycles.) 
 Suppose $v \in \tilde M$ and choose $r$ small enough so that $p$ is injective in $B(v,r)$ and also $B(v,r)^\d \cap \partial  = \emptyset$.
Let $H$ be a subset of vertices of  $B(v,r/2)^\d$.
Using the generalised Wilson's
algorithm, we now prescribe a way to sample the portion of the CRSF formed by branches starting from vertices in
$p(H)$ with the specified set of boundary vertices $\partial $. Consider a sequence of scalings $\{\frac{r}{2} 6^{-j} \Z^2\}_{j
\ge 1}$. Such a scaling divides the plane into squares of sidelength $6^{-j} r/2$ which we will refer to as cells. Let $H(s)$ be the subgraph induced by all vertices within Euclidean distance $s$ (in $\tilde M$) from $H$.  Define $\cQ_j = \cQ_j(H)$ as follows. Pick one vertex from $ H(2^{-j} r)$ in each cell of $\frac{r}{2} 6^{-j} \Z^2$ so that it is farthest from $v$ (break ties arbitrarily). Now we define the \textbf{good algorithm} which proceeds as follows. Sample
branches of the CRSF from the vertices of $p(\cQ_j)$ (in some arbitrary order) using
Wilson's algorithm to obtain $\cT^\d_j$. Then increase the value of $j$ to $j+1$ and repeat
 until we exhaust
all vertices in $H$. We denote this good algorithm by $GA_{\tilde \Gamma^\d,\partial, H}(v,r)$.
{Let $\cT_H^\d$ the set of branches containing $H$. It is clear that the branches sampled with the good algorithm $GA_{\tilde \Gamma^\d,\partial, H}(v,r)$ contains $\cT_H^\d$.}
\end{description}
%Note that the same algorithm would work for sampling the portion of
%wired UST in a domain $D \subset \C$ for a graph $\Gamma^\d$ satisfying the
%assumptions above \note{not written}. The only difference is that here
%we do not get any non-contractible loop. This version was described in
%\cite{BLR16}.

\medskip The proof of the following lemma is exactly the same as in Lemma 4.18
in \cite{BLR16} hence we do not provide a proof here and simply refer to that paper.

\begin{lemma}[Schramm's finiteness lemma \cite{BLR16,SLE}]\label{lem:Schramm_finiteness}
 Fix $\ve>0$ and let $v,r,H$ be as above. Then there exists an integer $j_0 =
j_0(\ve)$ depending solely on $\eps$ and the crossing constants from \cref{sec:setup} such that for all $j \ge j_0$ and all $\delta \le \min\{\delta_{{\bar B(v,r)}} ,C6^{-j_0}\delta_0 r\}$, where $\delta_{{\bar B(v,r)}},\delta_0$ are as in assumption (\ref{crossingestimate}), the
following holds with probability at least $1-\ve$:
\begin{itemize}
 \item The random walks emanating from all vertices in $\cQ_j(H)$ for $j > j_0$ stay in
$\cup_{z\in H }B(z,r/4)$.
\item All the branches of $\cT^\d$ sampled from vertices in $\cQ_j
\cap B(v,r/2)$ for $j > j_0$ until they hit $\cT^\d_{j_0}
\cup \partial$ or complete a noncontractible loop have Euclidean
diameter at most $\ve r$. More precisely, the connected components of
$\cT_H^\d \setminus \cT^\d_{j_0}$ within $B(v, r/2)$ have Euclidean
diameter at most $\eps r$.
\end{itemize}
\end{lemma}

\begin{proof}[Proof of \cref{thm:CRSF_universal}]
 {We first explain the main idea.} Using \cref{thm:main_tech} and \cref{lem:MC_branch}, we know that the portion of the discrete wired oriented CRSF sampled in any finite number of Wilson algorithm steps (i.e., after a finite number of branches) converges in law to a subset of $M$ sampled using the continuum algorithm described in \cref{SS:wilsoncontinuous}. Fix $\ve >0$. We will prescribe a way to sample the finite number of branches $\mathsf B$ so that the diameter of each remaining branch is at most $\ve$ with probability at least $1-\ve$ for all small enough $\delta$. This will complete the proof as this implies that the law of the discrete CRSF is Cauchy in the L\'evy--Prokhorov metric associated with the Schramm topology. Also, we know from \cref{thm:main_tech} that the law of the finite number of branches sampled above is close to the continuum branches sampled using the continuum Wilson's algorithm which does not depend on the sequence $\Gamma^\d$ chosen as long as it satisfies the conditions in \cref{sec:setup}. Hence the limiting law also is independent of the choice of $\Gamma^\d$.

\medskip
{Let us now give more details.} We concentrate on the hyperbolic case; the parabolic case (i.e., a torus under our assumptions) is almost exactly the same and in fact a little simpler.
An essential difference from the simply connected case consists in ruling out an accumulation of noncontractible loops near the boundary. For points in the bulk, one can simply invoke \cref{lem:Schramm_finiteness}. Let $K\subseteq \bar M$ be a compact subset (which may include portions of the boundary of $M$). Given $c>0$, consider an open cover of $K$ with $\cup_{z \in K} B_{M}(z,r_z)$ such that $r_z <c$ for all $z$ and $p$ is injective on every component of $p^{-1}(B_{M}(z,r_z))$ (recall $B_M$ is the ball induced by the metric in $M$ and this can include a portion of the boundary of $M$). Call a finite subcover $C(K,c)$ for any such choice of $K,c $ and some fixed choice of $r_z$.

Let $\eps>0$.
%Fix $\eta<\ve/4$.
Consider a finite cover $C(\overline M,\eps)$, with balls $B_1, \ldots, B_\ell$ (so $\ell$ is the number of sets in the cover). Now by the ``boundary Beurling" estimate (showing the boundary is hit quickly if one starts close to it) \cref{lem:Beurling}, we can choose a small $\beta = \beta( \eps)<\eps$ and $\delta=\delta(\beta)$ small so that for every vertex $v \in \Gamma^\d$ with $d_{M}(v,\partial M) < \beta$, the diameter (in the metric of $\bar M$) of a branch sampled from $v$ is at most $\ve/2$ with probability at least $1-\ve/\ell$.  Let $C'$ be the
set of balls in $C(\bar M, \eps)$ that intersect $\partial M$; call $\ell'$ the number of such balls and say without loss generality they are $B_1, \ldots, B_{\ell'}$. In each such ball $B_i$ ($1\le i \le \ell'$), let $v_i$ be a vertex such that $\beta / 2 \le d_M(v_i, \partial M) \le \beta$.

Let  $\cE_1$ be the event that none of the branches sampled from $\cV' = \{v_1, \ldots, v_{\ell'}\}$ has diameter more than $\ve$. By a union bound, $\cE_1$ has probability at least $1-\eps$ since $\ell' \le \ell$. Since every noncontractible loop in the manifold has $d_M$-length lower bounded by $\lambda>0$ which depends only on $M$ (as it has no punctures and finitely many holes), we can assume without loss of generality that $\ve$ is small enough so that on $\cE_1$, all the branches sampled from $\cV'$ don't form a noncontractible loop and hit the unique boundary component within distance $\eps>0$.

\medskip Let $D_{\beta} =\{z:d_{M}(z,\partial M) \ge \beta/2  \}$. Let $C(\overline D_\beta,\beta/4) = \cup_{i=1}^k B_{M}(z_i,r_{z_i})$. Pick $\eta>0$ (depending only on $\beta$ and hence only on $\eps$) small enough such that $p^{-1}(B_{M}(z_i,r_{z_i}))$ has at least one component in $(1-\eta) \D$ for all $1\le i\le k$. Note that the number of such components contained in $(1-\eta) \D$ is finite.
For each $1\le i \le k$, we fix $K_i$ to be one such pre-image of $B_M(z_i, r_{z_i})$ contained in $(1-\eta) \D$ and let $w_i:= p^{-1}(z_i) \cap K_i$.

Let $K_i^\d$ be the set of vertices of $\tilde \Gamma^\d$ in $K_i$.
For each $K_i$, we can define $\cQ_{j,w_i}:= \cQ_j(K_i^\d)$ as in \cref{lem:Schramm_finiteness}. We now sample the branches from  $(\cQ_{j,w_i})_{j \ge 1}$ as prescribed by the good algorithm. Since the Euclidean metric is conformally equivalent to the lift of the metric $d_M$ to $\D$, we see that there is an $\tilde \ve$ (depending only on $\eta$ and $\eps$, and hence only on $\eps$) such that if the Euclidean diameter of a connected set $X$ intersecting $(1-\eta)\D $ is less than $\tilde \ve$ in $\D$ then the diameter of $p(X) $ is at most $\ve$. Using \cref{lem:Schramm_finiteness}, we pick a $j_0$ depending on $\tilde \eps$ and $\beta$ (and so only on $\eps$)
such that for all $\delta < \delta(\beta)$, the following event $\cE_2$ holds with probability at least $1-\eps$: the Euclidean diameter of all the branches starting from vertices from $(\cQ_{j,w_i})_{1 \le i \le k,j>j_0}$ sampled according to the good algorithm is at most $\tilde \ve$.

\medskip Now consider the finite set of branches $\mathsf B$ consisting of all the branches starting from $v_i$, $1\le i \le \ell'$ and all the branches in $\cQ_{j, w_i}$, $j \le j_0$ and $1\le i \le k$. (This is the set $\mathsf B$ discussed at the start of the proof).
To finish the proof it therefore remains to point out that, once the branches $\mathsf B$ have been sampled,
conditional on the event $ \cE_1\cap \cE_2$, by planarity all the other branches deterministically have diameter in $M$ smaller than $6\eps$. Indeed, all the other branches are trapped in a cell of diameter at most $6 \ve$ on $\cE_1\cap \cE_2$. For $\delta < \delta(\beta)$ (and so $\delta$ small enough depending only on $\eps$), $\P( \cE_1 \cap \cE_2) \ge 1 - 2\eps$. This completes the proof of convergence in \cref{thm:CRSF_universal}. The proof of superexponential tail of the number of noncontractible cycles is immediate from  \cref{lem:non-contractible_cont} below.
%\note{Actually doesn't follow from what is written, although it is ok from proof...!}
\end{proof}

\begin{corollary}\label{lem:non-contractible_cont}
Let $(C_1^\d,\ldots, C_{K^\d}^\d)$ be the set of noncontractible cycles of
a wired oriented CRSF. Then
$$
 (C_1^\d,\ldots, C_{K^\d}^\d )
\xrightarrow[\delta \to 0]{(d)} (C_1,C_2, \ldots, C_{K})
$$
where $C_1, \ldots, C_{K}$ are almost surely disjoint noncontractible simple loops in $M$.  
Furthermore, for all $\ve>0$, there exists $C(\ve)<\infty$ such that for all $\delta$ small enough (depending on $\ve$), for all $k \ge 1$,
%Here $N_c$ is a random variable which satisfies for all $\ve>0$
\begin{equation}
 \P(K^\d>k) \le C(\ve) \ve^k. \label{eq:non-contractible_tail}
\end{equation}
%\note{Needs to be rephrased since that is not what we care about. I would also like to stick to the notation $K$ and perhaps $K^\d$ to be consistent with theorem.}
%\note{I Have no idea why there is the assumption of no boundary! Not only do we need it in general, but also the proof seems to be exactly the same! If anybody remembers this, it would be helpful.}
%\note{Show that the dual is meaurable with respect to primal or at least converges}.
\end{corollary}

We remark that Lemma 10 of Kenyon and Kassel \cite{KK12} provides a proof of  \eqref{eq:non-contractible_tail}, but we still include a proof for completeness and since it is rather short in our setting given the developed technology.
\begin{proof}
First note that the convergence in law of the set of noncontractible loops is a direct consequence of the convergence of the whole CRSF (\cref{thm:CRSF_universal}).

For the tail bound, notice that in the proof of \cref{thm:CRSF_universal}, once we have sampled $\mathsf B$, none of the rest of the branches have diameter more than $\ve$ with probability at least $1-\ve$. Since a noncontractible loop must have a diameter which is uniformly lower bounded, this proves that none of the rest of the branches sampled is a noncontractible loop with probability at least $1-\ve$, and hence trivially, the same statement is true for the next branch sampled. Since this probability is non-decreasing in the number of branches sampled, the proof is completed by iterating this bound.
\end{proof}
%  for any next branch we sample by Wilson's algorithm, the probability that the branch has diameter bigger than $\ve$ has probability smaller than $\ve$.
%
%  Further this probability decreases with every branch we sample. The proof is now complete by iterating this bound. Since for all $\ve>0$, we have sampled all the non-contractible branches by the above algorithm with probability at least $1-\ve$ and any finite number of branches converge by \cref{thm:main_tech}, the proof of \cref{lem:non-contractible_cont} is also complete.

\begin{remark}
The proof of the convergence of the wired oriented CRSF in \cref{thm:CRSF_universal} goes through if the manifold has a finite number of punctures. Indeed in that case, we do not need to worry about branches within distance $\ve$ of the punctures once the other branches have been sampled as they all must have diameter less than $\ve$. However, \eqref{eq:non-contractible_tail} is not true as there can be a lot of small noncontractible loops around the punctures.
\end{remark}

\section{CRSF and loop measures on surfaces }\label{sec:loop_measure0}
We now concentrate on convergence of the skeleton under $\Pwils$ and $\Ptemp$. This will be done in two steps. In \cref{S:Lawler} we prove a scaling limit result for two walks conditioned on an event of a similar nature as the Temperleyan condition, but in a simply connected domain (which one may think of as being a small disc centered at the puncture).
Then in \cref{sec:localglobal}, we see how this can be used to study the whole skeleton in a Temperleyan forest, without neglecting global topological issues; we call this step ``local to global''. In both these sections, we extensively use the connection between loop measures and spanning forests. We need an analogous theory for cycle rooted forests, which we develop in this section. We refer to the book of Lawler and Limic \cite{Lawlerbook} for a pedagogical introduction to the topic.

\subsection{Loop measures on surfaces}\label{sec:loop_measure}

It is well-known that loop-erased random walk and uniform spanning tree are related to loop measure and loop soup. In this section we review this relation in the classical case, give the extension we will need for surfaces and conclude with the main tools we will use to control the loop measure in the absence of fine estimates on the Green function. {Some of the results in this section work for any graph, not just the ones satisfying the assumptions in \Cref{sec:setup}. In this section we assume the latter is the case by default unless mentioned otherwise. We use the notation $\Gamma  = \Gamma^\d$ for the sequence of graphs embedded in a surface which fall into our setup in \Cref{sec:setup}}.

Consider a simple graph $G$, possibly oriented weighted and infinite, and let $q$ be the transition probabilities of a discrete time random walk on $G$. We emphasize that $q$ is defined on oriented edges of $G$ and that it is normalised with for all $v$, $\sum_{v'} q(v \to v') =1$. We naturally call a path in $G$ a sequence $(v_0, \ldots, v_n)$ of vertices such that $q(v_i \to v_{i+1} ) > 0$ for all $i$ and we extend the definition of $q$ to paths by simply saying that
\[
q( v_0, \ldots, v_n) = \prod_{i=0}^{n-1} q( v_i \to v_{i+1} ).
\]
We call \emph{rooted loop} a finite path $(v_0, \ldots, v_n)$ with $v_n = v_0$. We call length of a path the number of moves it makes and we denote it by $|.|$, i.e, if $\gamma = (v_0, \ldots, v_n)$ then $|\gamma| = n$ (a single vertex has length 0).
The rooted loop measure of $G$ is the measure $\mass$ defined on the set of all rooted loops by
\[
\mass( \ell ) = \frac{q(\ell)}{|\ell|}.
\]
Note that since $\mass (\ell)$ is invariant by re-rooting $\ell$, we can also see $\mass$ as a measure on \emph{unrooted loop} (formally equivalence class of rooted loops) and we will mostly think of loops as unrooted in the future.

\begin{remark}\label{rmk:loop_measure_intuition}
At this stage, it is useful to discuss briefly the qualitative behaviour that we expect from the measure $\mass$ for the sequence $\Gamma^\d$ from our setup (\Cref{sec:setup}), at least in the case where $\delta$ is very small compared to the whole manifold.
\begin{itemize}
\item The first thing we expect is that $\mass$ should satisfy some kind of scale invariance and that loops of all scale should have $O(1)$ mass:  For example, loops that have a diameter between $ R/2$ and $2R$ and intersect a given line of length $R$ should have a $O(1)$ mass independently of $R$. We will prove this in \cref{T:loopsoupRW} and more precisely \cref{cor:long_loops}.

\item There are loops of all sizes that surround a given point, and in total these have a mass which may be very large and is even infinite if the surface has no boundary. These can be extremely complicated (analogous to a Brownian motion on the torus run for time close to $\infty$, and hence essentially space-filling).

\item On the other hand, if ${M}$ has a (macroscopic) boundary, then macroscopic loops have bounded mass and very long loops (say whose length is a multiple of the typical hitting time of the boundary) have exponentially small mass.

\end{itemize}
The first two points suggest that we should not be able to really control the mass of loops if we cannot give both an upper and lower bound on their size. %The last point shows that in \cref{density_non_contractible}, the condition of $\eta$-contractibility must be act as a strong limitation on the length of loops.
\end{remark}

\begin{remark}
Note that if $\tilde \ell$ is a simple unrooted loop, we simply have $\mass (\tilde \ell ) = q( \tilde \ell)$ so one might think that we can directly define $\mass$ on unrooted loops. There are however some minor combinatorial complications such as the following: if $\tilde \ell$ repeats the same trajectory several times (say twice) there will be a combinatorial factor to account for the fact that the loop is of length $2n$ but there are only $n$ possible roots.
%Similarly, we note that we can think of a rooted loop as being rooted at an edge $(v_0 \to v_1)$. If we wanted to see it as rooted on a vertex, there would also be combinatorial factors depending  on the number of visits to the root vertex.
\end{remark}

The main reason for us to introduce this loop-measure is that it allows us to write in a simple form the density of the uniform spanning tree measure.

\begin{prop}\label{density_exit}\label{lem:law_pair}%\note{add reference to Lawler's book}
Let $G$ be a finite, simple, connected graph and let $\partial$ be a set of vertices of $G$. Let $\tau_\partial$ denote the hitting time of $\partial$ and assume that $\tau_\partial$ is almost surely finite independently of the starting point of the random walk. For any $v$, and for any path $\eta$ from $v$ to $\partial$ we have
\begin{equation}\label{LE1}
\P\Big( LE (X[0, \tau_\partial]) = \eta \Big) = q(\eta) \exp \mass \Big( \{\ell : \ell \cap \eta \neq \emptyset, \ell \cap \partial = \emptyset \} \Big).
\end{equation}
More generally for any $v_1, \ldots, v_n$, consider $Y_1, \ldots, Y_n$ the set of branches generated by running Wilson's algorithm from $v_1, \ldots, v_n$ and stopping at time $\tau_\partial$.
For any paths $\eta_1, \ldots, , \eta_n$ with $\eta_i$ from $v_i$ to $\partial$ and such that the $\eta_i$ merge when meeting each other, we have
\begin{equation}\label{LEn}
\P\Big( (Y_1, \ldots, Y_n) = (\eta_1, \ldots ,\eta_n) \Big) = q( \cup \eta_i ) \exp \mass \Big(\{\ell : \ell \cap (\cup \eta_i) \neq \emptyset, \ell \cap \partial = \emptyset  \}  \Big),
\end{equation}
where by $q( \cup \eta_i )$ we mean the product of the transitions probabilities over all edges of $\cup \eta_i$ (if a given edge appears in more than one path, its transition is counted just once).
\end{prop}
\begin{proof} The identity \eqref{LE1} is  Proposition 9.5.1 in \cite{Lawlerbook} (note that the setting includes nonreversible chains as needed here). The identity \eqref{LEn} then follows by induction and by the definition of Wilson's algorithm.
\end{proof}

To simplify notations when using expressions similar to the one in \cref{density_exit}, we introduce the following.
%If $\cA$ is a set of loops, we call $$\mass(\cA ) = \sum_{\ell \in \cA} \mass(\ell)$$
%the intensity measure of the loop soup. It may be easier to think of $\cA$ as a property that a given loop might satisfy.
If $G$ is a graph and $U$ is a subset of its vertices, we write
\begin{equation}\label{eq:massU}
    \mass_U( \cA ) = \sum_{\ell \in \cA : \ell \subset U} \mass(\ell).
\end{equation}
When $G = \Gamma$ is embedded on a surface $M$, and if $D \subset {\cM}$ is a portion of that surface, we also write with a slight abuse of notation
$\mass_D( \cA ) = \sum_{\ell \in \cA : \ell \subset D} \mass(\ell)$.
We also write for a set $S \subset {\cM}$ the loops that intersect $S$ as
\[
\cI( S ) = \{ \ell : \ell \cap S \neq \emptyset \}
\]
where with a small abuse of notation we identify here $\ell$ with its range $\ell [ 0, |\ell|]$. 
\begin{remark}
	If $\eta_1$ and $\eta_2$ are disjoint paths, we can compare the joint law with the independent case in a particularly simple way
	\begin{equation}\label{LEnbis}
	\P((Y_1, Y_2) = (\eta_1, \eta_2) ) = e^{ - \mass_{D \setminus \partial} ( \cI(\eta_1) \cap \cI(\eta_2) )} \prod_i \P( Y_i = \eta_i).
	\end{equation}
	This observation will be used extensively in the rest of the article.
\end{remark}

\subsection{Law of CRSF branch}\label{sec:law_branch}

A crucial ingredient for this paper will be to have an analogue of \eqref{LEnbis} for branches of a {Cycle-Rooted Spanning Forest} on our surface (or more generally a manifold). We first recall that this is the law on spanning subgraphs with exactly one outgoing edge out of every vertex, and no contractible loops: recall a loop on $\Gamma^\d$ is called \textbf{contractible} if when lifted to the universal cover it is still a loop, and note that contractibility is invariant by re-rooting.
Recall from \cref{sec:Wilson} that the CRSF law can be sampled by the following adaptation of Wilson's algorithm, first described (implicitly) by Kassel and Kenyon \cite{KK12}. We start with an arbitrary vertex $v$, and run a random walk; let $Y^t = LE (X[0,t])$. We stop the random walk at the time $\tau_{NC}$  given by
\begin{equation}\label{tauNC}
  \tau_{NC} = \inf\{ t \ge 0: Y^{t-1}\oplus \{X_t\} \text{  contains a noncontractible loop} \} \wedge \tau_\partial
\end{equation}
where we identify the concatenated path $Y^{t-1}\oplus \{X_t\}$ with the sequence of edges traversed by it, and $\tau_\partial$ is the first time that the walk hits boundary $\partial$ (when $\partial=\emptyset$, $\tau_\partial = \infty$). In other words, if we erase (simple) loops as we make them, $\tau_{NC}$ is the first time where such a simple noncontractible loop is created, or we hit the boundary. (It is easy to see using \Cref{lem:Beurling,lem:uniform_avoidance} that $\tau_{NC}$ is finite almost surely). Then the branch of the CRSF $\eta_v$ containing $v$ is defined to be $Y^{\tau_{NC} -1} \oplus X_{\tau_{NC}}$. We then iterate this algorithm, adding $\eta_v$ to the boundary, until every vertex has been covered. To describe the law of a branch of the CRSF, a key complication of this paper is that we will need a more subtle notion of contractibility of a loop, which is as follows.

%We will need a similar statement for the law of branches in a CRSF, meaning LERW when we stop whenever a non-contractible loop is formed instead of when a particular set is hit. We come back to considering only graphs $\Gamma^\d$ embedded in the surface ${\cM}$ and we let $\tau_{NC}$ we the first time where a random walk close a non-contractible cycle as in
%\note{I think we need to present Wilson's algorithm for surfaces somewhere} or the walk hits a boundary.

\begin{defn}
Suppose $\ell$ is a loop rooted at $v$.  We decompose $ \ell$ into simple loops by considering the chronological loop-erasure of $\ell$ starting at $v$. We say that $\ell$ is \textbf{Wilson-contractible} if all of these simple loops are contractible.
\end{defn}
Observe that if a rooted loop is Wilson contractible, and it visits its root $v$ several times, then re-rooting the loop at any other copy of $v$ keeps it Wilson contractible.
\begin{defn}\label{defn:eta_contractible}

Given a path $\eta$, we say that a loop $\ell$ is {\bf $\eta$-contractible} if $\ell$ intersects $\eta$ and the following holds. Let $v$ denote the point where $\eta$ hits $\ell$ for the first time, i.e $v = \eta_{\tau_\ell}$. Then $\ell$ rooted at (any copy of) $v$ is Wilson contractible. %\note{Notation for compatibility ?}

\begin{figure}[t]
\begin{center}
\includegraphics[scale=.3]{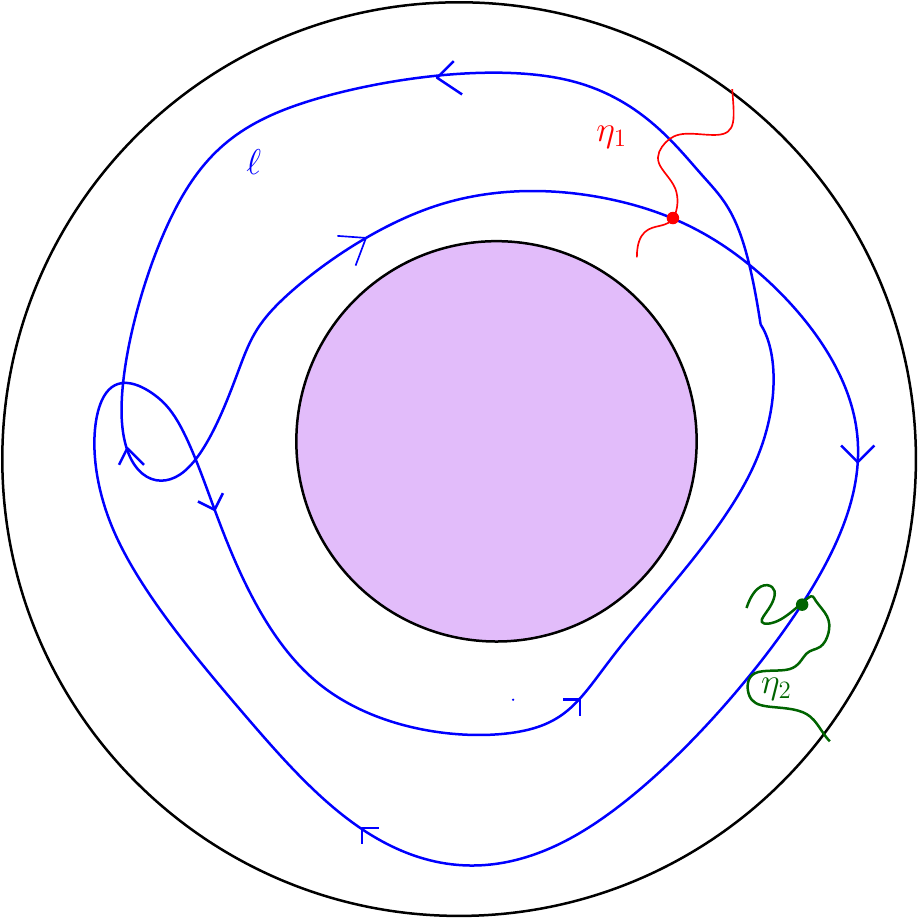}
\caption{Illustration of the difference between contractibility of a loop and the notion of $\eta$-contractible. The blue loop $\ell$ is contractible in itself. However, it is $\eta_1$-contractible and not $\eta_2$-contractible. The first intersections of $\eta_1$ and $\eta_2$ with $\ell$, which are relevant for this notion, are marked on the picture (the paths $\eta_1$ and $\eta_2$ terminate on the outer boundary of the annulus).}
\label{F:contractible}
\end{center}
\end{figure}

More generally, if $(\eta_1, \ldots, \eta_n)$ are ordered paths on the surface, then we say a loop $\ell$ is $(\eta_1, \ldots, \eta_n)$\textbf{-contractible} if the same condition holds, where we choose $v$ to be the first chronological intersection of $(\eta_1, \ldots, \eta_n)$ with $\ell$ in this order. We write $\cC_{\bs \eta}$ for the set of $\bs \eta$-contractible loops.
\end{defn}

See Figure \ref{F:contractible} for an illustration of what can go wrong. Note however that if the support of a loop is contained in a simply connected domain, then this loop is $\eta$-contractible for any $\eta$. We can now state the equivalent of \cref{density_exit} for the CRSF:

\begin{thm}\label{density_non_contractible}
For all $v$, let $Y$ denote the law of a LERW started at $v$ and stopped at $\tau_{NC}$. For any simple path $\eta$ either connecting $v$ to a boundary of $\Gamma^\d$ or finishing in a non-contractible loop, we have
\begin{equation}\label{LE_surf_1}
\P( Y = \eta ) = q(\eta) e^{\mass \{ \ell \in \cI(\eta), \eta\text{-contractible} \} }.
\end{equation}
Likewise, if $\eta_1, \ldots, \eta_n$ are candidate paths for the first $n$ steps in Wilson's algorithm,
\begin{equation}\label{LE_surf_n}
  \P( Y_1 = \eta_1, \ldots, Y_n = \eta_n ) = q( \cup_i \eta_i)e^{\mass ( \{ \ell \in \cI(\cup_i \eta_i) : (\eta_1 , \ldots, \eta_n)\text{-contractible}\})}.
\end{equation}
\end{thm}
\begin{proof}
      The proof follows Lemma 9.3.2 in \cite{Lawlerbook} closely but with the difference that we have to keep track of non-contractible loops. Variations of this argument are also used in what follows; hence we provide a sketch of proof for completeness. Let $A $ be a subset of vertices not containing any boundary vertex and pick $x\in A$. Let $\cL_x(A)$ denote the set of rooted loops $\ell$ supported in $A$, intersecting $x$  such that if we reroot $\ell$ at (any copy of) $x$, it becomes Wilson contractible. Given a loop $\ell$, let $\ell_0$ denote its root. Let $d(\ell)$ denote the number of times $x$ is visited in a loop.

      Define 
    \begin{equation*}
        g(A,x)  = \sum_{\ell \in \cL_x(A), \ell_0 =x} q(\ell);\qquad f(A,x) = \sum_{\ell \in \cL_x(A), \ell_0 =x, d(\ell)=1} q(\ell)
    \end{equation*}
    It is straightforward to check that $g(A,x)$ is the expected number of visits to $x$ by the random walk before exiting $A$ or creating a non-contractible loop, and $f(A,x)$ is the probability of returning to $x$ before exiting $A$ or creating a non-contractible loop. It is now standard to check that $g(A,x) =1+f(A,x)g(A,x)$.

     Now note the following combinatorial fact:
    \begin{equation*}
    \sum_{\ell \in \cL_x(A)} \frac{q(\ell)}{|\ell|}  = \sum_{\ell \in \cL_x(A), \ell_0 = x} \frac{q(\ell)}{d(\ell)}
        \end{equation*}
        and
        \begin{equation*}
    \sum_{\ell \in \cL_x(A), \ell_0 = x, d(\ell) = j} q(\ell) = \left(\sum_{\ell \in \cL_x(A), \ell_0 = x, d(\ell) = 1} q(\ell) \right)^j.
            \end{equation*}
    Therefore 
    \begin{equation*}
        \sum_{\ell \in \cL_x(A), \ell_0 = x} \frac{q(\ell)}{d(\ell)} = \sum_{j=1}^\infty\frac1j\Big(\sum_{\ell \in \cL_x(A), \ell_0 = x, d(\ell) =1} q(\ell)\Big)^j = -\log(1-f(A,x)) = \log (g(A,x)).
    \end{equation*}
        
   Therefore for any $\eta$ represented as an ordered sequence of vertices $ \eta= (x_0,x_1,\ldots, x_n)$, and letting $A_j = V \setminus (\partial \cup (x_0,x_1,\ldots, x_{j-1}))$ for $j \ge 1$ and $A_0 = V \setminus \partial$,
   \begin{equation*}
       \exp\left({\mass \{ \ell \in \cI( \eta), \eta\text{-contractible} \} }\right) = \prod_{i=0}^n \exp(\mass( \cL_{x_j}(A_j))) = \prod_{i=0}^n g(A_j, x_j). 
   \end{equation*}
   Indeed the first equality follows from the observation that the set of all $\bs \eta$ contractible loops can be partitioned into those which are Wilson contractible rooted at $x_0$, and then those which are Wilson contractible when rooted at $x_1$ and does not intersect $x_0$, and so on.
   Finally observe that (by definition of $g$),
   \begin{equation*}
          \P( Y = \eta ) = q(\eta) \prod_{i=0}^n g(A_j, x_j),
   \end{equation*}
   which completes the proof.
  \end{proof}

\begin{remark}
  Note that although the order of the vertices is irrelevant in Wilson's algorithm, it is \emph{a priori} not obvious that the expression in the identity \eqref{LE_surf_n} gives the same answer when we switch the order of the paths. {However in light of Wilson's algorithm, the loop measure term is invariant under reordering of the paths. Also for two different $\bs \eta$'s the loops which are $\bs \eta$-contractible for one but not for the other must have a non-contractible support, and we will see later in \cref{L:fewcontractible} that this set of loops has bounded mass.}
\end{remark}

  In Wilson's algorithm for a Uniform Spanning Tree, a remarkable and extremely useful feature is that the loops generated along the way form a realisation of the Poissonian loop soup with intensity $\mass$ (and are furthermore independent of the UST). 
This relationship is in turn a convenient way of proving estimates about the loop measure $\mass$, as we will see (Theorem \ref{T:loopsoupRW}). For what follows it is useful to explain this in a bit more detail. Given a random walk trajectory $(X_t)_{0 \le t \le \tau_D}$ considered up until its hitting time of the boundary, there is a well defined operation (involving extra randomness), which associates to $X$ a collection $\ell (X)$ of unrooted loops. This collection is obtained from merging the (necessarily simple) loops erased from $X$ in Wilson's algorithm, and resplitting them appropriately with the help of a suitable Poisson--Dirichlet distribution; see for instance \cite[Proposition 5.4 and Remark 15]{Chang2016} for details. We skip the details of the description as we will not be needing the exact description of the surgery behind this merging and splitting procedures. Rather, the fact that there is a Poissonian description of these loops is all we will use later.

We will state an analogue of this fact for CRSFs. However, in this case, because of the $\eta$-contractibility condition in Theorem \ref{density_non_contractible}, the relationship between loops coming from Wilson's algorithm and the Poissonian loop soup is more subtle. We can summarise this relationship as follows.
 
  \begin{thm}
    \label{T:loopsNC}
     Let $\cT$ be the CRSF obtained by performing Wilson's algorithm (under no conditioning) on $\Gamma^\d$ satisfying the assumptions of \Cref{sec:setup}. We view $\cT$ as an ordered sequence of paths $\cT = (\eta_1, \ldots, \eta_r)$. Then given $\cT$, the set of loops $\cup_i \ell (X^i)$ erased by Wilson's algorithm applied to the successive random walk paths $X^1, X^2, \ldots,$ is a Poisson point process with intensity $\mass|_{\{\ell: \ell \text{ is $\cT$-contractible}\}}$, viewed as a collection of loops in the space of unrooted loops\footnote{This Poisson point process is usually called the \emph{loop soup} in the case of the Uniform Spanning Tree.} 
    
    The same statement holds for the measure $\Lambda$ if we replace $\Gamma^\d$ by an  infinite, transient, connected graph and $\cT$ is a wired uniform spanning forest on it. 
  \end{thm}

%\begin{remark}\label{rmk:loopsNC}
%The implicit description we refer to in \cref{T:loopsNC} requires some care to make precise. Of course, the loops erased by Wilson's algorithm are necessarily simple, whereas $\Lambda$ is not restricted to simple loops. Nevertheless, there is a natural way to run the random walk, concatenate and re-split the erased simple loops (with some extra independent randomness) and then `forget' the root. Altogether this produces a sample from a Poisson point process of intensity $\mass|_{\{\ell: \ell \text{ is $\cT$-contractible}\}}$; see \cite[Proposition 5.4 and Remark 15]{Chang2016} for a precise description. As this precise description is not needed in the following we have chosen not to make the statement above completely precise.
%The result for non-simply connected case follows from arguing in exactly the same way (an example of such an analogous argument is provided in the proof sketch of \cref{density_non_contractible}).
%\end{remark}

As mentioned above, this result is known in the setting of finite Uniform Spanning Trees (see \cite[Proposition 5.9]{lawler2018topics} for a slightly weaker version and \cite[Proposition 5.4 and Remark 15]{Chang2016} for the full result). The result extends readily to the setting of CRSFs; see \cref{density_non_contractible} for a more elaborate example of how to implement such an extension.

The wired uniform spanning forest on an infinite transient graph can also be sampled using Wilson's algorithm `wired at infinity': simply run Wilson's algorithm where random walks are run forever or until they hit the boundary, started from an arbitrary ordering of vertices (see e.g. \cite{LyonsPeres}). Thus, the theorem also extends to the setting of wired uniform spanning forest on infinite transient graphs in a straightforward manner.

{Note that \cref{T:loopsNC} shows a big difference between the loops generated by Wilson's algorithm and the loop soup: indeed, there will be only a finite number of loops of diameter greater than some fixed macroscopic value $\epsilon>0$, even on a surface with no boundary such as the torus. On the other hand, the loop soup on such a surface contains an infinite number of loops of duration many times bigger than the mixing time or even the cover time of the torus.}

%\note{Add remark that the difference between notions of $\eta$-contractibility for various $\eta$ can only produce a bounded change in the mass of loops considered since this requires the support to be non-contractible (non simply connected).}

\subsection{Marginal on partial paths}\label{sec:law_marginal}

In this section, the goal is to study the law of a small connected portion of an UST or a CRSF branch containing its starting point in terms of loop measures. We are also interested in expressing the joint law of the beginning portions of two branches as well, again both in the context of CRSF or UST. The relevance of these  will appear later, in particular in \cref{S:local_global} when we will transfer estimates from the simply connected to the Riemannian setting. We will, in the following, mostly use $\eta, \eta_-, \eta_+$ for partial paths, while we will use $\gamma$ and $Y$ for completed paths.

\subsubsection{Marginals on UST paths}
We start with well-known results for a single path in the classical setting of spanning trees.

\begin{prop}\label{prop:marginal_single}
	Let $G$ be a graph (finite or infinite) %\note{maybe say `finite graph' to avoid complicacies?} I am pretty sure eveything is covered by the fact that the hitting time is a.s finite.
	and let $\partial$ be a set of vertices such that the hitting time of $\partial$ (call it $\tau_\partial$) is almost surely finite independently of the starting point of the random walk on $G$. Fix $x_0, x_-,x_+ \in G$ and let $Y$ denote a LERW started from $x_0$ and stopped when hitting $\partial$. Fix also $\eta_-$ a self avoiding path starting in $x_0$ and ending in $x_-$. Fix also another self-avoiding path $\eta_+$ starting from $x_+$ and ending in $\partial$. Then we have
\begin{align}
\P( \eta_- \subset Y) & = q(\eta_-)e^{\mass_{G \setminus \partial}(\cI(\eta_-))} \P_{x_-}[ \tau_\partial < \tau_{\eta_-} ] \nonumber\\
& = q(\eta_-)e^{\mass_{G \setminus \partial}(\cI(\eta_-))} \sum_{\gamma \supset \eta_-} q(\gamma \setminus \eta_-) e^{\mass_{G \setminus (\partial \cup \eta_-)} ( \cI( \gamma))},\label{partial_single_SC}
\end{align}
%\note{Do you mean $\mass_{G \setminus \partial}(\cI(\eta_-))$ in the second equation?} Fixed
where we use $\tau_{\eta_-}$ for the first hitting time of $\eta_-$ or similar notations for other sets and $\P_{x_0}$, $\P_{x_-}$ are the laws of random walk started in $x_0$ and $x_-$ respectively. The sum is over all self avoiding paths $\gamma \supset \eta_-$ starting at $x_0$ and ending at $\partial$.
Likewise,
\begin{align}
\P( \eta_+ \subset Y) &= q(\eta_+) e^{\mass_{G \setminus \partial}(\cI(\eta_+))} \P_{x_0}[X(\tau_{\partial \cup \eta_+} ) = x_{+}]  \label{eq:marginal_single1} \nonumber\\
\P( \eta_- \subset Y, \eta_+ \subset Y) & =q(\eta_-) q(\eta_+) e^{\mass_{G \setminus \partial}( \cI( \eta_- \cup \eta_+))} \P_{x_-} [X(\tau_{\eta_- \cup \eta_{+} \cup \partial}) = x_{+} ]
\end{align}
\end{prop}
\begin{proof}
The first two identities easily follow from \cite[Proposition 9.5.1]{Lawlerbook}. For \eqref{eq:marginal_single1} we need the fact that the law of the path is unchanged when loops are erased in reverse chronological order \cite[Lemma 7.2.1]{lawler2013intersections} and then applying \cite[Proposition 9.5.1]{Lawlerbook}.
\end{proof}

%\note{write reference and modify the beginning of the quasi independence of paths parts.}
We now state a few propositions  regarding the marginals of paths and pairs of paths in a uniform spanning tree.   The proofs of \cref{cor:conditional_single,prop:extension,prop:marginal_pair,C:conditionalpair} again are simple applications of \cite[Proposition 9.5.1]{Lawlerbook}, we leave the details to the reader.

\begin{prop}\label{cor:conditional_single}
	In the same setting as \cref{prop:marginal_single}, for any pair of self avoiding paths $\eta_-, \gamma$ such that $\eta_- \subset \gamma$ and $\gamma$ connects $x_0$ to $\partial$,
	\begin{equation}\label{partial_single_cond_SC}
\P( Y = \gamma | \eta_- \subset Y) = \frac{1}{\P_{x_-}[ \tau_\partial < \tau_{\eta_-} ]} q(\gamma \setminus \eta_- ) e^{\mass_{G \setminus (\eta_- \cup \partial)} ( \cI(\gamma))},
\end{equation}
and given $\eta_- \subset Y$, $Y$ has the law of the loop-erasure of an excursion from $x_-$ to $\partial$. Here by excursion we mean a random walk conditioned to hit $\partial $ before returning to $\eta_-$.
\end{prop}

Although not actually needed in the following, %\note{Is this phrase awkward Englishwise..?} Better now ? 
it is useful to mention the marginal law of a pair of (complete) branches in a UST on a graph $G$ with boundary $\partial$. We believe it should be beneficial for the reader when they read the analogous result for the CRSF case.

\begin{prop}\label{prop:extension}
	Let $G, \partial$ be as in $\cref{prop:marginal_single}$ and fix $x^1, x^2$ in $G$. Let $Y_1, Y_2$ denote the two UST branches started from $x^1$ and $x^2$ respectively with the convention that each connects $x^i$ to $\partial$. Let $\eta_1$ and $\eta_2$ be \textbf{disjoint} self avoiding paths starting in $x^1$ and $x^2$ respectively. Then
	\begin{align}
	\P( \eta_1 \subset Y_1, \eta_2 \subset Y_2) &
	= q(\eta_1)q(\eta_2) e^{\mass_{G \setminus \partial} (\cI(\eta_1 \cup \eta_2))} \sum_{\gamma_1\supset \eta_1, \gamma_2 \supset \eta_2} q((\gamma_1 \cup \gamma_2) \setminus( \eta_1 \cup \eta_2)) e^{\mass_{G \setminus (\eta_1 \cup \eta_2 \cup \partial)} (\cI( \gamma_1 \cup \gamma_2))} \\
	&  = q(\eta_1)q(\eta_2) e^{\mass_{G \setminus \partial} (\cI(\eta_1 \cup \eta_2))} \P(\text{loop erasures compatible with $(G, \partial)$}), \label{eq:unconditioned_domain}
	\end{align}
	where in the last line, the probability term is interpreted as follows: we run two random walks from the endpoints of $\eta_1$ and $\eta_2$ while performing Wilson's algorithm with boundary $\partial \cup \eta_1 \cup \eta_2$ and we consider the event that the resulting two paths, when combined with $\eta_1$ and $\eta_2$, have positive probability under Wilson's algorithm in $G, \partial$.
\end{prop}
\begin{figure}
\centering
\includegraphics[scale = 0.5]{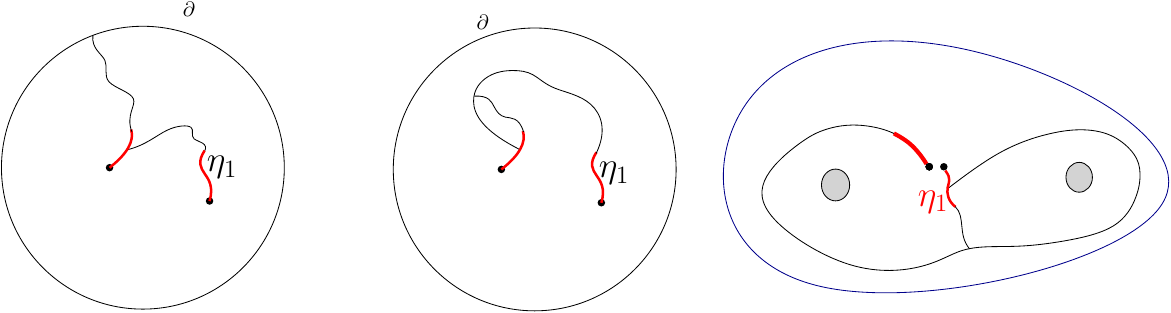}
\caption{The red curves are $\eta_1,\eta_2$ in \cref{prop:extension} (left, middle) and \cref{prop:marginalM} (right) for $k=1$. Left: A compatible loop erasure. Middle: An incompatible loop erasure. Right: An incompatible loop erasure in a pair of pants.}\label{fig:compatibility}
\end{figure}

Note that in the first line above, implicit in the notation (and consistent with our convention), $\gamma_1$ and $\gamma_2$ are complete paths which end in $\partial$; these are of course not necessarily disjoint. See \cref{fig:compatibility}.

%\begin{remark}
 %  As shown in \note{the picture}, there are two ways to obtain the event in the right hand side of \cref{eq:unconditioned_domain} : the first random walk can hit $\partial$ before $\eta_1 \cup \eta_2$ and then the second walk just has to avoid $\eta_2$ until reaching $\eta_1$, $\partial$ or the first loop-erasure ; or the first walk can hit $\eta_2$ but then the second must avoid not only $\eta_2$ but also $\eta_1$ and the first loop-erasure and reach $\partial$. This second case  case can be though of as reflecting the fact that if we had sampled $(Y_1, Y_2)$ with Wilson's algorithm, it is possible that a part of $\eta_2$ appeared already in the branch $Y_1$.
%\end{remark}

\medskip Remarkably, the situation is not much more complicated when we look at the law of two branches conditioned not to merge. In particular it becomes practical to compare the joint law of two branches with the independent case. We state this in the following propositions (\cref{prop:marginal_pair,C:conditionalpair}).
\begin{prop}\label{prop:marginal_pair}
	Let $G, \partial$ be as in \cref{prop:marginal_single} and let $Y_1, Y_2, \eta_1, \eta_2$ be respectively the two UST branches and two disjoint self-avoiding paths starting from $x^1, x^2$ as above. We have
	\begin{align}
	& \P( \eta_1 \subset Y_1, \eta_2 \subset Y_2 | Y_1 \cap Y_2 = \emptyset) \\
&		= \frac{q(\eta_1)q(\eta_2)}{\P( Y_1 \cap Y_2 = \emptyset)} e^{\mass_{G \setminus \partial} (\cI(\eta_1 \cup \eta_2))} \sum_{\gamma_1 \cap \gamma_2 = \emptyset} q(\gamma_1 \setminus \eta_1)  q(\gamma_2 \setminus \eta_2)  e^{\mass_{G \setminus (\eta_1 \cup \eta_2 \cup \partial)} \cI( \gamma_1 \cup \gamma_2)} \\
	&  = \frac{q(\eta_1)q(\eta_2)}{\P( Y_1 \cap Y_2 = \emptyset)} e^{\mass_{G \setminus \partial} (\cI(\eta_1 \cup \eta_2))} \P\Big(\tau^1_{\partial} < \tau^1_{\eta_1 \cup \eta_2},\tau^2_{\partial} < (\tau^2_{\eta_1 \cup \eta_2} \wedge \tau^2_{LE(X^1)}) \Big),\label{eq:RN_start2}\\
	& = \frac{e^{- \mass_{G \setminus \partial}(\cI(\eta_1) \cap \cI(\eta_2))}}{\P( Y_1 \cap Y_2 = \emptyset)} \P(\eta_1 \subset Y_1 ) \P( \eta_2 \subset Y_2) \P\Big(\tau^1_{\partial} < \tau^1_{\eta_2},\tau^2_{\partial} < (\tau^2_{\eta_1} \wedge \tau^2_{LE(X^1)}) \mid \tau^1_\partial < \tau^1_{\eta_1}, \tau^2_\partial < \tau^2_{\eta_2}  \Big) \label{eq:RNstart}
	\end{align}
	%\note{I am stuck at the last equality? How do you get the conditioning?}
	where the event in the second line is that while running sequentially two independent random walks starting from the endpoints of the $\eta_i$, both walks must reach $\partial$ before hitting $\eta_1 \cup \eta_2$ and additionally the second walk cannot hit the loop-erasure of the first one. In the sum of the first line we implicitly restrict to $\gamma_i \supset \eta_i$.
\end{prop}
The equation \eqref{eq:RNstart} allows us to compare the marginal law of the beginning of the pair of paths $\bs Y = (Y_1, Y_2)$ under the conditioning $Y_1 \cap Y_2 = \emptyset$ to the independent situation.
(Note that the conditional probability in the right hand side of \eqref{eq:RNstart} can also be naturally expressed in terms of a pair of excursions from the endpoints of the $\eta_i$ to $\partial$ conditioned to avoid $\eta_i$.) 
This comparison with independent paths can be extended to the conditional law of the remaining portion of $\bs Y$ (still under the conditioning $Y_1 \cap Y_2 = \emptyset$) given an early portion of these paths. This is the content of the next proposition and follows easily from the expressions in \cref{prop:marginal_pair}.
\begin{prop}\label{C:conditionalpair}
	Let $G, \partial, Y_i, \eta_i, x_i$ be as in \cref{prop:marginal_pair} for $i = 1, 2$ and let $\bs \gamma := (\gamma_1, \gamma_2)$ be disjoint paths connecting $x_i$ to $\partial$ with $\gamma_i \supset \eta_i$, we have
	%\note{Isn't $Y_1 \cap Y_2 = \emptyset$ redundant below?} No, it matters in the global scaling of the probability
	\begin{equation}\label{E:lerwpartialDomain}
	\P( \bs Y = \bs \gamma | Y_1 \cap Y_2 = \emptyset, \bs \eta \subset \bs Y) = \frac{q(\gamma_1 \setminus \eta_1)q(\gamma_2 \setminus \eta_2)}{\P(\tau^1_{\partial} < \tau^1_{\eta_1 \cup \eta_2},\tau^2_{\partial} < \tau^2_{\eta_1 \cup \eta_2} \wedge \tau^2_{LE(X^1)})} e^{\mass_{G \setminus(\eta_1 \cup \eta_2 \cup \partial)} \cI( \gamma_1 \cup \gamma_2)}.
	\end{equation}
	where $\bs Y = (Y_1,Y_2)$. Furthermore if $X_1$ and $X_2$ denote independent random walks started from the endpoints of $\eta_1$ and $\eta_2$, we have
	\begin{align}
	&\P( \bs  Y= \bs \gamma | Y_1 \cap Y_2 = \emptyset, \bs \eta \subset \bs Y) \nonumber  \\
& = \frac{e^{-\mass_{G \setminus (\eta_1 \cup \eta_2 \cup \partial)} (\cI(\gamma_1) \cap \cI(\gamma_2))}}{\P(\tau^2_\partial < \tau^2_{LE(X^1)} \mid \tau^1_{\partial} < \tau^1_{\eta_1 \cup \eta_2},\tau^2_{\partial} < \tau^2_{\eta_1 \cup \eta_2}) } \prod_{i=1,2} \P( LE( X_i) = \gamma_i \setminus \eta_i \mid \tau^i_{\partial} < \tau^i_{\eta_1 \cup \eta_2}).\label{eq:RNend}
	\end{align}
\end{prop}

\begin{remark}
	Even though the notations become a bit heavy, note that \cref{eq:RNstart,eq:RNend} are very similar to \cref{LEnbis}: the joint law of two branches is given by an independent term weighted negatively by the loops intersecting both paths, with just an additional $h$-transform term in \cref{eq:RNstart} and a change of the ``reference'' independent law in \cref{eq:RNend}: namely, these are now excursions (instead of random walks) from the endpoints of $\eta_i$, conditioned to leave $G \setminus (\eta_1 \cup \eta_2 \cup \partial)$ through $\partial$.
\end{remark}

\subsubsection{Marginals on CRSF paths}
We now prove the analogues of \cref{prop:marginal_single,cor:conditional_single,prop:marginal_pair,C:conditionalpair} but for the case of branches of a CRSF on a graph $\Gamma = \Gamma^\d$ 
 embedded in a surface $\cM$, satisfying the conditions in \Cref{sec:setup} (we also write $(\cM, \partial)$ if we want to emphasize the role of the boundary, but otherwise do not specify it). We call $\cA_{\cM}$ the event that a CRSF is Temperleyan and recall that by assumption (\ref{punctures}), $\P(\cA_M) >0$ for every $\delta$, although it converges to 0 as $\delta \to 0$. As a convention, when tracing branches of a CRSF, we trace each path until it hits the boundary $\partial$ or closes a non-contractible loop.

The case of a single path is very similar to the simply connected case (compare with \eqref{partial_single_SC} and \eqref{partial_single_cond_SC}). Let $\cC_{\eta}$ denote the set of all $\eta$-contractible loops with respect to the partial path $\eta$.

\begin{prop}\label{prop:marginal_single_M} Let $Y$ be a CRSF branch from some point $z \in G \subset \cM$, and let $\eta$ denote a partial path starting from $z$ and let $x$ be its endpoint.
  \begin{align}
\P( \eta \subset Y) & = q(\eta)e^{\mass_{\cM }(\cI(\eta) \cap \cC_{\eta})} \P_{x}[ \tau_\cM \le \tau_{\eta} ] .
%& = q(\eta_-)e^{\mass(\cI(\eta_-))} \sum_{\gamma \supset \eta_-} q(\gamma \setminus \eta) e^{\mass_{G \setminus (\partial \cup \eta_-)} ( \cI( \gamma))},
\end{align}
where $\tau_{\cM}$ is the stopping time when either $\partial $ is hit or a non-contractible loop is created by the simple random walk.
Furthermore,
	\begin{equation}\label{partial_single_cond_M}
\P( Y = \gamma | \eta \subset Y) = \frac{1}{\P_{x}[ \tau_\cM \le \tau_{\eta} ]} q(\gamma \setminus \eta ) e^{\mass_{\cM \setminus \eta } ( \cI(\gamma) \cap \cC_{\gamma})}.
\end{equation}
\end{prop}

\begin{proof}
The proof is very similar to \cref{prop:marginal_single,cor:conditional_single}, the only difference being that we need to restrict to $\eta$-contractible paths in light of \cref{density_non_contractible}. It is immediate from the definition of $\eta$-contractibility that the loops erased must be $ \eta$-contractible. The remaining steps of the proof are identical to the proof of  \cref{prop:marginal_single,cor:conditional_single}, which in turn follow by running the argument in the proof of \cite[Proposition 9.5.1]{Lawlerbook}.

The key input for \eqref{partial_single_cond_M} is that with a single path there is compatibility between the notion of being $\eta$-contractible and $\gamma$-contractible for a given loop, as $\eta$ is the initial portion of $\gamma$: the first intersection point of the loop with $\gamma$ is in $\eta$ if it intersects $\eta$.
\end{proof}

 This is no longer true with more than one path, and the generalisations of \cref{prop:marginal_pair} and \cref{C:conditionalpair} are therefore more cumbersome (see \cref{fig:eta_contractible} and the accompanying explanations). Nevertheless, we are able to obtain up-to-constant bounds which are enough for our purpose.

%\begin{figure}[h]
%\centering
%\includegraphics[scale = 0.15]{loop_nc_term}
%\caption{The green loop is $\gamma_2, \gamma_3$ contractible, but not $\gamma_1$ contractible. Hence the loops in \eqref{eq:start_surface} may intersect the beginning portion of a previously sampled branch, which did not appear in the loop measure term before as it was  not contractible. \note{may replace by a better picture or delete it, but I think it is useful to have an example in mind. Side note: can you find an example with just one puncture?}}
%\end{figure}

\begin{prop}\label{prop:marginalM}
	Let $\bs{Y} = (Y_1, \ldots, Y_{2k})$ be the branches of the skeleton of a CRSF of $\Gamma$ (assume $k \ge1$) and let $\bs{\eta} = (\eta_1, \ldots, \eta_{2k})$ be simple \textbf{disjoint} paths starting from the marked points $(u_i, v_i)$. We have
	\begin{align}
	\P( \bs{\eta} \subset \bs{Y} | \bs{Y} \in \cA_\cM ) & = \frac{q(\cup \eta_i )}{\P(\bs{Y} \in \cA_\cM )} e^{\mass( \cI(\cup \eta_i) \, \cap \cC_{\bs{\eta}})} \sum_{ \bs{\gamma}\in \cA_\cM} e^{\mass( \cI(\cup \gamma_i)\cap \cC_{\bs{\gamma}}) - \mass( \cI(\cup \eta_i) \cap \cC_{\bs \eta})} q(\cup \gamma_i \setminus \cup \eta_i ) \label{eq:start_surface}\\
	& \asymp  \frac{q(\cup \eta_i )}{\P(\bs{Y} \in \cA_\cM )} e^{\mass( \cI(\cup \eta_i) \cap \cC_{\bs \eta})} \sum_{\bs \gamma \in \cA_\cM} e^{\mass_{\cM \setminus \cup \eta_i}( \cI( \cup \gamma_i) \cap \cC_{\bs \gamma})} q(\cup \gamma_i \setminus \cup \eta_i)\label{eq:start_surface4}\\
	& \asymp \frac{q(\cup \eta_i )}{\P(\bs{Y} \in \cA_\cM )} e^{\mass( \cI(\cup \eta_i) \cap \cC_{\bs \eta})} \P(\cA_\cM, \text{loop erasures compatible with $\bs \eta$})\label{eq:start_surface3}
	\end{align}
	Where $A \asymp B$ means $\frac{1}{C}A \leq B \leq C A$ for some $C > 0$.  %with constants $c, C >0$ depending only on the mass of contractible loops intersecting $\bs{\eta}$ and with a non-contractible support.
	As above, in the sums in the first two lines, the paths $\gamma_i$ are not necessarily disjoint. In the last line, the probability means that we run Wilson's algorithm starting from the endpoints of $(\eta_1, \eta_2, \ldots, \eta_{2k})$ (in this order), stopping when we either hit $\partial$ or any of the $\eta_i$ or create a non-contractible loop. The event is that the union of the $\eta_i$ and the resulting loop-erased paths forms a CRSF of $(\cM, \partial)$ and is in $\cA_\cM$. %\note{Isn't it better to write `loop erasures compatible with $\cA_\cM$ and $\bs{\eta}$'? Also it was called $\tau_M < \tau_{\bs \eta_r}$ later. Maybe it is a good idea to name this event, maybe $\cA_{\cM}(\bs \eta)$.}
	
	Furthermore, the constants $C$ can be taken such that $\log( C)$ is the mass of contractible loops that intersect $\bs \eta$ and have a non-contractible support.
\end{prop}

\begin{remark}\label{R:constants}
  Combined with the scaling behaviour estimates coming from Schramm's lemma in the next section (see in particular \cref{L:fewcontractible}), the constant $C$ really depends just on the surface $\cM$ and on the diameter of the paths $\eta_i$, and converges to 1 as this diameter converges to 0.
\end{remark}

Before proving the proposition, it is useful to discuss why we do not get equalities instead of the claimed approximations, and why we cannot be more explicit in the probability in the last line. For the latter, note that, as pictured in \cref{fig:eta_contractible}, the paths $\gamma_i$ do not need to be disjoint, even in the simplest situation of a Temperleyan CRSF where we have only one puncture and two skeleton branches in a pair of pants. For the former, it comes from the fact that in \cref{eq:start_surface}, the notion of $\eta$-contractibility and $\gamma$-contractibility do not match exactly, see again \cref{fig:eta_contractible}.
\begin{figure}[h]
\centering
\includegraphics[width = \textwidth]{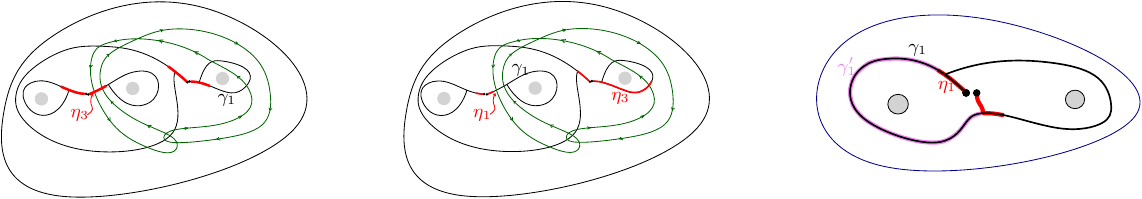}
\caption{The red curves correspond to $\bs \eta$ and the black curves correspond to $\bs \gamma$. The green loop is $\bs \gamma$-contractible but not $\bs \eta$ contractible on the left and $\bs \eta$-contractible but not $\bs \gamma$-contractible in the middle. Right: The purple path is $\gamma_1'$ and the black path is $\gamma_1$. Here $\gamma_1 \setminus \eta_1$ is not the same as $\gamma_1'$. Note that in this case $\gamma_1 \cap \gamma_2 \neq \emptyset$, it is in fact the long non-contractible cycle going around both holes.}\label{fig:eta_contractible}
\end{figure}

%As an example, assume that a loop $\ell$ does not intersect $\eta_1$, intersects $\eta_2$ at two points $a$, $b$ with $a$ before $b$ along $\eta_2$, and that $\ell$ is contractible when started at $a$ but not in $b$. Then $\ell$ is $\bs{\eta}$-contractible but for the paths $\gamma_1$ containing $b$ but not $a$, $\ell$ is no longer $\bs\gamma$-contractible.

\begin{proof}
First observe that
\eqref{eq:start_surface} follows simply from \cref{density_non_contractible} by decomposing the loop measure terms appropriately. The difference between the expressions \eqref{eq:start_surface4} and \eqref{eq:start_surface}, come from loops which  intersects $\cup_i \eta_i$ and are either $\bs \gamma$-contractible but not $\bs \eta$-contractible or vice versa (the latter come with a negative sign, see \cref{fig:eta_contractible}). Nevertheless, the total mass of such loops is bounded by the mass of loops with non-contractible support. Hence their total mass is always upper bounded by a uniform constant $C$ depending only on $(\cM, \partial)$.

Now we turn to \eqref{eq:start_surface3}. Define $\gamma'_i \subset \gamma_i$ as follows. Let $\gamma_1'$ be a path starting from the tip of $\eta_1$ until it hits $\cup_i \eta_i$ or creates a non-contractible loop, $\gamma'_2$ is the path started from the tip of $\eta_2$ until it hits $\gamma'_1 \bigcup (\cup_i \eta_i)$ and so on. Clearly,
$$
\P(\cA_\cM, \text{loop erasures compatible with $\bs \eta$}) =  \sum_{\bs \gamma' } e^{\mass_{\cM \setminus \cup \eta_i}( \cI( \cup \gamma'_i) \cap \cC_{\bs \gamma'})} q(\cup \gamma'_i )
$$
where the sum is over all possible $\bs \gamma'$ which are in $\cM \setminus \cup \eta_i$, $\cup_i (\gamma'_i \cup\eta_i)$ forms a compatible CRSF configuration in $(M, \partial)$ and also $\cA_\cM$ is satisfied. Furthermore, the reason why this quantity is not exactly equal to \eqref{eq:start_surface4} is that $\gamma'_i$ may not be equal to $\gamma_i \setminus \eta_i$ for each $i$, even if $\cup \gamma_i  = \cup_i (\eta_i \cup \gamma'_i)$ (see rightmost figure in \cref{fig:eta_contractible}). Thus the loop measure term may include some loops in $\cM \setminus \cup_i\eta_i$ which are not $\bs \gamma$-contractible  but  $\bs \gamma'$ or vice versa (see rightmost figure in \cref{fig:eta_contractible}). However in any case, these loops must have non-contractible support, therefore must be upper bounded by the same $C$ (we do not attempt to solve the topological problem of identifying cases where exact equality holds as we do not need it).
\end{proof}

\subsection{Scaling behaviour of loop measures and Schramm's finiteness lemma}\label{sec:finiteness}

We now turn to one of the main tools that allows us to control the mass of a set of loops. On $\Z^2$ control of the loop soup measure on large scales comes from explicit and precise knowledge of Green functions, potential kernels and heat kernel estimates which are clearly not available in our setup. We will use as an effective substitute instead the following theorem which compares the loop soup measure on a given scale to random walks on that same scale. Together with the uniform crossing estimate of random walk this will often give us the required approximate ``scale invariance'' of the loop soup measure. The key idea is to use Wilson's algorithm to generate the loops together with Schramm's finiteness lemma.

For a simple random walk $X[0,t]$ run up to time $t>0$, $\ell(X_t)$ denotes the collection of loops erased when performing the loop erasure procedure, viewed according to the loop soup description of \cref{T:loopsNC}. For the following theorem, we need to be in the simply connected setup of \cite{BLR16}, which involves a sequence of graphs $\Gamma^\d$ which is faithfully embedded in $\R^2$ and satisfies the simply connected version of the uniform crossing assumption (this can be found in \cite[Section 4.1]{BLR16}). We restate it here for convenience:

\paragraph{Uniform crossing assumption of \cite{BLR16}:} Let
    $\cR$
   be the horizontal rectangle $[0,3]\times [0,1]$ and $\cR'$ be the vertical
rectangle with same dimensions, and let $B_1 := B((1/2,1/2),1/4)$ be the
\emph{starting ball} and $B_2:= B((5/2,1/2),1/4) $ be the \emph{target ball}. There
exist constants $\delta_0 >0$ and $\alpha_0>0$ such that for all $z
\in \C, \delta>0$, $\ell \ge 1/\delta_0$, $v \in \ell \delta B_1$ such that $v+z \in \Gd$,
\begin{equation}
  \P_{v+z}(X \text{ hits }(\ell \delta  B_2+z) \text{ before exiting }  (\ell \delta \cR+z))
>\alpha_0.
\end{equation}

\begin{thm}\label{T:loopsoupRW}
% Let $0<\epsilon<1 $ and $R>0 $ be arbitrary, and let $\cA$ be a set of loops such that all elements $\ell$ of $\cA$ stay inside $D= B( x ,R)$ and have diameter at least $\epsilon R$. The following two properties hold for all $\delta$ sufficiently small:
 Fix $0<\eps<1/100 $, $R>0 $, $x \in \R^2$ and let $D \subset \R^2$ be a possibly unbounded simply connected domain containing $B(x,3R)$. Suppose that we are given a graph sequence $\Gamma^\d$ on $\R^2$ satisfying the uniform crossing assumption of the simply connected setup as described above. Let $\cA$ be a set of loops supported in $D$. Let $\cA(\eps, R)$ denote the set of loops in $\cA$ of diameter at least $\eps R$ and which intersect $B(x, R)$. Let $\tau_D$ be the first time that a random walk $X$ on $\Gamma^\d$ leaves $D$ (assume that $\tau_D < \infty$ a.s.). The following two properties hold for all $\delta$ sufficiently small:

 (1) %There exists $\eta = \eta( \eps) >0 $ such that
 For all $p > 0$, there exists $\eta = \eta (\eps, p)$ such that
 \begin{equation} \label{E:orderoneloops}
% \mass( \cA ) \leq 1/\eta,
\Big( \exists y \in B(x, (1+\eps) R), \P_y ( \ell(X_{\tau_D}) \cap \cA(\eps, R) = \emptyset ) \geq p \Big) \Rightarrow \mass_D( \cA(\eps, R) ) \leq 1/\eta,
 \end{equation}

   (2) For all $ 0 < \eta < 1$, there exists $p = p (\eps, \eta) >0$ such that
 \begin{equation}\label{E:few loops}
 \Big( \forall y \in B(x, (1+\eps) R), \P_y ( \ell(X_{\tau_D}) \cap \cA(\eps, R) \neq \emptyset ) \leq p \Big) \Rightarrow \mass_D( \cA(\eps, R) ) \leq \eta,
 \end{equation}

 \end{thm}
 
%\note{Are these if and only if?}
%The estimate in \eqref{E:orderoneloops} says that there are at most order one loops on any given scale. On the other hand, the estimate in \eqref{E:few loops} is more precise, and says that if a property is unlikely for a random walk on a given scale, then there will in fact be \emph{very few} loops for which this property holds.
The theorem allows us to relate the total mass of loops in a loop soup to the behaviour of a single random walk. The estimate in \eqref{E:orderoneloops} says that if it is possible for the loops erased from a random walk not to satisfy some property, then there will be at most order one  loops satisfying that property in a loop soup. On the other hand the estimate in\eqref{E:few loops} says that if a property is unlikely for the loops erased from a random walk on a given scale, then there will in fact be \emph{very few} loops for which this property holds. To prove these, we will extensively use the Poisson point process description of erased loops, see \cref{T:loopsNC}.
\begin{remark}
  The condition on the diameter of loops being at least $\eps R$ is in fact necessary for the statement to hold. The problem is the following: in a loop soup there are many microscopic loops which may include some whose behaviour is markedly different from a typical random walk realisation.

  For instance, one may consider on a box of $\Z^2$ of sidelength $n$ all loops $\cA$ which do not move away from their starting point, and stay there for a duration of order $\beta$, then these have mass
  $$
  \mass (\cA) \asymp n^2  e^{- \beta}/ \beta.
  $$
  On the other hand, for the loops $\ell(X)$ created by a random walk stopped at time $\tau$ when leaving the box, we can use the connection to loop-erased random walk $Y$ and the fact that this its length is $|Y| = n^{5/4+ o(1)}$ when $n \to \infty$: thus, conditioning on $Y$ and letting $\mathbf{L}$ denote a Poisson point process of loops with intensity $\Lambda$,
  \begin{align*}
  \P ( \ell (X)  \cap \cA = \emptyset) &= \P( Y  \cap \{ \ell  \in \mathbf{L}: \ell \in \cA\}  = \emptyset) \\
  & \approx \E [(1- e^{- \beta} / \beta)^{|Y|}] \\
  & \gtrsim (1- e^{-\beta}/\beta)^{n^{5/4 + o(1)}},
  \end{align*}
  where we derived the bound in the last line by considering a typical event where $|Y|$ has size at most $n^{5/4 + o(1)}$,
  and so
  $$
  \P ( \ell (X)  \cap \cA = \emptyset) \gtrsim \exp( - n^{5/4 + o(1)} e^{- \beta} / \beta ).
  $$
  We can choose $\beta$ such that $n^2 e^{- \beta} / \beta \to \infty$ but
  $ n^{5/4+ o(1)} e^{- \beta} / \beta \to 0$; in this case $\cA$ is unlikely for the loops of a random walk but very likely for some loop in the loop soup $\mathbf{L}$. This would contradict both \eqref{E:few loops} and \eqref{E:orderoneloops}.
\end{remark}

\begin{proof}[Proof of Theorem \ref{T:loopsoupRW}]
As mentioned before the statement, the proof relies essentially on Wilson's algorithm and Schramm's finiteness lemma (\cref{lem:Schramm_finiteness}). In our context, this theorem yields the following. Let $\eps, \pi>0$ be small constants and let $R>0$ be positive. Consider the wired uniform spanning tree in $\Gamma^\d$ restricted to $D$ (i.e., we treat every edge leading out of $D$ as an edge leading to the wired boundary). In case that $D$ is unbounded, one should consider the wired uniform spanning forest rather than wired uniform spanning tree. Then there exists $m = m ( \eps, \pi)$ (in particular, $m$ does not depend on $R$ or $D$) and vertices $y_1, \ldots, y_m \in B(x, (1+\epsilon) R)^\d$ such that the following holds if $\delta$ is small enough.
Start by sampling the branches emanating from $y_1, \ldots, y_m$ in the tree, and then sample all remaining branches from $B(x, R)^\d$ before any other point, then with probability at least $1- \pi$, every single subsequent branch from $B(x, R)^\d$ will be obtained by erasing the loops of random walks which will never leave a ball of radius $\eps R$ around their starting points. This is true even when we condition on the $m$ branches from $y_1, \ldots, y_m$.

Let us start with a proof of \eqref{E:orderoneloops}. We use the fact that Wilson's algorithm actually generates not only a wired uniform spanning tree in $D^\d$ but also an independent Poisson point process on $\mathbf{L}$ with intensity $\mass_D$ (\cref{T:loopsNC}).
Let $\eps>0$ be fixed and choose $m = m(\eps,1/2) $ as in Schramm's finiteness lemma associated with error probability $\pi = 1/2$ and space scale $\eps$. Let $y_1, \ldots, y_m$ be the associated starting points {and let $y$ be a point such that $\P_{y} ( \ell(X_{\tau_D}) \cap \cA(\eps, R) = \emptyset ) \geq p$ as in the assumption of \eqref{E:orderoneloops}. We first sample the branch $\gamma_y$ of the wired UST in $D^\d$ from $y$, and assume that no loop in $\cA(\eps, R)$ was created in this process.}  

We now want to fix $m$ continuous continuous smooth, simple paths $\gamma_1, \ldots, \gamma_m$ from $y_1, \ldots, y_m$ in $B(x, 3R)$ such that if the random walk follow these paths staying at distance at most $\eps R/10$ from them, we can guarantee the following properties: on the one hand, the walks will have either hit
$\gamma_y$ (hence Wilson's algorithm will have stopped for these paths), and on the other hand, no loop of size greater than $\eps R$ will have been created in the process. Note that since $y_1, \ldots, y_m, y \in B(x, R(1+ \eps))$, the choice of such paths can be made independently of $D$: essentially we simply choose $\gamma_i$ to be a straight line segment between $y_i$ and $y$ (run until distance $ R/10$ from $y$) followed by a circular loop of radius $ R/10$ around $y$. Note that by planarity, \emph{any} path that stays at distance at most $\eps R/10$ from $\gamma_i$ will hit $\gamma_y$ before creating a loop of diameter at least $\eps R$. 

%For example, we can consider a straight line to the disc of radius $2R$ centered at $x_i$ and also consider a straight line which gets close to $y$, followed by a simple loop around $y$. The first choice takes care of the case when the boundary of the domain is closer to $x_i$ compared to $y$. The second choice takes care of the case when the boundary of $D$ is far away. 

Once we have such paths, by uniform crossing, there is a positive probability $p' = p'(\eps)$ (in particular, independent of $D$ and $R$) such that $m$ independent random walks from $y_1, \ldots, y_m$ stay in a neighbourhood of size $\eps R/10$ around $\gamma_1, \ldots, \gamma_m$ respectively. On that event of probability at least $pp'$, no loop of diameter greater than $\eps R$ is created in the first $m+1$ steps of Wilson's algorithm. Furthermore, by Schramm's finiteness lemma, conditional on this event, no subsequent loop has diameter greater than $\eps R$ with probability 1/2. Thus overall
$$
\P ( \mathbf{L} \cap \cA (\eps, R)= \emptyset ) \ge pp'/2.
$$
But the left hand side is $\exp( - \mass_D(\cA(\eps, R)))$, so this implies
$$
\mass_D(\cA(\eps, R)) \le - \log (pp' / 2)  = : 1/\eta.
$$

Now let us give a proof of \eqref{E:few loops}. Let $\eta, \eps>0$ be fixed as in the statement of the theorem. We assume without loss of generality that $\eta <\eta_0$, $\eta_0$ to be chosen below. Let $m$ be as in Schramm's finiteness lemma associated with error probability $\eta/2$ and space scale $\eps/2$, so $m = m (\eps/2, \eta/2)$. Let $y_1, \ldots, y_m$ be the associated points. Let us choose the value of $p$ in the theorem as $p = \eta/(2m)$. For this value of $p$, suppose we have as in the assumption:
$$
\forall y, \quad \P_y ( \ell(X_{\tau_D}) \cap \cA (\eps, R)\neq \emptyset ) \leq p.
$$
We can rewrite
$$
\P_y ( \ell(X_{\tau_D}) \cap \cA (\eps, R)\neq \emptyset )  = \P ( \{ \ell \in \mathbf{L}: Y_y \cap \ell \neq \emptyset\} \cap \cA (\eps, R) \neq \emptyset)
$$
with $Y_y$ a loop-erased random walk run until the hitting time of $D$, by pretending that $y$ is the first point in Wilson's algorithm. (In other words, $Y_y$ is the branch emanating from $y$ in the uniform spanning tree.)
We apply this with $y = y_1 , \ldots, y_m$ and since $p = \eta/ (2m)$, taking a union bound, we obtain
$$
\P ( \{ \ell \in \mathbf{L} : \ell \cap (Y_{y_1} \cup \ldots \cup Y_{y_m}) \neq \emptyset \}\cap \cA  \neq \emptyset) \le \tfrac{\eta}{2}.
$$
On the other hand, by construction the remaining walks which generate the tree \emph{all} have diameter at most $\eps R$ with probability at least $1- \eta/2$. On this event, none of the corresponding loop can be in $\cA(\eps , R)$ since the loops of $\cA (\eps, R)$ must have diameter at least $\eps R$. Hence
$$
\P ( \mathbf{L} \cap \cA (\eps, R)\neq \emptyset ) \le \eta.
$$
But the left hand side is simply $1- \exp ( - \mass_D(\cA(\eps, R)))$. So
$$
\mass_D(\cA(\eps, R)) \le - \log (1- \eta) \le 2\eta
$$
for $\eta < \eta_0$ sufficiently small. \eqref{E:few loops} follows.
\end{proof}

To illustrate a quick application: we obtain the following ``scale invariance'' for loops:
\begin{corollary}\label{cor:long_loops}
Assume we are in the setup of \cref{T:loopsoupRW}.
For all $\epsilon > 0$, there exits a $C > 0$ such that for all $R> 0$ and $x \in M$ and for all $\delta$ small enough
%\note{we can be more explicit on the maximum $\delta$ allowed}
\[
\mass_{B( x, R)}( \{ \ell : \diam (\ell) \geq \epsilon R \}) \leq C.
\]
\end{corollary}
\begin{proof}
By uniform crossing, a random walk started from $x$ has a positive probability to move almost as a straight line with no loops of size greater than $\epsilon R$ (as follows by considering a straight line segment from $x$ to the boundary of $B(x, R)$ and the uniform crossing assumption). This proves the assumption from \eqref{E:orderoneloops}.
\end{proof}

When we don't have an a priori bound on the size of loops, we said in \cref{rmk:loop_measure_intuition} that we expect the mass to be very large. We can however recover a finite mass if we also constrain the geometry to loops which do not surround a fixed point, which will be of special interest for us.
\begin{prop}\label{lem:loop_crossing}
Assume we are in the setup of \cref{sec:setup} and let  $N_{\max}$ be the maximum $m$ such that $B(x,2^{m})$ is simply connected.
There exists $C, c > 0$ such that for any $n, k, N_{\max }, x$ and for any $\delta$ small enough, the mass of loops in $B( x, 2^{N_{\max}})$ which intersect both $B(x, 2^n)$ and $B( x, 2^{n+k})^c$ but do not disconnect $x$ from $B( x, 2^{N_{\max }})^c$ is at most $Ce^{-ck}$.
\end{prop}
\begin{proof}
We fix $n, k, N = N_{\max}$ and we let $E$ denote the set of loops in the proposition. As in the theorem above, the strategy is to control the probability that there is as least a loop in $E$ when a loop soup is generated by Wilson's algorithm. Without loss of generality we assume $x=0$.

We start Wilson's algorithm from $0$, let $X$ denote the random walk started from $x$ and let $Y$ denote its loop-erasure. As in \cite{BLR16}, on $X$ we introduce stopping times $\tau_i$ defined inductively as follow: $\tau_1$ is the first exit time of $B(0, 2^{N_{min}})$ where $N_{min}$ is chosen such that uniform crossing applies outside $B(0, 2^{N_{min}})$. Given $\tau_i$ (which will have by construction norm approximately $2^r$ for some $r$), $\tau_{i+1}$ is the first exit time after $\tau_{i}$ of $A( 0, 2^{r-1}, 2^{r+1} )$.

Given a path $Y$, we decompose the set of loops in $E$ according to the location of their first intersection point in $Y$ as follows. For a loop $\ell$, let $t_{\ell} = \inf \{ t : Y_t \in \ell \}$ and let $E_r = \{  \ell \in E : Y_{t_\ell} \in A( 0, 2^r, 2^{r+1} ) \}$ (where loops which do not intersect $Y$ are not in any $Y_{t_\ell}$).

For a given $r$, note that in order for the walk $X$ to create a loop in $E_r$, there must be times $t < t'$ such that $X_{t} \in A( 0, 2^r, 2^{r+1})$, $X_{t'} \in A( 0, 2^r, 2^{r+1})$, {$X[t, \infty]$ does not separate $0$ from $B(0, 2^{r+1})^c$}, $X[t, t']$ intersects both $B(0, 2^n)$ and $B( 0, 2^{n+k})^c$, $X[t, t']$ does not disconnect $0$ from $B(0, 2^N)^c$. Indeed the last conditions are obviously necessary to have a loop in $E$ and the first ones are necessary to have $Y_{t_\ell} \in A( 0, 2^r, 2^{r+1})$.

It is easy to see that one can further restrict ourself to times of the form $\tau_i$ such that $||X_{\tau_i}|| = 2^{r}$ or $||X_{\tau_i}|| = 2^{r+1}$, and furthermore, the number of such times has an exponential tail independent of everything (these are the times when the annulus is visited after the last full turn by the random walk around $0$ in $A( 0, 2^r, 2^{r+1})$, see Lemma 4.15 in \cite{BLR16}). Therefore the expected number of pairs $\tau_{i}, \tau_{i'}$ such that $X[\tau_i, \tau_i']$ intersects both $B(0, 2^n)$ and $B( 0, 2^{n+k})^c$ but $X[\tau_i, \infty]$ does not separate $0$ from $B( 0, 2^{r+1})^c$ is finite. If we further take into account the condition that $X[ \tau_i, \tau_{i'}]$ cannot disconnect $0$ from $B(0, 2^{N_{\max}})$, we see that the expected number of loops in $E_r$ is at most $Ce^{-c(n+k -r)}$ (resp. $Ce^{-ck}$, $Ce^{r-n}$) if $r < n$ (resp. $n \leq r \leq n+k$, $r > n+k$) by uniform crossing. Summing over $r $, the expected number of loops in $E$ created by the walk $X$ is at most $Ce^{-ck}$ for some (different) $C, c$ (since it is a Poisson with a parameter bounded by $Ce^{-ck}$).

Changing the point of view on the above bound, it means that if we sample a loop soup $\mathbf{L}$ and an independent LERW $Y$, we have
\[
\E(|\{ \ell \in \mathbf{L} : \ell \cap Y \neq \emptyset, \ell \in E \}| )\leq C e^{-ck}.
\]
On the other hand it was proved in \cite[Corollary 26]{Laslier_robust} that one can couple two LERWs $Y_1$ and $Y_2$ such that they agree up to their last exit of $B(0, 2^n)$ and also agree after their first exit of $B(0, 2^{n+k} )$ and on an event of probability $p$,
\[
W(Y_2, 0) = W( Y_1, 0) + 2\pi
\]
where $p$ is a positive constant independent of $n$ and $k$. A few comments are needed for this application of Corollary 26 in \cite{Laslier_robust}. Essentially this argument boils down to showing that conditional on the behaviour of $Y$ outside the annulus $A(0, 2^n, 2^{n+k})$ centered around zero and with radii $2^n$ and $2^{n+k}$, there is a positive, uniform in $n$, conditional probability for $Y$ to accumulate a winding of $\pm 2\pi$, but there is also a positive probability to not accumulate any winding. This positive conditional probability is shown in two stages, the first is to show that with positive probability the behaviour in scales $[2^{n-1}, 2^{n+1}]$ is in some sense ``nice'' (this is the meaning of the notion of isolated scale in \cite[Corollary 26]{Laslier_robust}), and conditional on that there is a positive chance of winding.

It is easy to check that for topological reasons, when this event happens any loop in $E$ has to be intersected by at least $Y_1$ or $Y_2$. In particular each loop in $E \cap \mathbf{L}$ has a probability at least $p/2$ to be intersected by $Y$ and $\E |E \cap \mathbf{L}| \leq 2C e^{-ck}/p$. This concludes the proof.
\end{proof}

We conclude this section by an estimate on the mass of contractible loops, which says that the mass of contractible loops on the surface which intersects a small ball, but have macroscopic diameter is bounded. This is similar in spirit to the above lemmas but more subtle due to Theorem \ref{T:loopsNC}.

\begin{lemma}\label{L:fewcontractible}
Assume $\Gamma = \Gamma^\d$ be a sequence of graphs embedded in a hyperbolic surface $M$ satisfying the assumptions in  \cref{sec:setup}.
For all $\ve > 0$ there exists $r > 0$ such that for all $\delta$ sufficiently small, for all $x \in {\cM}$,
\[
\mass (\{\ell : \ell \text{ is contractible }, \diam (\ell ) \geq \eps; \ell \cap B_{\cM}(x, r) \neq \emptyset  \}  ) \leq \eps.
\]
The same holds when replacing contractible by ``$\eta$-contractible for any $\eta$''.
\end{lemma}
\begin{proof}First note that if a loop is not contractible, it also cannot be $\eta$-contractible for any $\eta$ so we only need to prove the first point. 

We consider separately the case where $\partial M \neq \emptyset$ and the case where $\partial M = \emptyset$. Suppose first that $\partial M \neq \emptyset$. In this case consider an annulus $A$ centered at $x$ and of radii $ \eps/4, \eps/2$ on $M$. Using Schramm's finiteness lemma to sample all the walks from $A$, and using the fact a fixed walk starting from $A$ has a probability to enter $B(x,r)$ which can be made arbitrarily small if $r$ is chosen small enough (because random walk converges to Brownian motion, which will hit $\partial M$ in finite time and not come close to $x$), we easily see in that case that if $r$ is small enough, with high probability none of the walks starting from $A$ enters $B(x,r)$; in particular no walks (including those starting close to $x$) generate a loop of diameter at least $\eps$ which also intersects $B(x,r)$. 

Now let us suppose that $\partial M \neq \emptyset$. Let $f$ be a fundamental polygon corresponding to $M \simeq \D / F$ where $F$ is the corresponding discrete subgroup of M\"obius maps (thus $f$ is a hyperbolic polygon with $2g$ sides for which opposite sides are identified). We may assume that $f$ lies at some fixed positive distance from $\partial \D$.
 %\note{Can we not mention $F$ in the whole argument, simply fix a particular lift of $B$? Also a bit confused about the same choice of $\eps$ both for the length and for the probability bound.  }
The key idea will be to perform Wilson's algorithm, but \emph{on the universal cover} to generate a wired USF of the graph lifted to the cover. It will be crucial to work in the hyperbolic case. Although there is no direct relation between the wired USF on the cover and the CRSF on ${\cM}$, nevertheless the mass of a loop on ${\cM}$ and of one representative in its cover are equal, which is why we are able to say something.

Let $p^{-1}(\Gamma^\d)$ denote the lift of $\Gamma^\d$ to the universal cover of ${\cM}$ taken to be the disk. By definition, a loop $\ell$ in ${\cM}$ is contractible if and only if its lift to the universal cover is still a loop. Furthermore, it is easy to see that if $\ell$ is a loop in $p^{-1} (\Gamma^\d)$, then
\begin{equation}\label{masscover}
\mass_{p^{-1}(\Gamma^\d)} (\ell) = \mass_{\Gamma^\d} (p(\ell)),
\end{equation}
where $\mass_{p^{-1}(\Gamma^\d)} $ refers to the measure on loops computed in the cover, whereas $\mass_{\Gamma^\d} $ refers to the measure on loops in the surface itself.
Of course, a given contractible loop in the surface has many representatives in the universal cover.

%Nevertheless, if $N$ is a neighbourhood in the unit disc on which $p$ is injective and its image contains $x$ and $B_M(x,r)$

Let $\tilde B$ be a lift of $B = B(x,r)$ on its cover and assume without loss of generality that $\tilde B$ is entirely contained by $f$) then from \eqref{masscover} we have
\begin{equation}\label{masscover2}
\mass_{\Gamma^\d} (\{\ell : \ell \text{ is contractible }, \diam (\ell ) \geq \eps; \ell \cap B \neq \emptyset \}  )  \leq \mass_{p^{-1}(\Gamma^\d)} ( \ell \in \cI( \tilde B) : \diam (\ell ) \geq \tilde \eps )
\end{equation}
where $\tilde \eps> 0$ depends only on $\eps$ and the choice of $f$ (more precisely its distance to $\partial \D$). The inequality comes from the change in the diameter from the manifold to the cover and from the fact that several copies of the same loop are counted separately in the right hand side.

We essentially use the same argument as in the previous case. Let $A$ be an annulus on the unit disc centered at $ \tilde x = p^{-1} (x) \cap f$ and of radii $\tilde \eps/4, \tilde \eps /2$. Observe that a random walk starting from $A$ on $p^{-1} ( \Gamma^\d)$ has a probability to enter $\tilde B$ which can be made arbitrarily small if $r$ is small enough (uniformly in $\delta$ and $x$). This follows from convergence to Brownian motion (stopped when it leaves $\D$) while the walk remains at positive distance from $\partial \D$, and the fact proved in \cref{lem:boundary_convergence} that starting at distance at most $\eta$ from $\partial \D$, a walk is very unlikely to ever come back $(1- \eps) \D$ for $\eta$ small enough. We conclude as above, using Schramm's finiteness lemma (with Wilson's algorithm being rooted at infinity, which is possible since the walk is transient as a consequence of \cref{lem:boundary_convergence}; see e.g. Proposition 10.1 in \cite{LyonsPeres} for details on this algorithm on transient graphs).
\end{proof}

\subsection{Rough control on the $h$-transform }

Throughout this subsection, we work on the graph $\Gamma^\d$ satisfying the assumptions in \Cref{sec:setup} and collect several estimates of hitting probabilies of sets, primarily exploiting the uniform crossing assumption.

We start with a lemma which shows that avoiding a given path $\gamma$ in a ball of radius $R$ roughly speaking amounts to avoiding the entire ball of radius $R$, at least if we start at distance $2R$ from the center of the ball. {Throughout this subsection, we assume that we are in the setup of \cref{sec:setup}}. All the statements in this subsection are about random walk in simply connected discs, and hold for small enough $\delta$, we do not mention the dependancy on  $\delta$ in the statements for the sake of brevity.

\begin{lemma}\label{lem:transience_excursion}
There exists a constant $p > 0$ such that the following holds for all $x, R,n>2R$ and for all set $A \subset B( x, R)$ such that $A$ connects $x$ to $\partial B( x, R)$. For any $v\in B( x, 2R)^c$, the random walk started in $v$ and conditioned to leave $B(x, n)$ before hitting $A$ has a probability at least $p$ not to enter $B( x, R)$:
$$
\P_v \Big( \tau_{B(x,R)} > \tau_{B(x,n)^c} \Big| \tau_A > \tau_{B(x,n)^c}\Big) \ge p
$$
where we recall that for any set $B$, $\tau_B$ is the hitting time of $B$.
(Note in particular that $p$ does not depend on $A$ or $n$).
\end{lemma}
\begin{proof}
Let $h_A(v) = \P_{v}(\tau_A> \tau_{B(x,n)^c})$, and similarly write $h_R = h_{B(x,R)} (v) = \P ( \tau_{B(x,R)} > \tau_n)$. Recall that both are harmonic functions in the appropriate domain.

It is clear that there exists $C< \infty$ such that for all $u, v \in \partial B( x, 2R)$, $$\frac1C h(v) \leq h(u) \leq C h(v)$$
where $h$ is either $h_A$ or $h_R$: this is essentially Harnack's inequality, and can be proved for instance as in Lemma 4.4 of \cite{BLR16}.
 Furthermore, by uniform crossing, there exists $p < 1$ such that for all $u \in B( x, R)$, the probability that a random walk reaches $2R$ without hitting $A$ is at most $p$.

By considering successive crossings between $B(x, R)$ and $B(x, 2R)^c$ (and summing over the number $N$ of such crossings),
\begin{align*}
h_A(u) &\le \sum_{N \ge 0 } \left( \sup_{v \in B(x, 2R)^c} \P_v (\tau_{B(x,R)} < \tau_{B(x,n)^c})  \sup_{v \in B(x, R)} \P_v ( \tau_{B(x, 2R)^c} < \tau_A) \right)^N \sup_{v \in \partial B(x, 2R)} h_R(v)\\
& \le \left(\sum_{N \ge 0} p^N\right)  C h_R(u)
\end{align*}
%\sum_{a_1, \ldots, a_n \in \partial B( x, R)}\sum_{b_1, \ldots, b_n \in \partial B( x, 2R)} h_B(b_n)\prod_{i=1}^n \P_{b_{i-1}}(\text{hit $a_i$})\P_{a_i}(\text{hit $b_i$ before $A$}).
Hence $h_R(u) / h_A(u) \ge c$ for some $c>0$, and this ratio is exactly the desired conditional probability so this concludes the proof.
\end{proof}

We will also use a version of the lemma for walks conditioned on leaving a domain through a particular point; this is essentially similar to the above by taking the point to play the role of $\infty$ in the previous lemma.

\begin{lemma}\label{lem:beurling_target}
For all $\eps>0$ there exists a $p = p(\eps)$ such that the following holds for all $x, R>0, n \ge 2R +1$. For all sets $A \subset B(x, R)^c$ connecting $\partial B(x, R)$ to $B(x,n)^c$ and for any $y \in A$ such that $\P_x( X_{\tau_A} = y ) >0$, for any $v \in B(x, R) \setminus B(x,R(1- \eps)$,
%assume that $P_x( X(\tau_\gamma) = \gamma(0)) > 0$, for all $y \in B(x, R)$ such that $|y - \gamma(0)| \leq R/4$ and $d(y, B(x, R)^c ) \geq \frac{1}{2} |y - \gamma(0)|$, we have
\[
\P_{v} \Big[ \tau_{B(x, R(1- 2\eps))} > \tau_A \Big| X(\tau_A) = y \Big] \geq p.
\]
\end{lemma}
\begin{proof}
The proof follows the same lines as \cref{lem:transience_excursion}. We may assume without loss of generality that $v \in \partial B(x, R(1-\eps))$ by applying the Markov property at the first hitting time of this set (if we hit $A$ before this set then there is nothing to show since we will also not enter $B(x, R(1- 2\eps))$). We use Harnack's inequality to show that
\[
h(v) \geq c \sup_{u \in \partial B( x, R(1-\eps) )} h(u)
\]
where $h(v) = \P_v[X(\tau_A) = y]$. On the other hand starting from any point in $\partial B(x, R(1-\eps))$, there is a probability $p$ uniformly less than 1 to hit $B(x, R(1- 2\eps))$ before $A$, because $A$ connects $\partial B(x,R)$ to $B(x, n)^c$ (so a full turn in the annulus $A(x, R, 2R)$ before touching $B(x, R(1-2\eps))$ would suffice -- this is why we assumed $n \ge 2R$.) We conclude in exactly the same manner as in \cref{lem:transience_excursion}, by summing over the number of crossings between distance $R(1-\eps)$ and $R(1- 2\eps)$.
\end{proof}

Finally we will need to say that a walk conditioned to avoid a path from $x$ to $\partial B(x, R)$ is unlikely to enter deep into $B(x, R)$, even if it starts adjacent from the endpoint of the path.
\begin{lemma}\label{lem:backtrack_excursion}
There exists constants $C , \alpha> 0$ such that the following holds for all $x$ and for all $r \leq R \leq n$ and for all set $A\subset B(x, R)$ connecting $x$ to $\partial B( x, R)$. For all $v \in B(x, R)^c$ (including $v$ adjacent to $A$)
\[
\P_v\Big[ \tau_{B(x, r)} \leq \tau_{B(x, n)^c} | \tau_A > \tau_{B(x, n)^c} \Big] \leq C \left(\frac{r}{R}\right)^\alpha
\]
The same bounds holds for a walk conditioned on hitting a set $B \subset B(x, n) \setminus B(x, R)$ at a particular point as in \cref{lem:beurling_target}.
\end{lemma}
\begin{proof}
Note that without loss of generality we can suppose that $r \leq R/2$ by taking $C$ large enough. We also assume to simplify notations that $n \geq 2R$ and (by the strong Markov property) $|v-x| \leq 2R$.
We fix $u$ with $\tfrac{4}{3} R \leq |u-x| \leq \tfrac{5}{3} R$, we clearly have by Harnack's inequality and the Markov property
\[
\P_v[ \tau_{B(x, r)} \leq \tau_{B(x, n)^c} \leq \tau_{A}  ] \leq C \P_v[ \tau_{A} \geq \tau_{B(v, R/3)^c} ] (\frac{r}{R})^\alpha \P_u[ \tau_{B(x, n)^c} \leq \tau_{A} ].
\]
Indeed the walk must first go at least distance $R/3$ from its starting point avoiding $A$, after first reaching distance $r$ it has to come back to distance $R$ without making any full turn, which has probability less than $(r / R)^\alpha$ for some $\alpha$, and then it must exit from distance roughly $3R/2$ which has a probability comparable to the same exit probability from $u$ by Harnack.

On the other hand, using Proposition 4.6 in \cite{uchiyama}, we see that the walk conditioned on exiting $B(v, R/3)$ before hitting $A$ has a positive probability to do so far away from $A$, say at a point $v'$ with $|v' - x| \geq 11R/10$. From $v'$, by \cref{lem:transience_excursion}, the walk conditioned to exit $B(x, n)$ before hitting $A$ has a positive probability not to enter $B(x, R)$, in which case the overall trajectory does not enter $B(x, r)$, in other word, we have (using again Harnack inequality to compare $v'$ to $u$)
\[
\P_v( \tau_{B(x, n)^c} \leq \tau_A \wedge \tau_{B(x, r)} ) \geq \frac{1}{C} \P_v[ \tau_{A} \geq \tau_{B(v, R/3)^c} ]\P_u[ \tau_{B(x, n)^c} \leq \tau_{A} ].
\]
Putting together both bounds completes the proof of the first bound. The second is done similarly.
\end{proof}

%\subsection{Coupling in a single walk.}
%
%
%In this section, we let $\mu$ denote the law of the loop-erasure of a walk started at some point $x$ and stopped when it exists $B(x, 2^N\delta)$. We show that if two walks agree
%\begin{lemma}\label{lem:coupling_1W}
%
%\end{lemma}

\section{Wired scaling limit for a pair of paths}
\label{S:Lawler}

\subsection{Overview}

%We now concentrate on proving convergence of $\{\mathfrak B^{\mathsf{loc}, \d}_{ij}:1 \le  i \le 2g+b-2, j \in \{1,2\} \}$. By independence, we only prove the convergence for $i =1, j \in \{1,2\}$. To that end, we lift the neighbourhood $N_{\mathsf e_1}$ to the universal cover, and suppose $x=\tilde b_{11}, y = \tilde b_{12}$ are the lifts of the endpoints of $b_{11}, b_{12}$. Without loss of generality, we assume that this lift is actually the unit ball $B(0,1)$.

In this section, we work with the following setup. We have a simply connected surface $D$ which, without loss of generality we choose to be the unit disc. We also have a graph $D^\d$ embedded in it, and a neighbouring pair of vertices which in this section we call $x^1$ and $ x^2$ (these should not be confused with the choice of punctures $x_1, \ldots, x_{\mathsf{k}}$ in $M$). Consider the UST on $D^\d$, wired at the natural boundary $\partial$ coming from the embedding of $D^\d$ in $D$.
Recall that a branch of a wired UST started at $x$ is the path of the UST  in $D^\d$ started at $x$ and ending in the boundary of $D^\d$. (Alternatively, we may orient the wired UST towards the boundary of $D^\d$, then the branch starting at $x$ is the path obtained successively following the unique outgoing arrow coming out of the current vertex, starting from $x$.)

We assume that  $x^1, x^2 \in D^\d$ are at distance $O(\delta)$ from 0. We aim to show the existence of a scaling limit for the branches $\bs \eta = (\eta^1, \eta^2)$ in the wired UST on $D^\d$ emanating from $x^1$ and $x^2$ respectively, conditioned on the event $\cA$ that $\eta^1 $ and $\eta^2$ are disjoint.

Without loss of generality we assume $\delta$ is of the form $2^{-N}$ for some $N\ge 1$. Let $B_n = B(0,2^n\delta)$ where $1 \le  n \le N $.
Let $ \mu =  \mu_{x^1,x^2}$ be the  law of a pair of self-avoiding paths generated by loop-erasing two \emph{independent} simple random walks started from $x^1$ and $x^2$ respectively, run until leaving $B_N$. Let $\underline{ \lambda} =\underline{\lambda}^\d_{x^1,x^2}$ denote the law of $\mu$ but conditioned on the event that the random walk started from $x^1$ does not intersect the loop erasure of the random walk started from $x^2$. In other words, $\underline{\lambda}$ is the law of two branches of a wired UST in $D^\d$ started from $x^1,x^2$, conditioned on $\cA$.
\begin{figure}[h]
    \centering
    \includegraphics[scale = 0.5]{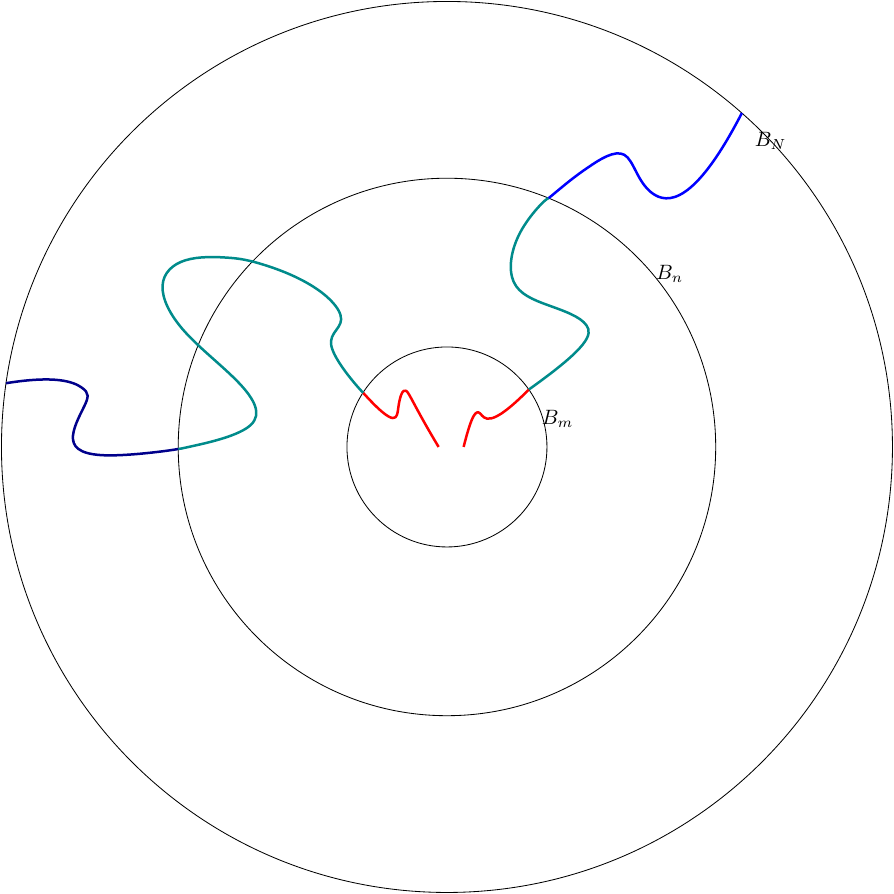}
    \caption{The red part is ${\bs \eta}_m$, the blue part is $\dot {\bs \eta}_{n,N}$ and the dark cyan part is $\bs \eta^*_{m,n}$.}
    \label{fig:path_cut}
\end{figure}

Given a pair of paths $\bs \eta \in \cA$, and given $1\le m \le n \le N$, we may decompose uniquely these paths into three portions (see \Cref{fig:path_cut}):
\begin{equation}\label{eq:pathdec}
\bs \eta = \bs \eta_m \oplus \bs \eta^*_{m,n} \oplus \dot {\bs \eta}_{n, N}
\end{equation}
where $\oplus$ denote the concatenation of adjacent paths, and this writing implicitly indicates that $\bs \eta_m$ is the restriction of $\bs \eta$ until its first exit of $B_m$; and $ \dot {\bs \eta}_{n, N}$ is the restriction of $\bs \eta$ after its last visit to $B_{n}$. Then $\bs \eta_n \in \cA_n$, where $\cA_n$ is the set of disjoint pairs of paths from $x^1, x^2$ to the first exit of $B_n$. Thus with this notation, $\cA = \cA_N$.

The main result we prove in the remainder of this section is the following, which says that the scaling limit exists, and does not depend in a significant way on the beginning of the paths: in fact, given the beginning of the paths $\bs \eta_m$, this has only an effect on at most $O(1)$ further scales. Also, the intermediate portion of the path is unlikely to stray far away from $B_n$.

\begin{prop}\label{prop:local_convergence}
There exists $\alpha>0$ such that the following holds. Uniformly over $\delta>0$, $1\le m \le N$, any pair of paths $\bs \gamma_m \in \cA_m$ and $\tilde { \bs  \gamma}_m \in \cA_m$, and $n = m+k$,
$$
d_{TV} \Big( \underline{\lambda}^\d (\dot{\bs\eta}_{n,N} | \bs \eta_m = \bs \gamma_m ) ;  \underline{\lambda}^\d (\dot{\bs\eta}_{n,N} | \bs \eta_m = \tilde {\bs \gamma}_m ) \Big) \le  e^{-\alpha k},
$$
Finally,
$$
\P ( \bs \eta^*_{m,n} \cap B_{n+k}^c \neq \emptyset) \le e^{-\alpha k}
$$
\end{prop}

\begin{prop}\label{prop:scalinglimit_disc}
  As $\delta \to 0$, $\underline{\lambda}_{x^1,x^2}^\d$ converges in law (in the Hausdorff sense) towards a conformally invariant limiting law on pair of disjoint paths (except at their starting points).
\end{prop}

The rest of this section is devoted to a proof of \cref{prop:local_convergence,prop:scalinglimit_disc}. {The proof borrows the major ideas from Lawler's construction of infinite two-sided loop-erased random walk on $\Z^d$ \cite{Lawler2sided}, combined with robust estimates of loop measures in our more general setting. Readers familiar with Lawler's argument may skip to \cref{sec:localglobal} admitting the results in this section.}

%\note{Disjoint remains to be proved...}

\medskip Let $\cL$ be the set of loops (in $D^\d$) as in \cref{sec:loop_measure}. For any set of vertices $V$ in $D^\d$ recall that $\cL(V) \subset \cL$ denotes the set of loops staying in $V$. We also denote by $\cI(\bs \eta)$ the set of loops $\ell \in \cL$ which intersect \emph{both} paths in $\bs \eta$ {(instead of $\cI(\eta_1 ) \cap \cI(\eta
_2)$; note that $\cI ( \bs \eta)$ differs from $\cI ( \eta_1 \cup \eta_2)$)}.
%
%We now define the relevant measure on pairs of loop erased simple random walk conditioned to avoid each other. For $x ,y \in V \setminus \partial$ and $x \neq y$, let $\cA(x,y)$ denote the set of pairs of self-avoiding paths joining $x$ to $\partial$ and $y$ to $\partial$, and which do not intersect each other. Let $\boldsymbol \mu = \boldsymbol \mu_{x,y}$ be the  law of a pair of self-avoiding paths generated from two independent simple random walks started from $x$ and $y$ respectively. Let $\boldsymbol \lambda = \bs \lambda_{x,y}$ denote the law of $\boldsymbol \mu$ but conditioned on the random walk started from $x$ not intersecting the loop erasure of the random walk started from $y$.  In other words, $\bs \lambda_{x,y}$ is the law of two branches of a uniform spanning tree on $G$, wired in $\partial$, and conditioned not to intersect each other.
By \cref{lem:law_pair,C:conditionalpair},
the Radon-Nikodym derivative of $\underline{\lambda}$ with respect to $\mu$ is given by
\begin{equation}\label{lambdaprop}
\frac{d  \underline{\lambda} }{d \mu} (\bs \eta) \propto{\exp\Big(-
%\sum_{\ell \in \cL, \ell \cap \eta_1 \neq \emptyset, \ell \cap \eta_2 \neq \emptyset}\mass(\ell)
\Lambda ( \cI(\bs \eta))\Big)  1_{\bs \eta \in \cA} }
%{
%\sum_{\boldsymbol \eta \in \cA} \exp\left(-\sum_{\ell \in \cL, \ell \cap \eta_1 \neq \emptyset, \ell \cap \eta_2 \neq \emptyset}\mass(\ell) \right)  1_{\boldsymbol \eta \in \cA} \boldsymbol \mu( \boldsymbol \eta)}.
\end{equation}
Note that loops that surround the origin necessarily intersect $\bs \eta$. Therefore in the above expression we may also restrict to loops that do \textbf{not} surround the origin. We call $\masss$ the measure $\mass$ restricted to such loops. This is useful because it suppresses a very large (essentially infinite when $N = \infty$ ) term which occurs everywhere and is thus irrelevant.
%These measures are not probability measures, nor do they have an obvious probabilistic interpretation. Nevertheless they satisfy convenient algebraic decomposition formally reminiscent of the Markov property, and are convenient because when $n \to \infty$ they give us the desired law, up to a global constant (corresponding to a partition function of loops surrounding the origin). In reality, we could have done the same thing directly with $\mass$ rather than $\masss$. In that case, the limit does not even need to be renormalised to give us the desired law. Still, it is a little bit more convenient in what follows to use $\masss$ rather than $\mass$, not least because we remain close to Lawler's paper \cite{Lawler2sided} in this way.
%\end{remark}

\textbf{The Markov chain.} Following the strategy of Lawler in \cite{Lawler2sided}, it is convenient to view the pair of paths $\bs \eta$ as the terminal value of a Markov chain $(\bs \eta_1, \bs \eta_2, \ldots, \bs \eta_N)$ where for $1\le m \le N$, $\bs \eta_m$ is as in \eqref{eq:pathdec}. It is important to note right away that the transition probabilities of this chain (from $\bs \eta_m$ to $\bs \eta_{m+1}$, say) differ significantly from loop-erasing paths starting $\bs \eta_m$ until they reach $B_{m+1}^c$. Moreover, one cannot immediately use Wilson's algorithm to sample the transitions of this chain. Rather, we will understand it as an $h$-transform in the following manner.

For $1 \le m \le N$, we denote by ${\lambda_m}$ the following \textbf{un}normalised measure on paths reaching $B_m^c$:
\begin{equation}\label{lambdam}
{\lambda_m} (\bs \eta_m) = \exp\Big(-\masss( \cI_m(\bs \eta_m) \Big)  1_{\boldsymbol \eta_m \in \cA_m}  \mu(\bs \eta_m),
\end{equation}
where in the right hand side, $ \mu(\bs \eta_m)$ denote the law under $\mu$ of the first piece in the decomposition \eqref{eq:pathdec} and $\cI_m(\bs \eta_m)$ denotes the set of loops \textbf{in $B_m$} which intersect both paths of $\bs \eta_m$. Note that when $m = N$, the right hand side of \eqref{lambdam} corresponds to the right hand side of \eqref{lambdaprop}, and so is an unnormalised version of $\underline \lambda$. However when $m \neq N$, the restriction to loops that stay in $B_m$ is somewhat arbitrary and does not have a clear interpretation; however, on the algebraic level it is convenient to do so, as we will see.

We now write the transitions of the Markov chain $(\bs \eta_1, \ldots, \bs \eta_N)$ in terms of $\lambda_m$. First we need to introduce an unnormalised ``marginal distribution''  of $\lambda_n$ on $\cA_m$ (where $1\le m \le n$) by setting
\begin{equation}
  \label{marginallambda}
  \lambda_n(\bs \eta_m ) = \sum_{\bs \eta_n \succeq \bs \eta_m} \lambda_n (\bs \eta_n),
\end{equation}
where $\bs \eta_m \in \cA_m$ and the sum is over $\bs \eta_n \in \cA_n$ such that the initial portion of $\bs \eta_n$ is $\bs \eta_m$. We wrote ``marginal distribution'' with quotation marks because not only is it unnormalised, but also, if $n \neq N$, if we were to divide by its total mass, the corresponding probability measure would not correspond to the law of $\bs \eta_m$ under $\underline{\lambda}$, and in fact does not have any clear probabilistic interpretation (in terms of the probability measure $\underline \lambda$). To obtain the true marginal law of $\bs \eta_m$ under $\lambda_N$ (or equivalently, up to a scaling factor, $\underline{\lambda}$), we introduce the ratio
\begin{equation}
  \label{htransform}
  h(\bs \eta_m) = \frac{\lambda_N( \bs \eta_m)}{\lambda_m(\bs \eta_m)}.
\end{equation}
This quantity is denoted in Lawler's paper \cite{Lawler2sided} by $b(\bs \eta_m)$ (or rather $b(\bs \gamma_n)$ with Lawler's notation) but we choose the letter $h$ as $h$ is used as an $h$-transform, as we are about to make clear. Note also that informally, the quantity $h(\bs \eta_m)$ represents how difficult it is to achieve the rest of the constraint $\cA$, given the initial portion $\bs \eta_m$. The normalisation is chosen so that $h (\bs \eta_m)$ is typically of order $\exp ( - \text{const.}(N-m))$, i.e., it is a normalised probability, and the rest of the constraint has a cost of one per scale.

Given the measures $\lambda_m$ and $h$, we may now write the transition probabilities of the Markov chain as follows:
\begin{equation}
  \label{Markovchain}
  p (\bs \eta_m, \bs \eta_{m+1} ) = \frac{\lambda_{m+1} (\bs \eta_{m+1})}{\lambda_m (\bs \eta_m)} \frac{h( \bs \eta_{m+1})}{ h (\bs \eta_m)}.
\end{equation}
Note indeed that by definition of $\lambda_m$ and $h$, this really is the same as $\underline{\lambda} (\bs \eta_{m+1} | \bs \eta_m)$. In view of \eqref{Markovchain} it is convenient to introduce an ``unnormalised conditional probability''
\begin{equation}\label{lambdacond}
 \lambda_{m+1} (\bs \eta_{m+1} \mid \bs \eta_m): = \frac{\lambda_{m+1} (\bs \eta_{m+1})}{\lambda_m (\bs \eta_m)}
\end{equation}
so that \eqref{Markovchain} becomes
\begin{equation}
  \label{Markovchain2}
  p (\bs \eta_m, \bs \eta_{m+1} ) =  \lambda_{m+1} (\bs \eta_{m+1} \mid \bs \eta_m) \frac{h( \bs \eta_{m+1})}{ h (\bs \eta_m)}.
\end{equation}
Formally it is now clear from \eqref{Markovchain2} why we want to think of $h$ as an $h$-transform term, but we caution the reader that the term $ \lambda_{m+1} (\bs \eta_{m+1} \mid \bs \eta_m)$ is not really a conditional probability (although its total mass is typically of order 1).

\paragraph{Main steps of the proof.} Our aim is now to couple the Markov chain starting from two possibly different initial conditions $\bs \gamma_m$ and $\tilde{\bs \gamma}_m$. Roughly speaking, we need to show two things:

\begin{itemize}
  \item Whatever $\bs \gamma_m$ and $\tilde{ \bs \gamma}_m$, there is always a positive chance that we can couple them in the next step (meaning that they agree for a positive fraction of the last scale).

  \item If $\bs \gamma_m $ and $\tilde{\bs \gamma}_m$ agree on many scales, then we can couple them also at the next scale with very high probability in such a way that the agreement also includes the last scale.
\end{itemize}

Roughly speaking, to deal with the first point it will suffice to get up to constants estimates on $h$ and separately show that different scales are independent ``up to constant'' for the conditional measures \eqref{lambdacond}. A very useful tool for this will be a ``quasi-independence'' type estimate (\cref{lem:loop_independence_scale,lem:quasi_independence}), showing respectively decorrelation of scales for $\mu$ and for loop measures. This depends on being able to say that most of the relevant pairs of paths $\bs \eta_m$ are well separated (the two paths are not too close to one another in the last scales). We prove such a statement in two different versions, in Lemmas \ref{lem:separation} and \ref{lem:separation_end}.

To deal with the second point, we show that when two pairs of paths $\bs \gamma_m$ and $\tilde{\bs \gamma}_m$ agree on the last $j$ scales, then the ratio $h(\bs \gamma_m)/ h (\tilde{\bs \gamma}_m)$ is quite close to 1 (in fact, we allow the error to be quite small but at least $e^{- j/4}$) while the ratio of the conditional measures are even closer to 1 (by at least $e^{- j/2}$). This shows we can couple the Markov chains in the next step with exponentially high probability in $j$. This is stated in Lemma \ref{lem:lambda_b_compare}, but the crucial argument is the one saying that there are very few loops intersecting both $B_{m-j}$ and $B_m^c$ without surrounding the origin (Proposition \ref{lem:loop_crossing}).

\begin{remark} All in all, as transpires from this summary, the $h$-transform term $h$ is viewed as a perturbation for which relatively weak estimates are obtained. On the other hand, we need to control the conditional measures \eqref{lambdacond} quite precisely. But these only correspond to adding one scale, and can be handled by controlling the conditional law of $\mu$ (can be sampled with Wilson's algorithm) and loops of a single scale.
\end{remark}

\subsection{Quasi-independence of loops}
\label{S:QIloops}
In this section, we prove the analogue in our setup of Proposition 2.27 in \cite{Lawler2sided} (with a few technical differences, see \cref{R:problem}). If $\bs \eta = (\eta^1, \eta^2)$ is a pair of paths, recall that $\cI(\bs \eta)$ is the set of loops which intersect \textbf{both} paths in $\bs \eta$. It may also on occasion be useful to consider the set of loops which are only required to intersect the first or the second of these two paths: thus we introduce $\cI^1 (\bs \eta) = \cI( \eta^1)$ and $\cI^2 (\bs \eta) = \cI( \eta^2)$.
%If a reference point $x$ has been chosen, we will also write $\cI_n$ for loops which are also required to stay in a ball $B_n$ of radius $2^n\delta$ around $x$.

We now introduce \textbf{Separation events}, used throughout the rest of this proof, defined as follows. First, we let $\mathtt{Sep}_n = \mathtt{Sep}(\cA_n)$ to be the set of $\bs \gamma = (\gamma^1, \gamma^2)$ in $\cA_n$ (pairs of paths which do not intersect up to leaving $B_n$) such that the distance between $\gamma^1$ and $\gamma^2$ inside $B_{n} \setminus B_{n-1/2}$ is at least $2^{n-3}\delta$. Similarly we let $\dot{\sep}_{n, N}$ be the set of pairs of paths $\dot{ \bs \gamma}_{n, N}$ connecting the inside to the outside of the annulus $B(N) \setminus B(n)$ such that the distance between their restrictions to $B_{n+1/2} \setminus B_{n}$ is at least $2^{n-1} \delta$. Finally we let $\mathtt{Sep}_{m,n}$ be the set of pairs of paths $\bs \eta_N \in \cA_N$ such that if we write $\bs \eta_N = \bs \eta_m \oplus \bs \eta^* \oplus \bs \dot{\bs \eta}_{n, N}$ as in \eqref{eq:pathdec}, then
\begin{itemize}
  \item $\bs \eta^* \subset A: = A (x, 2^{m-1}, 2^{n+1})$;
  \item $\bs \eta_m \in \sep_m$ and $\dot{ \bs \eta}_{n, N} \in \dot{ \sep}_{n, N}$.
  \item $d ( \eta^{*,1} , \eta^2 ) \ge 2^{m-2}$ and vice-versa exchanging the role of $\eta^1$ and $\eta^2$. ($d(U,V) = \inf_{u\in U, v\in V} |u-v|$.)
\end{itemize}
As should be clear from the notations, $\sep_m$ should be thought of as saying that the paths are separated just \emph{before} scale $m$, $\dot{\sep}_{n,N}$ says that the paths are separated just \emph{after} scale $n$, and $\sep_{m,n}$ says that the paths are separated \emph{between} the two scales and do not make long excursions away from the scales.

%$\bs \gamma = (\gamma_1, \gamma_2)$ going from $\partial B_m$ to $B_n^c$ without returning to $B_m$,
%such that the Euclidean distance between $\gamma_1$ and $\gamma_2$ in $B_{m+\frac1{10}} \setminus B_m$ is at least $2^{m-1}\delta$.

\begin{lemma}\label{lem:loop_independence_scale}\label{L:QIloops}
There exists a constant $c>0$ such that for all $1\le  m < n\le  N $,
\begin{multline}
c \exp\Big(-\masss (\cI_m(\bs \eta_m))  -\masss (  \cI_{n}(\bs {\dot \eta_{m+1,n}}))\Big)
 1_{\mathtt{Sep}_{m,m+1}} \le \exp\left (-\masss( \cI_n(\bs \eta ) ) \right) \\
 \le 
 \exp\Big(-\masss (\cI_m(\bs \eta_m))  -\masss (  \cI_{n}(\bs {\dot \eta_{m+1,n}}))\Big).
\end{multline}
\end{lemma}
%where
%
%$\cG = \cG_{m,n}$ is an event of \textbf{separation} defined by:
%\begin{itemize}
%  \item $\bs \eta^* \subset A: = A (x, 2^{m-2}, 2^{m+3})$;
%  \item $d ( \eta_1 \cap A, \eta_2 ) \ge 2^{m-3}$ and vice-versa exchanging %the role of $\eta_1$ and $\eta_2$. (Recall that $d$ is distance on the surface, and $d(U,V) = \inf_{u\in U, v\in V} d(u,v)$.)
%\end{itemize}
% \notet{$\bs \eta $is separated in the middle and $\bs \eta^*$ does not exit from $B(x, 2^{m+k})$}
%In fact we will get this with $C = 1$. \note{change it to $C=1$ in the inequality anyway?}

\begin{proof}
%{\color{red}*******Upper bound, loop measure of long loops, lower bound long loops because of separation of $\bs \eta$ in the middle******}
For the upper bound, we simply observe that
$$
\cI_n(\bs \eta) \supset \cI_m (\bs \eta_m) \cup \cI_n(  \bs{ \dot\eta}_{m+1, n} ),
$$
and this union is in fact disjoint (in the second set, loops are forced to leave the ball of radius $2^m\delta$). Therefore,
$$
\masss( \cI_n(\bs \eta ) \ge \masss(\cI_m (\bs \eta_m)) + \masss( \cI_n(  \bs{ \dot\eta}_{m+1, n} )
$$
which leads to the desired upper bound.
%\footnote{The argument in Proposition 2.27 of \cite{Lawler2sided} seems more complicated: his argument shows the upper bound with $\cI_m( \bs \eta_m)$ replaced by $\cI_n(\bs \eta_m)$, which is stronger than what is claimed in the statement.}
%it is enough to control the mass of loops who intersect both paths at least one path in $\bs {\dot \eta_{n+1, m}}$ and one in $\bs \eta_{m}$. This is clearly controlled by \cref{lem:loop_crossing}. \note{I think we need to modify the lemma by removing all loops that surround the origin from all the measure.}

For the lower bound, we note that the missing loops in the above upper bound must be of the following type: they either intersect the middle part but not the extremities, or they intersect the early part but leave the ball of radius $2^m \delta$, or stretch between the ball $B_m$ of radius $2^m\delta$ and the outside of $B_{m+1}$: thus
\begin{equation}\label{lowerbound_qi}
\cI_n(\bs \eta) \setminus (\cI_m (\bs \eta_m) \cup \cI_n(  \bs{ \dot\eta}_{m+1, n} ) ) \subset  \Big[\bigcup_{j=1,2} \cI_n (\eta^{*,j}) \cap \cI_n(\eta^{3-j}) \Big]\cup \Big[ \cI_n( \bs \eta_m ) \setminus \cI_m( \bs \eta_m ) \Big] \cup \Big[ \cI_n ( B_m) \cap \cI_n( B_{m+1}^c)\Big]
\end{equation}
We bound the mass of each of these three sets separately. For the first one, we note that $\eta^{*,j} \subset A = A (x, 2^{m-1}, 2^{n+1})$ and $\eta^j \cap A$ is macroscopically far from $\eta^{3- j}$ on $\mathtt{Sep}_{m,m+1}$. Therefore, loops in this set which also stay in $B_{m+10}$ have a bounded mass by \cref{cor:long_loops}. Loops that in this set that also leave $B_{m+10}$ have a mass (for $\masss$) which is bounded by \cref{lem:loop_crossing}.

Loops in the second set intersect both paths in $\bs \eta_m$ but leave $B_m$. If they also leave $A$ then they must be macroscopic and are handled as above (by a combination of \cref{cor:long_loops} and \cref{lem:loop_crossing}). Otherwise, they intersect at least one of the paths in $A$ and also the other path anywhere, but by assumption these loops have a diameter comparable to that of $B_m$ and so have bounded mass by \cref{cor:long_loops}.

The last term has bounded mass under $\masss$ directly by \cref{lem:loop_crossing}. This concludes the proof.
%
%we need to control the mass of loops that intersect both paths in $\bs \eta^*$. By separation a loop must have diameter at least $\epsilon 2^m\delta$ to intersect both paths in $\bs \eta^*$, therefore the mass of such loops that stay in $B(0, 2^{m+2k}\delta)$ is bounded by \cref{T:loopsoupRW}. On the other hand the mass of loops who enter both $B(0, 2^{m+k}\delta)$ and $B(0, 2^{m+2k}\delta)^c$ but do not surround $x$ is bounded by \cref{lem:loop_crossing}.
\end{proof}

\begin{remark}\label{R:problem}
  The analogous Proposition 2.27 in \cite{Lawler2sided} is different in a few details. One of them is that the proof of the upper bound in \cref{lem:loop_independence_scale} gives a slightly stronger statement than required ($\cI_m(\bs \eta_m)$ is replaced by $\cI_n( \bs \eta_m)$, so the union is not disjoint). Also, our definition of the event of separation $\sep_{m,n}$ is a little different since it is not sufficient to require $\eta^{*, 1}$ to be far from $\eta^2$ and vice-versa to get a bounded mass in the second term of \eqref{lowerbound_qi}: indeed, if the paths $\eta^1$ and $\eta^2$ get very close to one another and to the boundary of $B_m$ before separating, many loops may intersect both paths and leave $B_m$. Our definition prevents this from happening by adding extra cushioning in what is required to be separated.
\end{remark}

\subsection{Quasi-independence of pairs of paths}

\begin{comment}
We start with some preliminary estimates. The first lemma states that the loop erased random walk pieces are roughly independent across scales. Let $N_{\min} < m < n-1 < N_{\max} $. Let us define $\bs {\dot \mu}_{m+1,n}$ denote the projection of the probability measure $\bs \mu_n$ on the pair of paths from their last hit of $\cA_{m+1}$ to the first hit of $\cA_n$ (i.e. these are simple paths avoiding each other and lying completely inside $B_{n}\setminus B_{m+1}$). Let $\dot \cA_{m,n}$ denote the support of $\bs{\dot \mu}_{m,n}$. For $\bs\eta  \in \cA_n$ we decompose it as
\begin{equation}
\bs \eta  = \bs \eta_m \oplus \bs \eta^* \oplus \bs {\dot \eta_{m+1,n}} \label{eq:path_decomposition}
\end{equation}
with $\bs \eta_m$ being the portion of $\bs \eta$ up until the first hit of $\partial B(x,2^m\delta)$, $\bs {\dot \eta_{m+1,n}}$ being the portion of the path from the last hit of $\partial B(x,2^{m+1}\delta)$ to the first hit of $\partial B(x,2^n\delta)$ and $\bs \eta^* $ is the rest of $\bs \eta$.
%For $\bs\eta  \in \cA_n$ and decompose it as $\bs \eta  = \bs \eta_m \oplus \bs \eta^* \oplus \bs {\bar \eta_{n+1,m}}$ with $\bs \eta_m$ being the portion of $\bs \eta$ up until the first hit of $\partial B(x,2^m\delta)$, $\bs {\bar \eta_{n+1,m}}$ being the portion of the path from the last hit of $\partial B(x,2^{n+1}\delta)$ to the first hit of $\partial B(x,2^m\delta)$ and $\bs \eta* $ is the rest of $\bs \eta$.
\end{comment}
We now focus on proving a similar quasi-independence lemma for the measure $\mu$ (independent pair of loop erased random walk).

\begin{lemma}[Quasi-independence]\label{lem:quasi_independence}\label{L:QIpaths}
 There exists constants $c, C>0$ such that for all  $\delta>0$, and for all $1\le m <  N $, we have
\begin{equation}\label{eq:qipaths}
c\mu(\bs \eta_m) \mu (\bs {\dot \eta_{m+1,N}}) \le \sum_{\bs \eta^*} \mu ( \bs \eta_m \oplus \bs \eta^* \oplus \bs {\dot \eta_{m+1,N}}) \le  C \mu(\bs \eta_m)  \mu (\bs {\dot \eta_{m+1,N}}).
\end{equation}
Furthermore, if $\bs \eta_m \in \mathtt{Sep}_m$ and if $\dot{\bs \eta}_{m+1, N} \in \dot{\sep}_{m+1, N}$, then
\begin{equation}\label{eq:qirestrict}
c\mu(\bs \eta_m) \mu (\bs {\dot \eta_{m+1,N}}) \le \sum_{\bs \eta^*\in \mathtt{Sep}_{m, m+1}} \mu ( \bs \eta_m \oplus \bs \eta^* \oplus \bs {\dot \eta_{m+1,N}}) \le  C \mu(\bs \eta_m)  \mu (\bs {\dot \eta_{m+1,N}})
\end{equation}
where the event $\mathtt{Sep}_{m, m+1}$ is defined at the beginning of \cref{S:QIloops}.
\end{lemma}
\begin{proof}
%{\color{red}**Crossing estimate of excursions and harnack*****}
Note that despite being stated for a measure on pair of paths, by independence we only have to prove the equivalent statement when $\eta$ is a single path.
We start with the explicit density of $\mu$. To lighten up notations, we will use $\eta_-$ and $ \eta_+$ for $ \eta_m$ and $\dot \eta_{m+1, N}$ respectively. We also call $x_-$ and $x_{+}$ respectively the last and first points of $\eta_-$, $\eta_{+}$ and we let $\tau_{\pm}$ be the hitting time of $\eta_\pm$ or $B(0, 2^N\delta)^c$, and in general $\tau_{A}$ will be the first hitting of $A \cup B(0, 2^N\delta)^c$. By \eqref{partial_single_SC} and \eqref{eq:marginal_single1} we have %\note{maybe add the $q$ terms below  as this is not quite correct?}
\begin{align}
\mu(\eta_-) &= q(\eta^-) e^{\mass(\cI(\eta_-))} \P_{x_-}[ X(\tau_{-}) \not \in \eta_- ] \label{E:decompositionloopsoup1} \\
 \mu(  \eta_{+} ) &= q(\eta^+) e^{\mass(\cI(\eta_+))} \P_{0}[X(\tau_{+} ) = x_{+}] \label{E:decompositionloopsoup2}\\
\sum_{ \eta^*}  \mu ( \eta_- \oplus \eta^* \oplus { \eta_{+}}) & = q(\eta^-) q( \eta^+) e^{\mass( \cI( \eta_- \cup \eta_+))} \P_{x_-} [X(\tau_{\eta_- \cup \eta_{+}}) = x_{+} ] \label{E:decompositionloopsoup}.
\end{align}
We point out that here we need to use $\mass$ and not $\masss$.

%The first identity is well known (see e.g. Proposition 11.2.4 in \cite{Lawlerbook}) and can be checked by direct computation. The second and third identities are perhaps less well known but follow in the same manner.

We will deal separately with the loop measure terms and the terms involving simple random walk. We start with the random walk terms: we aim to show that the product of the first two terms above is comparable to the third term, up to normalising factors independent of $\eta$.
Let $A= A(0, 2^{m+1/4}\delta, 2^{m+3/4} \delta )$ and fix a point $a$ in $A$. We will write $ p\asymp q$ to mean that there exists a constant $c$ (which does not depend on any of the parameters defining $p$ and $q$) such that $cp \leq q \leq c^{-1} p$.

Let us start with the first term. Clearly, by Harnack's inequality,
$$
\P_{x_-}[ X(\tau_{-}) \not \in \eta_- ]     \asymp  \P_{x_-}[X(\tau_{A\cup  \eta_-}) \not \in \eta_-] \P_{a}[ X(\tau_{-}) \not \in \eta_- ]
$$
Now note that by \cref{lem:transience_excursion}, the second term in the right hand side above is (up to constant) independent of $\eta$:
\begin{equation}
\P_{x_-}[ X(\tau_{-}) \not \in \eta_- ]     \asymp \P_{x_-}[X(\tau_{A\cup  \eta_-}) \not \in \eta_-] \P_{a}[ X(\tau_{B_m}) \notin B_m] \label{eq:debut}
\end{equation}
Now consider the third term \eqref{E:decompositionloopsoup2} and observe also that by Harnack's inequality,
\begin{equation}\label{eq:middle0}
\P_{0}[X(\tau_{+} ) = x_{+}]  \asymp \P_{a}[X(\tau_{+} ) = x_{+}]
\end{equation}
Now consider a random walk started from $a$ conditioned on $X(\tau_{+}) = x_+$.
By Lemma 4.4 in \cite{BLR16}, the conditioned walk still satisfies the uniform crossing estimate as long as it is away from $\eta_{+}$. Hence it has a fixed positive probability to go to distance $2^{m-1} \delta$ of $x_+$ without entering $B(0, 2^m\delta)$. Furthermore, once it is there, there is a positive probability $p$ to hit $x_+$ without entering $A$ by \cref{lem:beurling_target}. Hence
$$
\P_a ( \tau_+ < \tau_- | X_{\tau_+ } = x_+ ) \ge p
$$
and hence combining with \eqref{eq:middle0},
\begin{equation}\label{eq:middle}
\P_{0}[X(\tau_{+} ) = x_{+}] \asymp  \P_{a}[X(\tau_{+} ) = x_{+}  ; \tau_+ < \tau_-] = \P_a [X (\tau_{\eta_- \cup \eta_{+}})=x_{+}].
\end{equation}

Finally,  for the third term \eqref{E:decompositionloopsoup}, observe that (again by Harnack's inequality),
\begin{align}
 \P_{x_-} [X(\tau_{\eta_- \cup \eta_{+}}) = x_{+}  ]
 & \asymp \P_{x_-}[X(\tau_{A\cup \eta_-}) \not \in \eta_-] \P_{a} [X(\tau_{\eta_- \cup \eta_{+}})=x_{+}].\label{eq:fin}
\end{align}
Therefore, putting together \eqref{eq:debut}, \eqref{eq:middle} and \eqref{eq:fin}, we get
\begin{equation}\label{E:productprob}
\P_{x_-}[ X(\tau_{-}) \not \in \eta_- ] \P_{0}[X(\tau_{+} ) = x_{+}] \asymp \P_{x_-}[X(\tau_{\eta_- \cup \eta_{+}} ) = x_{+}] \P_a( X_{\tau_B} \not \in B)
\end{equation}
which is of the desired form.

We now turn to the loop measure terms. Recall that $\cI(\eta)$ denote the set of loops intersecting $\eta$. We evidently have
\begin{equation}\label{E:productloopsoup}
\mass ( \cI( \eta_-))  + \mass ( \cI(\eta_+))  - \mass ( \cI ( \eta_- \cup \eta_+)) =  \mass  ( \cI ( \eta_-) \cap \cI( \eta_+)).
\end{equation}
\begin{comment}
\[
e^{\mass(\{ \ell : \ell\cap \eta_- \neq \emptyset\})} e^{\mass(\{ \ell : \ell\cap \eta_{+} \neq \emptyset\})} e^{-\mass(\ell : \ell\cap (\eta_- \cup \eta_{+}) \neq \emptyset\})} =
 e^{\mass(\{ \ell : \ell \cap \eta_- \neq \emptyset, \ell \cap \eta_{+} \neq \emptyset\})} .
\]
\end{comment}
Note that because of possible very long loops the right hand side of the equation is not bounded in general. To account for these loops, we define $\cL$ as the set of loops that surround $0$ and $C$ as the complement of $B( 0, 2^{m+1}\delta)$ and we decompose the mass of loops intersecting $\eta_-$ \emph{and} $\eta_+$ as
\begin{align*}
\mass( \cI (\eta_-) \cap \cI ( \eta_+) ) & = \mass( \cL \cap \cI ( B) \cap \cI (C)) + \mass ( \cL^c \cap \cI ( \eta_-) \cap \cI( \eta_+)) \\
& \ \ - \mass ( \cL \cap \cI ( \eta_-) \cap \cI (C) \cap \cI( \eta_+)^c)\\
&\ \  - \mass ( \cL \cap \cI ( \eta_-)^c \cap \cI( B)  \cap \cI( \eta_+))\\
& \ \ + \mass ( \cL \cap \cI ( B) \cap \cI(C) \cap \cI(\eta_-)^c \cap \cI( \eta_+)^c).
\end{align*}

%\begin{comment}
%\begin{multline*}
%\mass(\ell : \ell \cap \eta_- \neq \emptyset, \ell \cap \eta_{+} \neq \emptyset) = \mass( \ell \in \cL : \ell \cap B \neq \emptyset, \ell \cap C \neq \emptyset ) + \mass( \ell \in \cL^c : \ell \cap \eta_{-} \neq \emptyset, \ell \cap \eta_{+} \neq \emptyset ) \\
%- \mass(  \ell \in \cL : \ell \cap \eta_- \neq \emptyset, \ell \cap C \neq \emptyset, \ell \cap \eta_{+} = \emptyset) \\
%- \mass(  \ell \in \cL : \ell \cap B \neq \emptyset, \ell \cap \eta_- = \emptyset, \ell \cap C \neq \emptyset) \\
%+ \mass( \ell \in \cL : \ell \cap B \neq \emptyset, \ell \cap \eta_{-} = \emptyset, \ell \cap C \neq \emptyset, \ell \cap \eta_{+} = \emptyset )
%\end{multline*}
%\end{comment}
The first term in the first line is independent of $\eta_-$ and $\eta_{+}$ so can be factored into the partition function. The second term (still in the first line) is bounded by a constant which does not depend on $\eta_-$ or $\eta_+$ by \cref{lem:loop_crossing}, since all loops in this set intersect an annulus, cannot have a small diameter (since they must intersect both $\eta_-$ and $\eta_+$) and cannot surround the origin.
Finally, the last three terms are bounded by \cref{T:loopsoupRW} because all three sets require the loops to have a diameter $\gtrsim 2^m \delta$, and
a random walk started in $B(0,  2^{m+2} \delta)$ has a positive probability to hit either of $\eta_-$ and $\eta_{+}$ without creating any loop with such a diameter.

Overall, when we include the partition functions  (combining this argument with \eqref{E:decompositionloopsoup}, \eqref{E:productprob} and \eqref{E:productloopsoup}) we see that there exists a constant $C_{m,N}$ independent of $\eta_-, \eta_+$ (but which a priori depends on $N$ and $m$) such that
\[
\mu(\eta_-)  \mu(\eta_+)\asymp C_{m,N} \sum_{ \eta^*}  \mu ( \eta_- \oplus \eta^* \oplus { \eta_{+}}).
\]
This concludes the proof of \eqref{eq:qipaths} because both sides of the equation are probability measures, so $C_{m,N} \asymp 1$. For \eqref{eq:qirestrict}, observe that by \eqref{E:decompositionloopsoup}, it is clear that the conditional $\mu$-law of $\bs \eta^*$ given $\bs \eta_m$ and $\dot{\bs \eta}_{m+1, N}$ is given by the loop-erasure of a pair of random walks starting from $\bs x_-$ conditioned on $\bs X(\tau_{+}) = \bs x_+$ (this has to be interpreted coordinatewise, i.e. separately for each path). By our assumptions, on $\bs \eta_m$ and $\dot {\bs \eta}_{m+1, N}$, and by uniform crossing (which holds even conditional on the end point), it is clear that there is a positive probability for the event $\sep_{m, m+1}$ to hold: indeed, even the random walk paths have a positive chance of being separated and to stay within the annulus $A(0, 2^{m-1}\delta, 2^{m+2}\delta)$.
\end{proof}

%For $\bs \eta  = \bs \eta_m \oplus \bs \eta^* \oplus \bs {\dot \eta_{m+1,n}} \in \cA_n$, let $\cI_m(\bs \eta_m)$ denote the set of loops inside $B(x,2^m)$ which intersect \textbf{both} paths in $\bs \eta_m$. Also let $ \cI_{n}(\bs {\dot \eta_{m+1,n}})$ denote the set of loops inside $B(x,2^n)$ which intersect \textbf{both} paths in $\bs {\dot \eta_{m+1,n}}$.
%It will also be useful to consider the set of loops which

\subsection{Separation lemma}

\begin{comment}
Define ${\bs \lambda}_{m+1,n}$ denote the measure defined as
\begin{equation}
\bar{\bs \lambda}_{m+1,n} (\dot {\bs \eta}_{m+1, n})  = \exp\left (-\masss(\cI_n (\bs \eta_{m+1,n})) \right)\dot {\bs \mu}_{m+1, n} (\dot{\bs \eta}_{m+1,n}) 1_{\dot {\bs \eta}_{m+1, n} \in \dot \cA_{m+1,n}}\label{eq:end_lambda}
\end{equation}
and recall from \eqref{eq:lambda} that with these notations, $\bar {\bs \lambda}_n = \bar {\bs \lambda}_{0,n}$.
\end{comment}

We recall $\mathtt{Sep}_n = \mathtt{Sep}(\cA_n)$ to be the set of $\bs \gamma = (\gamma^1, \gamma^2)$ in $\cA_n$ (pairs of paths which do not intersect up to leaving $B_n$) such that the distance between $\gamma^1$ and $\gamma^2$ inside $B_{n} \setminus B_{n-\frac{1}{10}}$ is at least $2^{n-1}\delta$.
\begin{lemma}[Separation Lemma]\label{lem:separation}
There exists $c>0$ such that for all $\delta>0$, for all $1\le n < m \le N$, for all $\bs \gamma_n \in \cA_n$
\begin{equation}
\sum_{\bs \gamma_m \in \mathtt{Sep}(\cA_{m})  : \bs \gamma_n \prec \bs \gamma_m}  { \lambda}_{m}(\bs \gamma_m) \ge c { \lambda}_{m}(\bs \gamma_n).
\end{equation}
\end{lemma}

This lemma is stated as Lemma 2.29 in \cite{Lawler2sided}. As stated in that paper, this lemma is in fact an adaptation of results written before in two slightly different but related contexts, see \cite{lawler2016convergence} and \cite{masson2009growth}. Lawler's paper \cite{Lawler2sided} also contains some elements of the proof in an appendix but that may not be sufficient for all readers to get an idea of the proof. We provide here a description of the main ideas in our context.

\begin{proof}[Sketch of proof] For $0 \le u \le 1$, let $\bs \eta_{n+u} \in \cA_{n+u}$. Let $\Delta (\bs \eta_{n+u}) $ be the largest value of $r>0$ such that each endpoint of $\bs \eta_{n+u}$ is at distance at least $r 2^{n}$ from the other path. Suppose that $\Delta (\bs \eta_{n+u}) = 2^{-j}$. The goal is to show that by increasing $u$ slightly to $u'$ this can be improved to $\Delta (\bs \eta_{n+u'}) = 2^{-j+1}$ with high probability (or rather with high mass for $\lambda_n$). We can then iterate this argument using a ``Markovian'' way of discovering the paths $\bs \eta_{n+1}$.
We will take $u' = u + j^2 2^{-j}$, and divide the interval $[ u, u' ] $ into $j^2 $ ``mini-stages'', each of length $2^{-j}$. Let us call a mini-stage $[v,v']$ (with $v' - v = 2^{-j}$) a mini-failure if $\Delta(\bs \eta_{n + v'}) < 2^{-j+1}$, and a mini-success otherwise. Therefore, if $F_{[v,v']}$ is the event of a mini-failure during $[v, v']$, then we have for each $\bs \eta_{n + v} \in \cA_{n + v}$,
$$
\sum_{\bs \eta_{n + v' } \in F_{[v,v']}} \lambda_{n+v'} (\bs \eta_{n + v'}) \le e^{-c} \lambda_{n + v} (\bs \eta_{n + v}).
$$
(We do not even require the paths to be disjoint in the left hand side above, so we get a much stronger bound than we need.)
Indeed, no-matter how small the separation is at its beginning of the mini-stage, the probability of a mini-failure (under $\mu$) is at most $e^{-c}$ for some $c>0$, since the spatial separation between the two boundaries of the annulus $2^{n+v'} - 2^{n+v} $ is greater (up to a constant) than the required separation for a mini-success, i.e., $2^{n-j +1}$ (as can be immediately seen with a Taylor expansion near zero of $x \mapsto 2^x$).
Furthermore, the loop measure term $e^{-\Lambda ( \cI_{n + v'} (\bs \eta_{n + v'}))}$ is smaller than the analogous one on the right.

Therefore, by iterating this, one can see that having only mini-failures at each of the $j^2$ mini-stages has a mass
\begin{equation}
\sum_{\bs \eta_{n + u'} \in \cA_{n+ u'}\cap F_j} \lambda_{n+ u' } ( \bs \eta_{n+ u'} ) \le \exp ( - cj^2) \lambda_{n+u} (\bs \eta_{n+u})
\end{equation}
where the $F_j$ in the sum indicates that $\bs \eta_{n+ u' }$ contains only mini-failures. Moreover, so long as $u' \le 1$, $\lambda_{n+ 1 }  (\bs \eta_{n+ u' }) \le \lambda_{n+u' } (\bs \eta_{n + u'} )$, hence trivially,
\begin{equation}
\sum_{\bs \eta' \in \cA_{n+ u'}\cap F_j} \lambda_{n+ 1 } ( \bs \eta_{n+ u'} ) \le \exp ( - cj^2) \lambda_{n+u} (\bs \eta_{n+u})
\end{equation}
On the other hand, it is clear\footnote{In reality, an annoying technical detail is that this only holds for those $\bs \eta_{n+u}$ which not only satisfy $\Delta( \bs \eta_{n+u} ) \ge 2^{-j}$ but also satisfy a similar separation in an annulus $B_{n+ u} \setminus B_{n+u - 2^{ -j}}$: otherwise there are boundary effects similar to the problem mentioned in Remark \ref{R:problem} in the loop terms. These terms would need to be included in the expression, but do not change the argument significantly.} that
\begin{equation}\label{survive}
\lambda_{n+1} (\bs \eta_{n+u} ) \ge   e^{- cj} \lambda_{n+u} (\bs \eta_{n+u})
\end{equation}
because one way to extend $\bs \eta_{n+u}$ up to scale $n+1$ is by doubling the separation repeatedly at scales
$$
u,u + 2^{ - j},u + 2 \cdot 2^{-j},u + 4 \cdot 2^{ - j}, \ldots, 1$$ (and note that this correspond to an iteration over at most $j $ many scales, and the last scale could be of width smaller than $2^{-j}$ which only helps).
Putting these two bounds together, we see that if $\underline{\lambda}_{n+1}$ is the measure $\lambda_{n+1}$ normalised to be a probability measure,
\begin{equation}\label{failure}
\underline{\lambda}_{n+1}  ( F_j \mid \bs \eta_{n+u}) \le e^{- c' j^2}.
\end{equation}
Since this is really a fixed probability measure, we can apply this argument inductively over $j$: as soon, as we have a mini-success at some mini-stage $[v,v']$, we are free to start again the same argument above with $u$ replaced by $v'$ and $j$ changed into $j-1$.
In this way, if we fix $N_0$ large enough such that
$$
\sum_{j \ge N_0} j^2 2^{-j} \le 1/2, \qquad \sum_{j \ge N_0} e^{ - c'j^2} \le 1/2
$$
we see that with $\underline{\lambda}_{n+1}$ probability greater than 1/2, the two paths are have a mini-success at stage $N_0$ no matter what the initial separation, and that this mini-success occurs at a mini-stage $v'$ satisfying $v' \le 1/2$.  (The first sum ensures that $v'$ is never greater than 1/2, and the second controls the probability of a failure at stage $j$ for any $j \ge N_0$).
From that point, clearly the $\underline{\lambda}_{n+1}$ probability for the two paths to be in $\mathtt{Sep}_{n+1}$ is positive.
\end{proof}

We also need a separation lemma for the end portions of the paths (or rather the beginning of the end of these paths). This corresponds to Lemma 2.30 in \cite{Lawler2sided}. As mentioned in this paper, this result is similar to the above separation lemma (Lemma \ref{lem:separation}). To see the similarity, it is helpful to perform a time reversal: requiring the beginning of $\dot {\bs \eta}_{m+1, N}$ to be separated is the same as requiring the end portion of the time reversal of $ \dot {\bs \eta}_{m+1, N} $ to be separated. This is exactly the sort of events that Lemma \ref{lem:separation} deals with. Alternatively, we can use the same proof but changing the meaning of $\mu$ to be the loop-erasure of a walk starting from 0 and conditioned to leave a certain ball by a given point. In common with \cite{Lawler2sided} we also do not provide an additional proof of this fact. (Note however that Lemma 2.30 in \cite{Lawler2sided} has a typo, as one cannot require the paths to be separated constantly between scales $m$ and $n$ with positive probability independent of $n - m$).
We recall the definition of $\mathtt{Sep}_{m,n}$ at the beginning of  \cref{S:QIloops}.

 \begin{lemma}[Separation lemma for end portions  of paths]\label{lem:separation_end}
There exist constant $c>0$ such that for all $\delta>0$ and $1\le m <n \le N -1$ the following is true. Write for any $\bs \eta_n \in \cA_n$, $\bs \eta_n   = \bs \eta' \oplus \dot {\bs \eta}_{m,n}$ where similar to \eqref{eq:pathdec}, $\dot {\bs \eta}_{m,n}$ denotes the paths $\bs \eta_n$ after their last visit to $B_m$.
\begin{equation}
\sum_{\bs \eta_n \in \cA_n \cap \mathtt{Sep}_{m,m+1}}  \mu(\bs \eta_n)
\exp \left (- \masss (\cI_n (\dot {\bs \eta}_{m+1,n}))\right)
\ge c \sum_{\bs \eta_n \in \cA_n}  \mu(\bs \eta_n)
\exp \left(- \masss (\cI_n (\dot {\bs \eta}_{m+1,n}))\right)
\end{equation}
\end{lemma}

We now state some useful corollaries which we will rely on for the rest of the argument. First of all, we need to introduce $h_m$ which can be thought of as the typical value (up to constants of order 1) of $h(\bs \eta_m)= \lambda_N( \bs \eta_m)/\lambda_m(\bs \eta_m)$ (under the marginal law of the Markov chain, say):
\begin{align*}
h_m &= \sum_{\bs \eta_N \in \cA} \exp \left( - \masss [\cI_N (\dot{\bs \eta}_{m, N} )] \right)\mu (\bs \eta_N) \\
& = \E_\mu  \exp \left( - \masss [\cI_N (\dot{\bs \eta}_{m, N} )] \right)
\end{align*}
In fact it is also typical under $\lambda_m$, and can also be viewed as the \emph{maximal} value of $h(\bs \eta_m)$, up to constants. The connection with typical and maximal values of $h(\bs \eta_m)$ will come from the estimates in  \cref{cor:separation} below. For now we also point out that we can view $h_m$ as a partition function of a measure $\lambda_{m,N}$ defined by  $\exp \left( - \masss [\cI_N (\dot{\bs \eta}_{m, N} )] \right)\mu (\dot{\bs \eta}_{m,N})$. It is this definition which is implicitly taken in \cite{Lawler2sided} with the notation $b_m$ instead of $h_m$ here\footnote{More precisely, $b_m$ in \cite{Lawler2sided} is explicitly taken to be a partition function but the definition of the measure is only implicit.}. More generally, by changing $N$ into $n$, with $1\le m < n \le N$, we could also define $h_{m,n}$ in this manner, but this won't be needed.

%Let $Z(\dot{\bs \lambda}_{m,n}) $ denote the partition function of the measure defined in \eqref{eq:end_lambda}.
Let $Z({ \lambda}_m) $ be the partition function of the measure $\lambda_m$ defined in \eqref{lambdam}. The following corollary, which summarises the information we gather from the quasi-independence lemmas and the separation lemmas, corresponds to Proposition 2.31 in \cite{Lawler2sided} (but with $n = N$ in some cases). %\note{to be checked it's not needed with $n$}
\begin{corollary}\label{cor:separation}
There exists constants $c_1,c_2 >0$ such that for all $\delta>0$, for $1\le m <n \le N$, the following is true
\begin{enumerate}[a.]
\item Fix $\bs \eta_m \in \mathtt{Sep}(\cA_m)$. Then ${\lambda}_{m+1} (\bs \eta_m) \ge c_1 {\lambda}_m (\bs \eta_m)$.
\item $c_1Z( \lambda_n) \le Z(\lambda_{n+1})  \le Z(\lambda_n)  $.

\item $c_1 Z(\lambda_N)/Z(\lambda_{m})   \ge h_{m}\ge c_2 Z( \lambda_N)/Z(\lambda_{m})$

\item
If $\bs \eta_m \in \mathtt{Sep}_m$ then $h(\bs \eta_m)\ge c_1 h_m$

\item If $\bs \eta_m \in \cA_m$ then $h(\bs \eta_m) \le c_1 h_m$.
\end{enumerate}
\end{corollary}

%Before we write the proof, we observe that taking $n =N$ in d. above, the left hand side is $h(\bs \eta_m)$  and the right hand side is comparable to $h_m$ by c. The inequality e. gives us the opposite comparison, so that indeed $h_m$ is the typical value of $h(\bs \eta_m)$, and is comparable to the maximum.

\begin{proof}
Proof of a. If $\bs \eta_m \in \mathtt{Sep}_m$, reasoning as in \cref{lem:beurling_target} and using Proposition 4.6 in \cite{uchiyama}, there is a positive probability for two random walks starting from the endpoints of either path, and conditioned to avoid these paths until leaving $B_N$, to remain separate separated until hitting $B_{m+1}$. In that case the corresponding loop measure term $\exp [ - \masss ( \cI_{m+1} (\bs \eta_{m+1}) + \masss (\cI_m (\bs \eta_m) )]$ is bounded away from zero. This proves a.

Proof of b. The upper bound is trivial because $Z(\lambda_n)$ is trivially decreasing. For the lower bound, we use part a. and the Separation \cref{lem:separation}. More precisely,
$$
Z(\lambda_n) \lesssim \sum_{\bs \eta_n \in \mathtt{Sep}_n} \lambda_n(\bs \eta_n)  \lesssim \sum_{\bs \eta_n \in \mathtt{Sep}_n} \lambda_{n+1}(\bs \eta_n) \le Z(\lambda_{n+1}).
$$
For the first inequality, we use \cref{lem:separation} for $m=1$.
This proves b.

Proof of c. We start with the lower bound.
By \cref{L:QIpaths} and \cref{L:QIloops},
\begin{align*}
  Z(\lambda_N) & = \sum_{\bs \eta_N \in \cA_N} \mu(\bs \eta_N) \exp [ - \masss (\cI( \bs \eta_N)) ]\\
  & \lesssim \sum_{\bs \eta_m \in \cA_m, \dot{\bs \eta}_{m+1, N}} \mu ( \bs \eta_m) \mu(\dot {\bs \eta}_{m+1, N} ) \exp [ - \masss (\cI_m( \bs \eta_m)) ] \exp [ - \masss (\cI( \dot{\bs \eta}_{m+1,N})) ]\\
  & = Z(\lambda_m) h_{m+1}.
\end{align*}
Since $Z(\lambda_m) \asymp Z(\lambda_{m+1}) $ by part b., this is the required second inequality of c. In the other direction, it is easier to start by considering $h_{m+1} Z_m$ and write by
\begin{align*}
h_{m+1}Z_m & = \sum_{\bs \eta_N} \mu(\bs \eta_N) e^{- \masss (\cI(\dot{\bs \eta}_{m+1,N}))} \sum_{\bs \eta_m} \mu (\bs \eta_m) e^{-\masss (\cI_m( \bs \eta_m))}\\
& \lesssim \sum_{\bs \eta_N \in \mathtt{Sep}_{m, m+1}} \mu(\bs \eta_N) e^{- \masss (\cI(\dot{\bs \eta}_{m+1,N}))} \sum_{\bs \eta_m \in \mathtt{Sep}_m} \mu (\bs \eta_m) e^{-\masss (\cI_m( \bs \eta_m))} \text{ (by  Lemmas \ref{lem:separation_end} and \ref{lem:separation})} \\
& \le \sum_{\bs \eta_m \in \mathtt{Sep}_m } \sum_{\dot{\bs \eta}_{m+1,N} \in \dot{\mathtt{Sep}}_{m+1, N}  } \mu (\bs \eta_m)  \mu(\dot{\bs \eta}_{m+1,N}) %\P_\mu ( \mathtt{Sep}_{m, m +1} \mid \dot{\bs \eta}_{m+1,N})
e^{-\masss (\cI_m( \bs \eta_m))} e^{- \masss (\cI(\dot{\bs \eta}_{m+1,N}))}\\
& \lesssim \sum_{\bs \eta_N \in \mathtt{Sep}_{m, m+1}}\mu( \bs \eta_N) e^{- \masss (\cI(\bs \eta_N))} \text{ (by Lemmas \ref{L:QIloops} and \ref{L:QIpaths})}.
 \end{align*}
The latter is obviously smaller or equal to $Z(\lambda_N)$ as desired in first inequality of c.

Proof of d. and e. In view of c., it suffices to show that
$
h(\bs \eta_m) \le c_1 \tfrac{Z(\lambda_N)}{Z(\lambda_m)}
$ (resp. $\ge c_1 \tfrac{Z(\lambda_N)}{Z(\lambda_m)}$  if $\bs \eta_m \in \mathtt{Sep}_m$).
The proof is thus done similarly to c., and is a consequence of the quasi-independence \cref{L:QIloops,L:QIpaths}.
\end{proof}

\subsection{Coupling of pair of paths}

The arguments in these sections are essentially those of Lawler \cite{Lawler2sided} and do not require new input except for Lemma \ref{lem:couple_onestep}.
In this section, we will work with the Markov chain on the pair of non-intersecting self avoiding paths described by \eqref{Markovchain}. To avoid confusions it is useful to write the Markov chain as $(Y_1, \ldots, Y_N)$. Let us recall the transition probability of the chain
\begin{equation*}
p ( Y_m, Y_{m+1} ) = \frac{\lambda_{m+1} (Y_{m+1})}{\lambda_m (Y_m)} \frac{h( Y_{m+1})}{ h (Y_m)}.
\end{equation*}
The first step is to show that we can couple the Markov chains in one step with positive probability, uniformly over the initial condition.

\begin{lemma}[Separation] \label{cor:separation_chain}
There exists a constant $c >0$ such that for all $\delta>0$,  and $N_{\min } \le  n-1 \le N_{\max} -1$, and $\bs \gamma_n \in \cA_n$,
\begin{equation*}
\P( Y_{n+1} \in \mathtt{Sep}_{n+1} |  Y_n = \bs \gamma_n) \ge c
\end{equation*}
\end{lemma}
\begin{proof}
This is an immediate consequence of the Separation \cref{lem:separation} for the $\lambda$ term and \cref{cor:separation} parts d. and e. for the $h$ term.
\end{proof}

The consequence is that we can couple paths for a given number of scales (say $j$) so that they agree at the end of these $j$ scales with positive probability. For any two path pairs $\bs \gamma_n ,  {\bs  \gamma}'_n$ in $\cA_n $ for $ n \ge j$, we say $$\bs \gamma_n =_j   {\bs \gamma}'_n \quad \quad  \text{ if } \quad \quad \dot{\bs \gamma}_{n-j,n}  = \dot{\bs \gamma}'_{n-j,n}.$$ That is,
the portion of $\bs \gamma$ from the last hit of $2^{n-j}\delta$ up until first hit of $2^n \delta$ is the same as that of  $\tilde {\bs \gamma}$. The next lemma is the analogue of Lemma 2.34 in \cite{Lawler2sided}.

\begin{lemma}\label{lem:couple_onestep}
For all $j \ge 1$, there exists a constant $a_j >0$ (which may depend on $j$ but not on anything else) such that for any $\delta > 0$, for all $1  \le n \le N - j-1$, for any $\bs \gamma_n ,  \bs {\gamma}'_n \in \cA_n$, we can couple the Markov chains $Y$ and $Y'$ (with initial conditions $\bs \gamma_n$ and $\bs \gamma'_n$ respectively) such that
 $$\P( Y_{n+j} =_{j -2}  { Y}'_{n+j},  Y_{n+j} \in \mathtt{Sep}_{n+j}   \mid  Y_n =\bs \gamma_n,  Y'_n =  {\bs \gamma}'_n ) \ge a_j.$$
\end{lemma}

\begin{proof} Lawler's argument in \cite{Lawler2sided} relies on the existence of cutpoints among other estimates to deduce that walks that can be coupled also have agreeing loop-erased random walks (in general it is otherwise possible to find examples of two paths $\eta$ and $\eta'$ which agree for sufficiently large time but their eventual loop-erasures are disjoint!) We therefore need to find a ``softer'' argument since existence of cutpoints are not known in our context.
By the previous lemma (\cref{cor:separation_chain}), we see that by letting $Y_n$ and $Y'_n$ evolve independently over one step, both $Y_{n+1} \in \mathtt{Sep}_{n+1}$ and $Y'_{n+1} \in \mathtt{Sep}_{n+1}$ with positive probability, which we now assume. Ìn order to couple $(Y_{n+2}, \ldots, Y_{n+j})$ and $(Y'_{n+2}, \ldots, Y'_{n+j})$ we will in fact prove the stronger statement that we can couple $(Y_{n+2}, \ldots, Y_{N})$ with $(Y'_{n+2}, \ldots, Y'_{N})$ in such a way that the resulting paths agree with positive probability after their first visit to $B_{n+2}^c$ (given $Y_{n+1},Y'_{n+1} \in \mathtt{Sep}_{n+1})$. To do this we will actually look at the Radon--Nikodym derivative of one law with respect to the other using the explicit formula:
\begin{equation}\label{lawconditional}
{\lambda_N}(\dot {\bs \eta}_{n+2, N} ; \bs \eta_{n+1} = \bs \gamma_{n+1}) = \sum_{\bs \eta^*} \exp ( - \masss (\cI( \bs \gamma_{n+1} \oplus \bs \eta^* \oplus \dot{\bs \eta}_{n+2, N} ))) \mu (\bs \gamma_{n+1} \oplus \bs \eta^* \oplus \dot{\bs \eta}_{n+2, N})
\end{equation}
and using the quasi independence lemmas for both $\mu$ and the loop measures. This will give us up-to-constant upper and lower bounds in which the dependence on $\bs \gamma_n$ can be factored out as desired.
Consider first the upper bound which is a bit easier: by the upper bound in \cref{L:QIloops} and in \cref{L:QIpaths},
\begin{align*}
  {\lambda_N}(\dot {\bs \eta}_{n+2, N} ; \bs \eta_{n+1} = \bs \gamma_{n+1}) & \lesssim \exp \Big( - \masss \cI_{n+1} (\bs \gamma_{n+1}) - \masss \cI_N ( \dot {\bs \eta}_{n+2, N} )\Big)\mu (\bs \gamma_{n+1}) \mu(\dot{\bs \eta}_{n+2, N}).
\end{align*}
For the lower bound, it will be convenient to assume that the two strands of $\dot {\bs \eta}_{n+2, N} $ are in $\dot{\sep}_{n+2, N}$. We may do so without loss of generality since this event has positive `probability' both for the measure on the left hand side of \eqref{lawconditional} and the same measure with $\bs \gamma_n$ replaced by $\bs \gamma'_n$.
We deduce from \eqref{lawconditional} that we can restrict the sum to intermediate portions $\bs \eta^*$ that are well separated in the sense of \cref{L:QIloops}, to get: %\note{IS THE SUM BELOW IN THE WRONG PLACE?}
\begin{align*}
  {\lambda_N}(\dot {\bs \eta}_{n+2, N} ; \bs \eta_{n+1} = \bs \gamma_{n+1}) & \ge \sum_{\bs \eta^* \in \sep_{n+1, n+2}} \exp ( - \masss (\cI( \bs \gamma_{n+1} \oplus \bs \eta^* \oplus \dot{\bs \eta}_{n+2, N} ))) \mu (\bs \gamma_{n+1} \oplus \bs \eta^* \oplus \dot{\bs \eta}_{n+2, N})\\
  & \gtrsim  \exp \Big( - \masss \cI_{n+1} (\bs \gamma_{n+1}) - \masss \cI_N ( \dot {\bs \eta}_{n+2, N} )\Big) \sum_{ \bs \eta^* \in \sep_{n+1, n+2}}   \mu (\bs \gamma_{n+1} \oplus \bs \eta^* \oplus \dot{\bs \eta}_{n+2, N})\\
 % & \gtrsim  \exp \Big( - \masss \cI_{n+1} (\bs \gamma_{n+1}) - \masss \cI_N ( \dot {\bs \eta}_{n+2, N} )\Big) \sum_{ \bs \eta^* }   \mu (\bs \gamma_{n+1} \oplus \bs \eta^* \oplus \dot{\bs \eta}_{n+2, N})\\
  & \gtrsim \exp \Big( - \masss \cI_{n+1} (\bs \gamma_{n+1}) - \masss \cI_N ( \dot {\bs \eta}_{n+2, N} )\Big)\mu (\bs \gamma_{n+1}) \mu(\dot{\bs \eta}_{n+2, N}),
\end{align*}
where we used \eqref{eq:qirestrict} in the last line.
Putting these two bounds together, we obtain directly:
\begin{equation}\label{RNcontrol}
\frac{d \underline{\lambda}( \dot {\bs \eta}_{n+2, N}  \mid \bs \eta_{n+1} = \bs \gamma_{n+1}) }
{ d \underline{\lambda}( \dot {\bs \eta}_{n+2, N}  \mid \bs \eta_{n+1} = \bs \gamma'_{n+1})} \asymp 1.
\end{equation}
This would be enough to prove that the chains can be coupled so that the event that they agree and both are in $\mathtt{Sep}_{n+j}$ has positive probability $a_j$ depending only on $j$. {Indeed, if $X$ and $Y$ are  discrete random variables and there exists $c,C,A$ with $\P(X \in A) \ge c$, and $1/C \le \P(X=\omega)/\P(Y = \omega) \le C$ for all $\omega \in A$, then, using the optimal coupling for total variation (see e.g. \cite{MCMT}), \begin{equation*}\P(X=Y, X\in A) = \sum_{\omega \in A}  \min\{\P(X=\omega) , \P(Y=\omega)\} \ge \frac cC.\end{equation*}}
 \end{proof}

 %{\color{gray}Restricting \eqref{RNcontrol} to paths $ \dot {\bs \eta}_{n+2, N}  \in \mathtt{Sep}_{n+j}$, which carry a positive fraction of the mass under either law by \cref{cor:separation_chain}, we obtain the desired result. Indeed, for any $X,Y$, there exists a coupling between $X,Y$ such that $\P(X=Y, X\in A) = \sum_{\omega \in A} \min\{\P(X=\omega), \P(Y=\omega)\} \ge \frac cC$.}

We now state a lemma analogous to the first three parts of Lemma 2.33 in \cite{Lawler2sided}, saying that if two paths are identical for $j$ scales before a scale $n$, then the transition probabilities of the Markov chain to the next scale are the same up to an error which is exponentially small in $j$. (We also added a part a' for reasons that will be clear later).

\begin{lemma}\label{lem:lambda_b_compare}
There exists constants $c_1,c_2 >0$ such that for all $\delta>0$,  and $1\le n \le N$, the following is true.
\begin{enumerate}[{(}a{)}.]
\item $$\sum_{\bs \eta \in \cA_{n,n+1}:  \bs \eta \cap B_{n-j/2} \neq \emptyset} \lambda_{n+1}(\bs \gamma_{n} \oplus \bs \eta \mid \bs \gamma_{n} )  \le c_2e^{-c_1j} $$
where we recall that $ \lambda_{n+1}(\bs \gamma_{n} \oplus \bs \eta \mid \bs \gamma_{n} )   =  \frac{\lambda_{n+1}(\bs \gamma_{n} \oplus \bs \eta) }{  \lambda_{n}(\bs \gamma_{n} ) }$,
the sum over all $\bs \eta \in \cA_{n, n+1}$ means the sum over $\bs \gamma_n$ such that $\bs \gamma_n \oplus \bs \eta \in \cA_{n+1}$.

\item[(a')] $$\sum_{\bs \eta \in \cA_{n,N}:  \bs \eta \cap B_{n-j/2} \neq \emptyset} \lambda_{N}(\bs \gamma_{n} \oplus \bs \eta \mid \bs \gamma_{n} )  \le c_2e^{-c_1j}  h_n $$
%where we recall that $ \lambda_{n+1}(\bs \gamma_{n} \oplus \bs \eta \mid \bs \gamma_{n} )   =  \frac{\lambda_{n+1}(\bs \gamma_{n} \oplus \bs \eta) }{  \lambda_{n}(\bs \gamma_{n} ) }$,
%the sum over all $\bs \eta \in \cA_{n, n+1}$ means the paths such that $\bs \gamma_n \oplus \bs \eta \in \cA_{n+1}$.

\item Fix $\bs \gamma_n =_j \tilde {\bs \gamma}_n$ both in $\cA_n$. For $\bs \eta $ such that $\bs \gamma_n \oplus \bs \eta \in \cA_{n+1}  $ and $\bs \eta  \cap B_{n-j/2} = \emptyset $,
$$
\left | \frac{\lambda_{n+1} (\bs \gamma_n \oplus \bs \eta  \mid \bs\gamma_n ) }{ \lambda_{n+1} (\tilde{\bs \gamma}_n \oplus \bs \eta \mid \tilde{\bs\gamma}_n ) } -1 \right| \le c_2 e^{-c_1j}
$$
\item Fix $\bs \gamma_n =_j \tilde {\bs \gamma}_n$ both in $\cA_n$. Then
$$
| h(\bs \gamma_n ) - h(\tilde {\bs \gamma}_n)| \le c_2e^{-c_1j} h_n.
$$
\end{enumerate}
\end{lemma}

\begin{proof}
Proof of (a): We have
$$
\sum_{\bs \eta \in \cA_{n,n+1}:  \bs \eta \cap B_{n-j/2} \neq \emptyset} \lambda_{n+1}(\bs \gamma_{n} \oplus \bs \eta \mid \bs \gamma_{n} ) = \sum_{\bs \eta \in \cA_{n,n+1}:  \bs \eta \cap B_{n-j/2} \neq \emptyset} \mu (\bs \gamma_n \oplus \bs \eta ) e^{- (\masss \cI_{n+1}( \bs \gamma_n \oplus \eta) - \masss \cI_n (\bs \gamma_n) )}.
$$
The $\exp ( - \masss)$ term is trivially bounded by $1$, and by \cref{lem:backtrack_excursion}, we see that under $\mu$ the probability of a large backtrack is exponentially small.

Proof of (a') : By the quasi-independence lemma for loops \cref{L:QIloops} we can bound
$$
\sum_{\bs{\eta} \in \cA_{n,N}:  \bs \eta \cap B_{n-j/2} \neq \emptyset} \lambda_{N}(\bs{\gamma}_{n} \oplus \bs \eta \mid \bs{\gamma}_{n} ) \lesssim \sum_{\bs \eta \in \cA_{n,N}:  \bs \eta \cap B_{n-j/2} \neq \emptyset} \mu(\bs \gamma_{n} \oplus \bs \eta \mid \bs{\gamma}_{n} ) e^{-\masss \cI_{N} ( \dot {\bs \eta}_{n+1, N})}.
$$
By the quasi-independence lemma for paths \cref{L:QIpaths} (and more specifically by \eqref{eq:qipaths}), we see that $\mu(\bs{\gamma}_n \oplus \bs{\eta} \mid \bs{\gamma}_n ) \asymp \mu( \dot {\bs \eta}_{n+1, N}) \mu(\bs{\eta}^* \mid \bs{\gamma}_n, \dot{\bs\eta}_{n+1, N} )$. Using again \cref{lem:backtrack_excursion}, is easy to see that we can bound the conditional probability that $\bs \eta^*$ contains a backtrack, uniformly over $\gamma_n$ and $\dot{\bs\eta}_{n+1, N}$. Comparing $h_n$ and $h_{n+1}$ by \cref{cor:separation} concludes.

Proof of (b): We have
\[
\frac{\lambda_{n+1} (\bs \gamma_n \oplus \bs \eta  \mid \bs\gamma_n ) }{ \lambda_{n+1} (\tilde{\bs \gamma}_n \oplus \bs \eta \mid \tilde{\bs\gamma}_n ) } = \frac{\mu (\bs \gamma_n \oplus \bs \eta  \mid \bs\gamma_n ) }{ \mu (\tilde{\bs \gamma}_n \oplus \bs \eta \mid \tilde{\bs\gamma}_n ) } \frac{e^{-\masss \cI_{n+1} (\bs \gamma_n \oplus \bs \eta) + \masss \cI_{n}( \bs \gamma_n)}}{e^{-\masss \cI_{n+1} (\tilde{\bs \gamma}_n \oplus \bs \eta) + \masss \cI_{n}( \tilde{\bs \gamma}_n)}}.
\]
For the ratio of loop measures, note that for the contribution of one loop not to cancel out, it must simultaneously enter inside $B_{n-j}$ (to see the difference between $\bs \gamma$ and $\tilde{\bs\gamma}$) and exit from $B_{n-j/2}$ (to see $\eta$ or the change from $\cI_n$ to $\cI_{n+1}$). By \cref{lem:loop_crossing}, the mass of such loops is $O(e^{-cj})$. For the ratio of $\mu$, we can write it as ratios of escape probabilities from the endpoints of $\bs\gamma_n$ and $\bs\gamma_n \oplus \bs\eta$ together with a loop measure term as in the proof of \cref{lem:quasi_independence}. The escape probabilities are exponentially close in $j$ to one another by \cref{lem:backtrack_excursion} and the loop measure term is controlled exactly as above.

Proof of (c). First we note that in point (b), we can also consider a path $\eta$ in $\cA_{n, N}$ with the same proof. Putting together this version of point (b) and point (a') gives
\[
|h(\bs\gamma_n) - h(\tilde{\bs\gamma}_n) | \leq 2 c_2 e^{-c_1 j} h_n + c_2 e^{-c_1 j} h( \bs \gamma_n).
\]
We conclude by point (e) of \cref{cor:separation}.
\end{proof}

As a corollary, we see that the Markov chain is unlikely to backtrack many scales.
\begin{corollary}\label{lem:transition_small_deep}
There exists constants $c_1,c_2 >0$ such that for all $\delta>0$,  and $1\le n\le N -1$, the following is true.
If $\bs \gamma_n \in \cA_n$, with $h (\bs \gamma_n) \ge e^{-c_1j/2}  h_n $, then
$$
\P(( Y_{n+1} \setminus \bs \gamma_{n} ) \cap B_{n-j/2} \neq \emptyset |  Y_n = \bs \gamma_n) \le c_2e^{-c_1j/2}.
$$
\end{corollary}
\begin{proof}
We apply part (a) of \cref{lem:lambda_b_compare} and \cref{cor:separation}, items (b), (c) and (e). %\note{Shouldn't this be part $(a)$ and $c_1$ be exactly the one taken in part (a) ?}
\end{proof}

As a consequence, when $\bs \gamma_n =_j \bs \gamma'_n$, then $h(\bs \gamma_n)$ and $h (\bs \gamma'_n)$ are very close and the corresponding transition probabilities of the Markov chain are also very close:

\begin{corollary}\label{cor:lambda_b_compare}
There exists constants $c_1,c_2 >0$ such that for all $\delta>0$, and $1 \le n\le N-1$, the following is true.
If $\bs \gamma_n =_j \tilde {\bs \gamma}_n$, both being in $\cA_n$, and $h (\tilde{\bs \gamma}_n) \ge e^{-c_1j/2}  h_n$ then
\begin{equation*}
\left| \frac{h(\bs \gamma_n)}{h(\tilde {\bs \gamma}_n) } -1\right| \le c_2 e^{-c_1j/2}.
\end{equation*}
Moreover, for any $\bs \eta$ such that: $\bs \gamma_{n+1 } = \bs \gamma_n \oplus \bs \eta \in \cA_{n+1}$, $h (\bs \gamma_{n+1}) \ge e^{-c_1(j+1)/2}  h_{n+1}$, and $\bs \eta \cap B_{n-j/2} = \emptyset$,
\begin{equation}
\left | \frac{\P( Y_{n+1} = \bs \gamma_{n} \oplus \bs \eta |  Y_n = \bs \gamma_n) }{\P( Y_{n+1} = \tilde {\bs \gamma}_n \oplus \bs \eta |  Y_n = \tilde {\bs \gamma}_n)} -1\right| \le c_2e^{-c_1j/2}.
\end{equation}
\end{corollary}

\begin{proof}
The first assertion follows from part c. of \cref{lem:lambda_b_compare} for the first point. The second point follows from part b. of \cref{lem:lambda_b_compare}. together with the first assertion.
\end{proof}
%\note{$c_1$ has to be chosen carefully as in \cref{lem:lambda_b_compare}. }

To iterate this lemma, it is useful to show that the assumptions of the above corollary are inductively satisfied:

\begin{lemma}\label{lem:couple_manystep}
There exists constants $b_1, b_2>0$ such that for any $\delta>0$, for any with $ 1\le n \le N-1$, any $\bs \gamma_n ,  \bs { \gamma}'_n \in \cA_{n} $ with $\bs \gamma_n =_{j}    \bs {\gamma}'_n $ and $h(\bs \gamma_n') \ge e^{-b_1 j }h_n$, we can couple $Y_{n+1}$ and $Y'_{n+1}$ such that
$$
\P(\cG | Y_{n} =\bs \gamma_n,    Y'_{n} = {\bs \gamma}'_n ) \ge 1 - b_2 e^{-b_1 j}.
$$
where $\cG$ is the event that
\begin{itemize}
\item $ Y_{n+1} =_{j+1} {Y}'_{n+1} $,
\item $(Y_{n+1}  \setminus Y_{n} ) \cap B_{n-j/2} = \emptyset$.
\item $h( Y_{n+1}) \ge e^{-b_1(j+1)} h_{n+1} $.
\end{itemize}
\end{lemma}
\begin{proof}

Let us actually start by showing that the last item holds with overwhelming probability. (This does not in fact require $h(\bs \gamma_n) \ge e^{-b_1 j }h_n$.) Note that
\begin{align*}
\P ( h (Y_{n+1} ) \le e^{ - b_1 (j+1)} h_{n+1} \mid Y_n = \bs \gamma_n ) & \le \sum_{y_{n+1}} \lambda_{n+1} ( y_{n+1} \mid  \bs \gamma_n) \frac{ h (y_{n+1} )}{h (\bs \gamma_n)}1_{h(y_{n+1} ) \le e^{ - b_1 (j+1) } h_{n+1}}\\
& \le  e^{ - b_1 (j+1) } h_{n+1}  \sum_{y_{n+1}} \lambda_{n+1} ( y_{n+1} \mid  \bs \gamma_n) \frac{ 1}{h (\bs \gamma_n)}.
\end{align*}
On the other hand, the sum in the right hand side above is of order 1. Indeed,
\begin{align*}
  1 & = \sum_{y_{n+1}} \lambda_{n+1} ( y_{n+1} \mid  \bs \gamma_n) \frac{ h (y_{n+1} )}{h (\bs \gamma_n)} \\
  & \gtrsim \sum_{y_{n+1} \in \mathtt{Sep}_{n+1}} \lambda_{n+1} ( y_{n+1} \mid  \bs \gamma_n) \frac{ h_{n+1} }{h (\bs \gamma_n)}\\
  & \gtrsim \sum_{y_{n+1} } \lambda_{n+1} ( y_{n+1} \mid  \bs \gamma_n) \frac{ h_{n+1} }{h (\bs \gamma_n)}
 \end{align*}
 where the first line follows by restricting to paths $y_{n+1} \in \mathtt{Sep}_{n+1}$ and point d. in \cref{cor:separation}, and the last line follows by the Separation \cref{lem:separation}. Thus
 $$
 \P ( h (Y_{n+1} ) \le e^{ - b_1 (j+1)} h_{n+1} \mid Y_n = \bs \gamma_n ) \le e^{- b_1 (j+1)}.
 $$
 This concludes the proof of the corollary. {For the coupling step, we use the same argument as in the last line of \cref{lem:couple_onestep}.}
 
 Now the first point follows from \cref{cor:lambda_b_compare} and the last item we just proved. The second point follows from  \cref{lem:transition_small_deep}  and the last item.
\end{proof}

\begin{proof}[Proof of \cref{prop:local_convergence}] To finish the proof, we will use literally the same argument as in \cite{Lawler2sided} (the analogous result in \cite{Lawler2sided} is Proposition 2.32, which is proved at the end of Section 2). We include it for completeness. We try to couple in some given number $j= j_0$ of steps using Lemma \ref{lem:couple_onestep}, where $j_0$ is fixed and chosen so that agreeing on $j_0$ scales makes the probability of staying coupled forever after greater than 1/2. If at any point the coupling fails, we simply start again.
More precisely, suppose we start with two initial conditions $Y_m = \bs \gamma_m$, $ { Y'_m} = {\bs \gamma}'_m $
 for any two arbitrary $\bs \gamma_m, {\bs \gamma}'_m \in \cA_m$. Also fix $b_1,b_2$ as in \cref{lem:couple_manystep} and fix $j = j_0$ large enough so that
 $$\sum_{j \ge j_0-2}b_2e^{-b_1j} \le 1/2.$$

 Now we define a random sequence of integer-valued random variables $\{Z_k\}_{k \ge 0}$ with $Z_0 = 0$, and $Z_k$ will be a lower bound on the number of scales such that $Y_{m+kj_0}$ agrees $Y'_{m+k j_0}$.
 The precise recursive definition is as follows. Define $Z_k$ in such a way that $Z_k$ is measurable with respect to $Y_{m+ kj_0}$, and:
 \begin{itemize}
\item  If $Z_k = 0$, we try to couple as in \cref{lem:couple_onestep} over the next $j_0$ scales. If the coupling succeeds (which has probability at least $a_{j_0}$, which we will simply write as $a_0$), we take $Z_{k+1}=j_0 -2$. Otherwise, set $Z_{k+1} = 0$ and continue.
\item If $Z_k  = j >0$, then it will always be the case that $ Y_{m+j_0 k} =_{ j -2}   { Y}'_{m+j_0k}$. In this case, try to couple as in \cref{lem:couple_manystep} so that the event $\cG$ there occurs for $j_0$ many steps. If this coupling succeeds, we set $Z_{k+1} = j+j_0$. Otherwise, we set $Z_{k+1} = 0$. Notice that the coupling succeeds with probability at least $$ 1-\sum_{t = j }^{j+ j_0} b_2e^{-b_1t}
$$
via \cref{lem:couple_manystep}.
\item We stop if when $k$ reaches the value $k_{\max} := \max\{k: m+j_0 k  \le  N\}$.
  \end{itemize}
 Let $\tau   = \max \{k: Z_k = 0\}$ in the above chain. We now show that $\tau$ has an exponential tail uniformly in $\delta$. 
\begin{lemma}[\cite{Lawler2sided}]\label{lem:coupling_quick}
There exists a constant $u>0$ such that for any $\delta>0$, $\E(e^{u\tau}) <\infty$.
\end{lemma}
\begin{proof}
This is exactly Lemma 2.26 in Lawler \cite{Lawler2sided}; we provide the details for completeness. Let $\sigma_0 = 0 $ and $\sigma_{\ell} = \min \{j \ge \ell: Z_j = 0\}$ (with $\sigma_{\ell} = \infty$ if the minimum does not exist). Notice that by our choice of $j_0$, $\P(\sigma_1 = \infty) \ge a_0/2$. Consequently by the Markov property, $$\P(\sigma_\ell <\infty) \le (1-\frac{a_0}{2})^\ell.$$
 Also $\P(\sigma_1 = j | \sigma_1  < \infty) \le Ce^{-b_2j_0 j}$ for some constants $C$ depending only on $a_0, b_1, b_2$. Choosing $j_0$ larger if necessary, we can find a constant $u$ depending only on $j_0, a_0, b_1, b_2$ so that $$ \E(e^{uj_0 \sigma_1} | \sigma_1 <\infty) \le \left (1+\frac{a_0}{2}\right).   $$
 By the Markovian property, this means $\E(e^{uj_0 \sigma_\ell} | \sigma_\ell < \infty) \le \left (1+\frac{a_0}{2}\right) ^\ell.$ Consequently
 $$
 \E(e^{u j_0 \tau})  =\sum_{\ell = 1}^\infty \E(e^{uj_0 \sigma_\ell}1 _{\sigma_{\ell} <\infty} 1_{\sigma_{\ell+1} = \infty}) \le  \sum_{\ell = 1}^\infty \E(e^{uj_0 \sigma_\ell}1 _{\sigma_{\ell} <\infty}  ) \le   \sum_{\ell=1}^\infty (1-a_0^2/4)^\ell <\infty
 $$
 which completes the proof.
\end{proof}
It is clear that \cref{lem:coupling_quick} and the properties of the coupling completes the proof of \cref{prop:local_convergence}.
\end{proof}

We now explain how to use this to prove the desired scaling limit result (in the disc), namely \cref{prop:scalinglimit_disc}.

\begin{proof}[Proof of \cref{prop:scalinglimit_disc}]
Let $\bs \eta$ be the pair of paths obtained by performing Wilson's algorithm starting from the neighbouring points $x, x'$. Fix $\epsilon >0$. The idea is to describe the behaviour of the paths $\bs \eta$ conditioned on the event $\cA$ that they remain disjoint, outside the ball of radius $\epsilon$ around 0 (indeed, the rest of the paths is irrelevant in the Hausdorff sense).

Let $\bs \gamma$ be a pair of straight line segments with distinct angles emanating from $x, x'$ until they leave $B(0, \epsilon^2)$. Perform Wilson's algorithm (i.e., erase the loop in chronological order) to a pair of random walk excursions starting from the endpoints of $\bs \gamma$, conditioned not to return to $\bs \gamma$, and let $\tilde{ \bs \eta}$ be the concatenation of $\bs \gamma$ with these excursions.

By \cref{prop:local_convergence}, since $\epsilon^2 \ll \epsilon$, the behaviour of $\bs \eta$ conditioned on $\cA$ outside $B(0,\ve)$ is in fact very similar (in the sense that it can be coupled with probability greater than $1- \epsilon^\alpha$ for some $\alpha>0$) to that of $ \tilde{\bs \eta}$ conditioned on $\cA$.

Furthermore, for $\tilde {\bs \eta}$, it is clear that the scaling limit exists. Indeed, the event $\cA$ has positive probability for $ \tilde{ \bs \eta}$, and the scaling limit of each path is known to exist in the Hausdorff (\cite{LSW}). This concludes the proof of the existence of the scaling limit. The conformal invariance follows since our assumptions on the sequence of graphs are clearly still satisfied after conformal transformation.

To prove that the paths are almost surely disjoint first we note that the result of \cref{prop:local_convergence} still holds for the limit. Indeed if $\P_n \to \P$ and $\Q_n \to \Q$ weakly, then by the portmanteau theorem $d_{TV} ( \P, \Q ) \leq \liminf d_{TV} (\P_n, \Q_n)$. %Now for the choice $\tilde {\bs \eta}$ given above, it is clear that the limit paths are almost surely disjoint since they are
Now fix $\eps > 0$ and let us show that in the limit $\eta^1 \cap \eta^2 \cap B_\eps^c = \emptyset$ almost surely. Use the approximation $\bs{\tilde \eta}$ of $\bs \eta$ (where the paths are replaced by straight lines within $B(0, \epsilon^2)$) from the previous paragraph. Both paths in $\bs{\tilde \eta} $ are absolutely continuous with respect to SLE$_2$ and the law of $\tilde \eta^2$ given $\tilde \eta^1$ is simply SLE$_2$ in the slitted domain. Since SLE$_2$ does not touch the boundary the paths are disjoint. The total variation bound shows that the probability to see an intersection in $\bs \eta$ beyond distance $\eps$ is at most $\eps^\alpha$, which concludes the argument.
\end{proof}

\section{Local to global}\label{sec:localglobal}
Having established \cref{prop:scalinglimit_disc}, we extend the convergence result  on the whole manifold. The strategy is to compare the simply connected measure $\underline\lambda_N$ with the actual measure $\underline \lambda_{\cM}$ which involves the conditioning with the event $\cA_\cM$.

\subsection{Convergence of skeleton of Temperleyan forest}\label{sec:conv_special}
\label{S:local_global}
%%{\color{gray}Let $\Ptemp$ be the law of a Temperleyan CRSF pair $(t, t^\dagger)$. Recall that given $t$, $t^\dagger$ is obtained from $t$ by considering its planar dual and assigning an orientation on each of its cycles uniformly among all such possible orientations. Furthermore, from \cref{lem:RN_dimer_CRSF},
%$$
%d\Ptemp (t) = \frac{2^{k^{\dagger}}}{\Ztemp} 1_{t \text{ is Temperleyan}} d\Pwils (t).
%$$
%where $k^{\dagger}$ is the number of cycles of $t^\dagger$.
%%The factor $2^k$ in front is a small nuisance, and we first show that when we do not take it into account we will have a scaling limit for $t$
%%: that is, define $\tPtemp$ by setting
%%\begin{equation}
%%\frac{d\Ptemp}{d\tPtemp} = 2^{c} \frac{\tZtemp}{\Ztemp}
%\end{equation}
%%where $\tZtemp$ is the appropriate partition function.
%Notice then that $\Ptemp$ has the law of a CRSF (sampled from Wilson's algorithm), conditioned to satisfy the Temperleyan condition, up to a global factor in the Radon-Nikodym derivative corresponding to $2^c$. The latter reweighting is not difficult to handle, as is explained in \cref{thm:specialbranches_Temp}, and most of the difficulty comes from controlling the effect of the singular conditioning $\cA_{\cM}$.} 

Recall the law of a CRSF conditioned to be Temperleyan, which we denote by $\Pwils$ and was defined in \eqref{eq:pwils}. Recall also from \cref{def:temperleyan} that the conditioning event of a CRSF being Temperleyan is written as follows:
\begin{equation}
\cA_{\cM}=\cA_{\cM}(\delta)  = \text{The skeleton $ \mathfrak s$ divides the manifold ${\cM}$ into annuli,}\label{eq:Am}
\end{equation}
where we recall that $\mathfrak{s}$ is the union of $2 {\mathsf{k}}$ branches of $\cT$, $\mathfrak{B}_{i,1} \cup \mathfrak{B}_{i,2}$ emanating from $u_i$ and $v_i$ respectively (on either side of each of the ${\mathsf{k}}$ punctures $x_1, \ldots, x_{\mathsf{k}}$), together with the paths $e_1, \ldots, e_{\mathsf{k}}$.

The main goal of this section is to reduce the above conditioning by $\cA_{\cM}$ to the simpler conditioning in simply connected neighbourhoods studied in Section \ref{S:Lawler}. This will be done in \cref{P:localglobal} below. We will also need to show that events that
are unlikely for a single loop-erased random walk in the surface and depend only on its behaviour away from the puncture, remain unlikely after the conditioning by $\cA_{\cM}$. This will be done in \cref{prop:quantitative,prop:quantitative_general}.

We now explain more precisely the setup for the above mentioned reduction. We will only do the proof in the case ${\mathsf{k}}=1$ as the rest is similar but just involves heavier notation (we will briefly comment on the general case ${\mathsf{k}} > 1$ at the end of the proof of \cref{P:localglobal}). Let $N$ be a fixed simply connected neighbourhood in ${\cM}$ of the puncture $x_1$.
%
%Recall that we wish to replace the conditioning given by $\mathsf C$ to a simpler conditioning involving $k$ independent pairs of paths ${\bs \eta}^1, \ldots, {\bs \eta}^k$ in $N_1, \ldots, N_k$ respectively sampled from Wilson's algorithm {\color{red}conditioned to avoid each other}.
Let us introduce some notations. Analogously to the setting of Section \ref{S:Lawler}, let $\underline{\lambda}_N$ denote the law of a pair of wired UST branches in $N$ on either side of the path $e_1$ (from now on we call $e_1$ a puncture with a slight abuse of terminology), sampled with Wilson's algorithm, and conditional on the event $\cA_N $ that the resulting loop erasures of the two walks ${\bs \eta}$ give disjoint paths connecting either side of the puncture to $\partial N$. Let also $\mu_N$ denote the law of two \emph{independent} loop-erased walks, from either side of the puncture $e_1$, to $\partial N$.

We also call $\underline{\lambda}_{\cM}$ the law of a pair of CRSF branches in the surface ${\cM}$ sampled using Wilson's algorithm starting from either endpoint of the puncture $e_1$, conditioned on the event $\cA_{\cM}$ above. (Therefore, $\underline{\lambda}_{\cM}$ coincides with the law of $\mathfrak{s}$ under $\Pwils$ in the case ${\mathsf{k}}=1$).
Let $\mu_{\cM}$ denote the law of two \emph{independent} loop-erased random walks in ${\cM}$ starting from either side of the puncture $e_1$, and ran until they form non-contractible loops.

In Section \ref{S:Lawler} (more precisely in \cref{prop:scalinglimit_disc}) we proved that $\underline{\lambda}_N$ has a scaling limit. We will now show that the convergence of $\underline{\lambda}_{\cM}$ follows.

\begin{prop}
  \label{P:localglobal} The law of $\mathfrak{s}$ given $\cA_{\cM}$ converges as $\delta \to 0$ in distribution under $\Pwils$ in the Schramm topology. The limit does not depend on the sequence of graphs chosen (subject to the conditions of \cref{sec:setup}) and is conformally invariant. Furthermore, the limit is almost surely a set of disjoint simple curves (except at the punctures where they overlap) which separate the hyperbolic manifold ${\cM}$ into disjoint annuli.
\end{prop}

\begin{proof}
As mentioned earlier, we write the key bounds used in the proof only in the case ${\mathsf{k}}=1$ to ease notations, which works for the general case in exactly the same way. We mention the key differences for $k>1$ at the end of step 3.
%Let $\P$ denote the law of $\bs \eta^i$ (denoted without ambiguity by $\bs \eta$ since $k =1$), when we sample the two branches $\bs \eta$ of the UST in $N$ on either side of the puncture $e$ with Wilson' algorithm. Let $\Q$ denote the law of two branches in ${\cM}$ on either side of the puncture (called $x$ here rather than $e$) when we perform Wilson's algorithm for the CRSF on $\Gamma^\d$. \note{Are we deliberately not writing the conditioning statements to make the writing more elegant? Or is this really the unconditional probability? I am confused...P.S. I think you mean conditional after reading a bit more... }

\textbf{Step 1.} Let $\eps>0$. Our first step is to write down the Radon--Nikodym derivative $d\underline{\lambda}_{\cM}/ d\underline{\lambda}_{N}$ for the marginal distribution of $\bs \eta_r$, which is simply the paths $\bs \eta$ considered up to their first exit of $B(x, r) \subset N$, where $r>0$ is fixed but small (and will later be chosen suitably depending on $\eps$), and $x$ is the midpoint of $e_1$. Let $\P_{\bs x}$ denote the law of two independent random walks started from two vertices $x^1$ and $x^2$ (writing $\bs x = (x^1,x^2)$). In what follows we need to consider the behaviour of the loop erased walk on the surface created by random walks run until specific stopping times dictated by  Wilson's algorithm, either on a surface or on a simply connected domain containing $\bs x$. However, we will think of these loop erasures as just functionals of the infinite walks with law $\P_{\bs x}$.  Recall \eqref{eq:RN_start2} and \eqref{eq:start_surface3}, written slightly differently as follows :%. Call  $x^1, x^2$ the endpoints \red{starting points?} of $\eta_r^1$ and $\eta_r^2$ respectively. Then
\begin{align*}
\underline{\lambda}_{N} (\bs \eta_r)   &= \frac{q(\eta^1_r \cup \eta_r^2)}{\P(\cA_N)}  \exp[ \mass \Big((\cI( \eta^1_r)\cup \cI( \eta^2_r))\Big)] \P_{\bs x}(\cA_N; \tau_{N} < \tau_{\bs \eta_r})\\
\underline{\lambda}_{\cM} (\bs \eta_r) &\asymp\frac{q(\eta^1_r \cup \eta_r^2)}{\P(\cA_{\cM})}\exp[ \mass \Big((\cI( \eta^1_r)\cup \cI( \eta^2_r)) \cap \cC_{\bs \eta}\Big)]\P_{\bs x} (\cA_{\cM}; \text{ compatible})
\end{align*}
where we recall the notation $A \asymp B$ means that $C^{-1}A \le B \le C A$. Here and in later use throughout this proof, $\log (C)$ can be taken to be the mass of all non-contractible loops in $M$ that intersect $\bs \eta$. Thus we find that 

\begin{equation}\label{RNlocal}
\frac{d\underline{\lambda}_{\cM}}{d\underline{\lambda}_N} (\bs \eta_r) \asymp  \frac1{{\mathsf Z_{M,N}}}\underbrace{\frac{\exp[ \mass \Big((\cI( \eta^1_r)\cup \cI( \eta^2_r)) \cap \cC_{\bs \eta}\Big)]}{\exp[ \mass \Big((\cI_N( \eta^1_r)\cup \cI_N( \eta^2_r)) \Big)]} }_{L}\overbrace{\frac{\P_{\bs x} (\cA_{\cM}; \text{ compatible})}{\P_{\bs x} (\cA_N; \tau_{N} < \tau_{\bs \eta_r})}}^{R}.
\end{equation}
where 
\begin{align}
    L &:= \frac{\exp[ \mass \Big((\cI( \eta^1_r)\cup \cI( \eta^2_r)) \cap \cC_{\bs \eta}\Big)]}{\exp[ \mass \Big((\cI_N( \eta^1_r)\cup \cI_N( \eta^2_r)) \Big)} \label{eq:L}\\
    R& := \frac{\P_{\bs x} (\cA_{\cM}; \text{ compatible})}{\P_{\bs x} (\cA_N; \tau_{N} < \tau_{\bs \eta_r})}\label{eq:R}\\
   {\mathsf Z_{M,N}} & :=  \frac{\P(\cA_{\cM})}{\P(\cA_N)}.
\end{align}

% and the constants implicit in the $\asymp $ symbol  are bounded in terms of the mass of loops intersecting $\bs \eta$ with a non-contractible support. 

Also, recall that  the set $\cC_{\bs \eta}$ denotes the $\bs \eta$-contractible loops. Let us now recall the meaning of the quantity $R$. Run Wilson's algorithm using the walks generated by $\P_{\bs x}$, where the first branch is sampled from $x^1$ and the second from $x^2$ (whatever be the domain).
 We now define the pair of stopping times $\tau_N$ (resp. $\tau_{\bs \eta_r}$) as the hitting times of $\partial N$ (resp. $\bs \eta_r$) using the independent random walks generated by $\P_{\bs x}$. We write $\tau_N < \tau_{\bs \eta_r}$ to mean that the inequality holds for each coordinate.
The event $\cA_{\cM}$ is the event that when we apply Wilson's algorithm in $\cM_r:=\cM \setminus  (\bs \eta_r \cup \partial)$ to the walks, we create branches which separate the manifold into annuli. Similarly, if we run (the usual classical version of) Wilson's algorithm in $N_r:= N \setminus \bs \eta_r$, we define $\cA_N$ to be the event that the resulting branches are disjoint.  Also ``compatible'' has the same meaning as in \eqref{eq:start_surface3}: the loop erasures of the walks must produce paths which are compatible with being CRSF branches when we do Wilson's algorithm in $\cM_r $ (see \cref{fig:compatibility}). Note that for the event corresponding to $N$, the conditions in the event guarantee compatibility, so we do not mention it in the event.

%
% performed from the tips of $\bs \eta_r$, first from $\eta_1$ and then from $\eta_2$  with boundary $\partial \cup \bs \eta$. Also $\tau_{\bs \eta_r}$ is the \emph{pair} of hitting times by each of these two walks of $\eta_r^1 \cup \eta_r^2$ (note that in the definition of $\tau_{\bs \eta_r}$, \emph{both} parts of $\bs \eta_r$ are concerned in the stopping time of \emph{each} walk). The event $\tau_{N} < \tau_{\bs \eta_r}$ is the event that both the walks in Wilson's algorithm never touch $\bs \eta_r$ before leaving $N$. Here $\cA_N$ is interpreded as follows: if we consider the loop erasure $\gamma$ of the first walk until the first time it exits $N$, the second walk does not intersect $\gamma$ before exiting $N$. Let us emphasise that $\gamma$ could look very different from the actual loop erasure of the complete first walk. Also ``compatible'' has the same meaning as in \eqref{eq:start_surface3}: the loop erasures of the walks must produce paths which are compatible to being CRSF branches.
\medskip
%\note{Unfortunately, I do not think this is true. For example, in the pair of pants case, I think one path can create a non-contractible loop around one hole, while the other can go around a hole and hit the other path. But perhaps the precise description is not so important. What is true I think is the following ``two independent walks starting from $x^1$ and $x^2$ never touch either path in $\bs \eta_r$, stop when they create a non-contractible loop, or hit the other path  in such a way that they divide the pair of pants into annuli. (I don't think they can hit the boundary)}.

\textbf{Step 2.}  We choose $r$ small enough so that the implicit constant in \eqref{RNlocal}, and the term denoted $L$ in \eqref{RNlocal} are both close to 1 (see \cref{prop:marginalM,L:fewcontractible}). More precisely, observe that $L \ge 1$ because the loops counted in the denominator are also counted in the numerator (since $N$ is simply connected). Let $d = \diam (N)$ and suppose without loss of generality that $N$ is a ball centered at $x$. Then note further that the extra loops in the numerator compared to the denominator all have diameter greater than $d/4$ if $r< d/4$ and these loops also intersect the ball $B(x,r)$. Similarly there is a lower bound depending only on the manifold $\cM$ on the diameter of any loop whose support is non contractible.
Therefore, by choosing $r$ small enough (depending only on $\eps$, $\cM$ and the diameter of $N$), \cref{L:fewcontractible} shows that  only a small mass of such loops intersect $\eta_r$ and overall we get% we get that $L \le 1+\eps$ after applying \cref{L:fewcontractible}. Hence with this choice of $r$, we have $1\le L \le 1+ \eps$.
\begin{equation}
\frac{(1-\ve)}{{\mathsf Z_{M,N}}} \frac{\P_{\bs x} (\cA_{\cM}; \text{compatible})}{\P_{\bs x} (\cA_N; \tau_{N} < \tau_{\bs \eta_r})}  \leq \frac{d\underline{\lambda}_{\cM}}{d\underline{\lambda}_N} (\bs \eta_r) \leq
\frac{(1+\eps)}{{\mathsf Z_{M,N}}} \frac{\P_{\bs x} (\cA_{\cM}; \text{compatible})}{\P_{\bs x} (\cA_N; \tau_{N} < \tau_{\bs \eta_r})}.\label{eq:RN_approx_L}
\end{equation}
\begin{figure}[h]
\centering
\includegraphics[width = 0.5\textwidth]{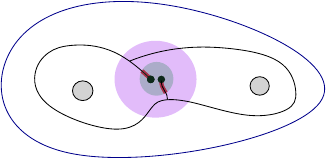}
\caption{The purple disc is $N$ and the blue disc is $B_r$. The red portions are $\bs \eta_r$. It is possible for the path started at $\bs \eta_r$ to stop even before leaving $N$.}\label{fig:merge_close}
\end{figure}

\textbf{Step 3.} We now concentrate on the term $R = R(\bs \eta_r)$ in \eqref{RNlocal}. We first obtain a lower bound on $R$ on a good event (and hence also a lower bound on ${\mathsf Z_{M,N}}$) by saying that one way of achieving $\cA_{\cM}$ (and also compatibility) is to first achieve $\tau_N < \tau_{\bs \eta_r}$, and arguing that given this, there is a good chance to also achieve $\cA_\cM$ (under $\ul_N$). (We remark however in passing that it is possible to achieve $\cA_{\cM}$ without achieving $\cA_N$ with positive probability in the limit -- see \cref{fig:merge_close}, so this lower bound is off by a constant factor. In fact, $\cA_{\cM}$ does not imply $\tau_N < \tau_{\bs \eta_r}$.)

We can rewrite $R$ as a ratio of conditional probabilities:
\begin{align}
R &\ge \frac{\P_{\bs x} (\cA_{\cM};\text{compatible} ; \tau_N < \tau_{\bs \eta_r})}{\P_{\bs x} (\cA_N ;  \tau_{N} < \tau_{\bs \eta_r})} \nonumber \\
& =
\frac{\P_{\bs x} (\cA_{\cM}; \text{compatible} | \tau_N < \tau_{\bs \eta_r})}{\P_{\bs x} (\cA_N |  \tau_{N} < \tau_{\bs \eta_r})} \label{Rprime}
\end{align}
In order to get a lower bound on $R$ we simply lower bound the numerator, and observe that this conditional probability is clearly at least constant depending only $r$ (and so only  on $\eps$) when $\bs \eta_r$ is separated. Also $\bs \eta_r$ is separated with positive probability under $\ul_N$ by results in Section \ref{S:Lawler}, see in particular \cref{cor:separation_chain}, which allows us to deduce
\begin{equation}\label{Zlb}
{\mathsf Z_{M,N}} \ge C_{\eps}^{-1},
\end{equation}
for some constant $C_\eps$ depending only on $r$ (and so only on $\eps$) but not on $\delta$.

We now prove a complementary \emph{a priori} uniform upper bound on $R$. A consequence of this upper bound is that ${\mathsf Z_{M,N}}$ is bounded.  We will rewrite both probabilities in $R$, as written in \eqref{RNlocal}, using Wilson's algorithm in the surface ${\cM}_r: = {\cM}\setminus \bs \eta_r$, so as to rephrase them in terms of independent loop-erasures of simple random walk, weighted by a loop soup term. This needs a bit of work. First recall from \eqref{eq:start_surface4}, we have
\begin{align}
  \P_{\bs x} ( \cA_{\cM}; \text{ compatible } ) & \asymp \sum_{\bs \eta = (\eta_1, \eta_2)}  e^{\mass_{{\cM}_r} ( \cI (\cup_i \eta_i) \cap \cC_{\bs \eta})}q(\cup_i \eta_i)1_{\{\bs \eta_r \oplus \bs \eta \in \cA_{\cM} \}},\label{num0}
\end{align}
\begin{figure}[h]
\centering
\includegraphics[width = 0.5\textwidth]{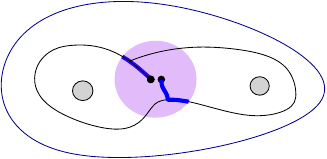}
\caption{The purple disc is the disc of radius $N/2$ and the blue paths are $\bs \gamma_{N/2}$.}\label{fig:merge_close2}
\end{figure}
We write $N/2$ with an abuse of notation for the ball with half the radius of $N$ and same center.
 Here it is important to think about the paths $\eta_i$ in order as they may coalesce to create complicated topologies. Let  $\eta_1$ be the branch started from $\eta_r^{1}$ ending when it hits $\bs \eta_r \cup \partial $ or creates a non-contractible loop, and $\eta_2$ is the path sampled from the tip of $\eta_r^2$ until it hits $\bs \eta_r \cup \eta_1 \cup \partial$, or creates a non-contractible loop. Note that although $\eta_2$ may intersect the whole path $\eta_1$ very close to the puncture, nevertheless it is the case that if a pair of path is in $\cA_\cM$, then the portions of the branches started from the endpoints of puncture  up to the first exits from $N/2$ must be disjoint, since $N/2$ is simply connected (we emphasise that the whole second branch might even contain portions of the $\eta_1$, see \cref{fig:merge_close2}). Call these paths $\bs \gamma_{N/2} = (\gamma_{N/2}^1, \gamma_{N/2}^2)$. In fact they are even disjoint until they exit $N$ which will be used later. We now divide the collection of paths $\{\bs \eta: \bs \eta_r \oplus \bs \eta \in \cA_{\cM}\}$ into two categories. In the first category, call it $\cP_1$, the first intersection of $ \eta_r^2 \oplus \eta_2$ with $\eta^1_r \oplus \eta_1$ is inside $N/2$. The second category is the complement of $\cP_1$ so that $\cA_\cM$ is still satisfied, call it $\cP_2$. We treat sum over each term of the expression in \eqref{num0} separately and upper bound them.

Roughly, the idea we employ is as follows. To employ our upper bound on random walks, we need the paths to be disjoint in a certain macroscopic region, and we need to ensure the topological constraints are such that this is forced. For $\cP_1$, we take the macroscopic region to be the annulus bounded between $N/2$ and $N$, and for $\cP_2$, we take it to be that bounded between $N/4$ and $N/2$. This culminates in the final bound \eqref{num1}.

First consider $\cP_1$.
Let $\cA_{N/2}$ the set of disjoint pairs of paths started from the punctures up to $\partial (N/2)$. We now decompose the expression \eqref{num0} on $\cP_1$ as follows:
 \begin{equation}
  \sum_{\bs \eta_r \oplus \bs \gamma_{N/2} \in \cA_{N/2}} q(\gamma_{N/2}^1) q(\gamma_{N/2}^2) e^{\mass_{{\cM}_r} \Big( \cI (\cup_i \gamma_{N/2}^i) \cap \cC_{\bs \gamma_{N/2}}\Big)} h_{N/2 \to \infty}(\bs \gamma_{N/2})\label{eq:coalesce2}
 \end{equation}
 where
\begin{align}
 h_{N/2 \to \infty}(\bs \gamma_{N/2}) = &  \sum_{\bs \eta \sim\bs \gamma_{N/2} , \bs \eta \in \cP_1}  e^{\mass_{{\cM}_r} ( \cI (\cup_i \eta_i) \cap \cC_{\bs \eta}) - \mass_{{\cM}_r} \Big( \cI (\cup_i \gamma_{N/2}^i) \cap \cC_{\bs \gamma_{N/2}}\Big) }q(\cup_i \eta_i \setminus \cup_i \gamma_{N/2}^i)1_{\{\bs \eta_r \oplus \bs \eta \in \cA_{\cM} \}}\nonumber\\
  \asymp& \sum_{\bs \eta \sim\bs \gamma_{N/2} , \bs \eta \in \cP_1}  e^{\mass_{{\cM}_r \setminus \bs \gamma_{N/2}} ( \cI (\cup_i \eta_i) \cap \cC_{\bs \eta})  }q(\cup_i \eta_i \setminus \cup_i \gamma_{N/2}^i)1_{\{\bs \eta_r \oplus \bs \eta \in \cA_{\cM} \}}\label{eq:coalesce1}\\
  \asymp &\  \P_{\bs \gamma_{N/2}}(\cA_\cM, \cP_1, \text{ compatible paths started from tips of $\bs \gamma_{N/2}$ })\label{eq:Coalesce}
 \end{align}
 and the sum in \eqref{eq:coalesce1} is over all $\bs \eta \in \cP_1$ which is compatible with having $\bs \gamma_{N/2}$ as the initial portion of the path. Exactly as in \eqref{num0}, the expression is another way of writing the probability of the event of obtaining branches in $\cP_1$, but starting from the tips of $\gamma_{N/2}^i$ in order. Here by a slight abuse of notation, we denote by $\P_{\bs \gamma_{N/2}}$, the law of independent random walks started from the tips of $\bs \gamma_{N/2}$. Now notice that $\eta_r^1 \oplus \eta_1$ coalesces with $\eta_2$ inside $N/2$. Therefore, the paths in $\cup_i \eta_i \setminus \cup_i \gamma_{N/2}^i$ started from the tips of $\gamma_{N/2}^i$ exits $N$ without intersecting (since they are both parts of  $\eta_1$).  Thus we can upper bound the event in \eqref{eq:Coalesce} by the event that two independent random walks from the tips of $\gamma_{N/2}^i$ exits $N$ without hitting $\gamma_{N/2}^i$.
 Call this last event $\cD_{N/2,N}$. %\note{UNFINISHED..missing: how to go from Wilson's algorithm to independent paths? In this case it is easy as the paths have to be disjoint I think..}

Thus overall, we get
 \begin{equation}
 \P_{\bs x} ( \cA_{\cM}; \text{ compatible }, \cP_1 ) \le  \sum_{\bs \eta_r \oplus \bs \gamma_{N/2} \in \cA_{N/2}} q(\gamma_{N/2}^1) q(\gamma_{N/2}^2) e^{\mass_{{\cM}_r} \Big( \cI (\cup_i \gamma_{N/2}^i) \cap \cC_{\bs \gamma_{N/2}}\Big)} \P_{\bs \gamma_{N/2}}(\cD_{N/2,N}). \label{eq:coalesce5}
 \end{equation}

We employ a similar tactic to upper bound $\cP_2$, where $\eta_r^2 \cup \eta_2$ does not hit $\eta_1$ inside $N/2$. Here we cut off the paths at $N/4$ to allow us some room. Define $\bs \gamma_{N/4} = (\gamma_{N/4}^1, \gamma_{N/4}^2)$ in an analogous way, and use a similar decomposition

  \begin{equation}
  \sum_{\bs \eta_r \oplus \bs \gamma_{N/4} \in \cA_{N/4}} q(\gamma_{N/4}^1) q(\gamma_{N/4}^2) e^{\mass_{{\cM}_r} \Big( \cI (\cup_i \gamma_{N/4}^i) \cap \cC_{\bs \gamma_{N/4}}\Big)} h_{N/4 \to \infty}(\bs \gamma_{N/4})\label{eq:coalesce3}
 \end{equation}
 with an obvious analogous expression for $h_{N/4 \to \infty}(\bs \gamma_{N/4})$.
 Notice here by definition that the paths  $\gamma_{N/4}^i$ exits $N/2$ without intersecting. Thus we can similarly upper bound \eqref{eq:coalesce3} on $\cP_2$. Overall we get
\begin{equation}
 \P_{\bs x} ( \cA_{\cM}; \text{ compatible }, \cP_1 ) \le \sum_{\bs \eta_r \oplus \bs \gamma_{N/4} \in \cA_{N/4}} q(\gamma_{N/4}^1) q(\gamma_{N/4}^2) e^{\mass_{{\cM}_r} \Big( \cI (\cup_i \gamma_{N/4}^i) \cap \cC_{\bs \gamma_{N/4}}\Big)} \P_{\bs \gamma_{N/4}}(\cD_{N/4,N/2}).\label{eq:coalesce4}
\end{equation}

%Now let $\mu_{\cM_r} $ denote the law of two \emph{independent} loop erased random walks started from the tips of $\bs \eta_r$ stopped when they hit $\bs \eta_r \cup \partial$ or create a non-contractible loop (we emphasise that the paths are independent and not stopped according to Wilson's algorithm). Since after reaching outside $N$ (or $N/2$), the random walk has a positive probability to not come back and intersect $N/2$ (or $N/4$), we see that
%\begin{equation}
%\P(\cD_{N/2,N})  \lesssim \mu_{\cM_r} (\cD_{N/2,N}) ; \qquad \qquad \P(\cD_{N/4,N/2})  \lesssim \mu_{\cM_r} (\cD_{N/4,N/2})
%\end{equation}
Let $\mu_{N_r} $ denote the law of two \emph{independent} loop erased random walks started from the tips of $\bs \eta_r$ stopped when they hit $\bs \eta_r$ or exit $N$. Now simply using \cref{prop:marginal_single}, we can upper bound both \eqref{eq:coalesce5} and \eqref{eq:coalesce4} by
\begin{equation}
\P_{\bs x} ( \cA_{\cM}; \text{ compatible }) \lesssim \sum_{\bs \eta_r \oplus \bs \eta_{N/4} \in \cA_{N/4}} e^{-\mass_{N_r} (\cI (\bs \eta_{N/4}))}\mu_{N_r} (\bs \eta_{N/4}) \label{num1}
\end{equation}
where the constant in the inequality comes from the fact that we are ignoring loops which intersect both $N/4$ and the complement of $N/2$, which has bounded mass. Notice above that we switched to the original dummy variable involving $\bs \eta$ rather than $\bs \gamma$ as we do not need to distinguish them any more.

Having upper bounded the numerator of $R$ in \eqref{RNlocal},  we now lower bound the denominator. This term is actually simpler to handle as there is no topological complexity here, we can directly write it as
\begin{equation}
  \label{den1}
  \P_{\bs x} ( \cA_{N}; \tau_{N} < \tau_{\bs \eta_r} )  = \sum_{\bs \eta_{N/4}} e^{ - \mass_{N_r} (\cI_{N} (\bs \eta_{N/4} ))} \mu_{N_r} (\bs \eta_{N/4}) h_{N/4 \to N} (\bs \eta_{N/4}) 1_{\{ \bs \eta_r \oplus \bs \eta_{N/4} \in \cA_{N/4}   \}}
\end{equation}
where $h_{N/4 \to N}$ is an $h$-transform term, which explicitly may be written as
\begin{align*}
h_{N/4 \to N} (\bs \eta_{N/4}) &= \sum_{\bs \eta_N \succeq \bs \eta_{N/4}} e^{-\mass_{N_r} (\cI_{N} (\bs \eta_{N})  \setminus \cI_N (\bs \eta_{N/4}))} \mu_{N_r} (\bs \eta_N \mid \bs \eta_{N/4}) 1_{\{ \bs \eta_r \oplus \bs \eta_{N} \in \cA_{N}   \}}\\
& = \E_{\mu_{N_r}}\left(e^{-\mass_{N_r} (\cI_{N} (\bs \eta_{N})  \setminus \cI_N (\bs \eta_{N/4}))} 1_{\{ \bs \eta_r \oplus \bs \eta_{N} \in \cA_{N}   \}}| \bs \eta_{N/4} \right).
\end{align*}
On the event that $\bs \eta_{N/4} \in \mathtt{Sep}_{N/4}$, this $h$-transform term is uniformly bounded below: indeed, when $\eta_{N/4} \in \cA_{N/4}$ and is separated, the event $\cA_N$ has positive probability, and by uniform crossing the paths remain separated with positive probability so that the mass of loops intersecting both paths is then bounded.
Therefore, restricting \eqref{den1}  to separated paths $\bs \eta_{N/4}$,
%\note{we need to refer to the separation lemma for s.c. case, do you think it is a good idea to push this part to the end?}
\begin{align}
  \P_{\bs x} ( \cA_{N}; \tau_{N} < \tau_{\bs \eta_r} ) & \gtrsim \sum_{\bs \eta_{N/4} \in \mathtt{Sep}_{N/4}} e^{ - \mass_{N_r} (\cI_{N} (\bs \eta_{N/4} ))} \mu_{N_r} (\bs \eta_{N/4})  1_{\{ \bs \eta_r \oplus \bs \eta_{N/4} \in \cA_{N/4}   \}}\nonumber \\
  & \gtrsim \sum_{\bs \eta_{N/4} \in  \mathtt{Sep}_{N/4} } e^{ - \mass_{N_r} (\cI_{N/4} (\bs \eta_{N/4} ))} \mu_{N_r} (\bs \eta_{N/4})  1_{\{ \bs \eta_r \oplus \bs \eta_{N/4} \in \cA_{N/4}   \}} \label{den3}\\
  & \gtrsim  \sum_{\bs \eta_{N/4} } e^{ - \mass_{N_r} (\cI_{N/4} (\bs \eta_{N/4} ))} \mu_{N_r} (\bs \eta_{N/4})  1_{\{ \bs \eta_r \oplus \bs \eta_{N/4} \in \cA_{N/4}   \}} \label{den2}
\end{align}
where \eqref{den3} follows from the fact that the extra mass of loops are uniformly bounded as the paths are separated. Then \eqref{den2} follows from    \cref{lem:separation} with $n = N/4$ and $\bs \gamma_m = \bs \eta_r$.
Comparing with \eqref{num1} and a trivial bound that the loop mass is positive we get that $R$ is upper bounded.

 {Overall, we deduce (with \eqref{Zlb}) that there is a constant $C_\eps$ such that
\begin{equation}\label{aprioriStep3}
R \le C_{\eps}; \quad \quad Z_{\eqref{RNlocal} }\ge C_{\eps}^{-1}; \text{ hence } \frac{d\underline{\lambda}_{\cM}}{d\underline{\lambda}_{N}} (\bs \eta_r) \le 2 C_\eps^2.
\end{equation}
}

{Suppose now that ${\mathsf{k}}>1$. Let $\{\underline{\lambda}_{N_i}\}_{1 \le i \le {\mathsf{k}}}$ be an independent collection of probability measures, each with the analogous law as $\underline{\lambda}_{N}$.   A similar line of argument gives for each collection $\bs \eta_r := \{\bs \eta_r^{(i)}\}_{1 \le i \le {\mathsf{k}}}$ of disjoint pairs of curves,
\begin{equation}
\frac{d\underline{\lambda}_{\cM} }{d\big(\prod_{1 \le i \le {\mathsf{k}}}\underline{\lambda}_{N_i} \big) }(\bs \eta_r) \le C_{\ve,{\mathsf{k}}} \label{aprioriStep3'}
\end{equation}

We now outline the places where the argument differs, which mainly boils down to some extra topological cases to consider.
\begin{itemize}
\item Steps 1 and 2 go through without any major changes. Firstly, \cref{prop:marginalM} is written for arbitrary number of punctures, which ensures that \eqref{RNlocal}, as well similar upper and lower bounds on $L$ is valid. The lower bound of $R$ in \eqref{Rprime} also works similarly. We restrict to only separated paths $\bs \eta_r^{(i)}$ and appeal to the fact that the random walk excursions have a positive chance of achieving $\cA_\cM$ in Wilson's algorithm after they hit $\partial N_i$.
\item The upper bound of $R$ (step 3) is the topologically complex step, where we now need to look at a more refined partition of the paths rather than looking at just $\cP_1, \cP_2$ as we did in step 3. Let $A_j(i)$ be the annulus of inner radius $N_i/2^{j+1}$ and outer radius $N_i/2^{j}$ for $0 \le j \le 2{\mathsf{k}}-1$ (with a similar abuse of notation used before). We sample the paths $(\eta_{i1}, \eta_{i2}, 1 \le i \le {\mathsf{k}})$ in this order. Consider one such path $\eta = \eta_{ij}$ and we look at the ordered sequence of points $\Xi_{ij}$ where $\eta$ coalesces with the previously sampled paths (e.g. $\eta$ might coalesce with some previously sampled path $\eta'$ which might itself coalesce with some other path $\eta''$ sampled even further before, and so on). Recalling the argument in step 3, we need some room where no such coalescence take place, so we can upper bound by crude events involving independent random walks.
By pigeonhole, there must be a $j$ so that none of $(A_j(i))_{1 \le i \le {\mathsf{k}}}$ contain any point of $\cup_{i,j} \Xi_{ij}$ (the first path do not coalesce with anything and there are $2{\mathsf{k}}$ annuli at each puncture). We define $\gamma_{ij}$ to be path $\eta_{ij}$ cut off when they enter the smallest such annuli for the first time (the reader may check that this surgery matches with that done in step 3 for ${\mathsf{k}}=1$). This allows now partition $\{\bs \eta: \bs \eta_r \oplus \bs \eta \in \cA_{\cM}\}$ into a finite number (in fact ${\mathsf{k}}$ many) of categories depending on which $j$ the paths are cut off. Nevertheless, the same sequence of arguments in step 3 yields

\begin{equation}
\P_{\bs x} ( \cA_{\cM}; \text{ compatible }) \lesssim \prod_{i=1}^{\mathsf{k}}\sum_{\bs \eta_r \oplus \bs \eta_{N_{\mathsf{k}}} \in \cA_{N_{\mathsf{k}}}} e^{-\mass_{(N_i)_r} (\cI (\bs \eta_{N_{\mathsf{k}}}))}\mu_{(N_i)_r} (\bs \eta_{N_{\mathsf{k}}})
\end{equation}
where $N_{\mathsf{k}} = N/2^{2{\mathsf{k}}}$ is the inner radius of the smallest annulus. 

Step 3 works with essentially no change for the denominator as well  the paths are independent for each puncture.
\end{itemize}
}

\textbf{Step 4.}
We combine the previous steps to conclude. As before, we write this step for ${\mathsf{k}}=1$ for simplicity but the proof for ${\mathsf{k}} > 1$ is in fact exactly the same.
Fix $F$ a bounded continuous, nonnegative functional on pairs of paths $\bs \eta$ (where continuity refers to the Hausdorff sense). Fix $\ve>0$. We first choose $r>0$ small enough so that the exponential of the mass of noncontractible loops which intersect the ball of radius $r$ around the punctures is in $[1, 1+\ve)$. Since the weak limit of $\bs \eta$ under $\underline{\lambda}_{N}$ consists of a.s. disjoint curves (except for their starting point), let us choose $\alpha >0$ sufficiently small such that for all sufficiently small $\delta>0$,
$$
\underline{\lambda}_{N}( \cS^c ) \le \eps/((1+\ve)2 C_\eps^2 \| F\|_{\infty}), \text{ where } \cS :=\{ B(x^i, \alpha r) \cap \eta_r^{3- i}  = \emptyset ,i = 1,2\}.
$$
Here $x^i$ denotes the tip of $\eta_r^i$.
In words, the (separation) event $\cS$ says that each endpoint of $\bs \eta_r$ is far from the other path.  Let $\cG $ be the set of paths where $\mathfrak B_i \setminus \eta_{r,i}$ does not intersect $B(x^{3-i}, \alpha r/M )$ for $i=1,2$.

 We prove later in  \cref{prop:quantitative_general} that $\underline{\lambda}_{\mathcal{M}}( \cG^c)$ probability can be made arbitrarily small for $M$ large enough. For the convenience of the reader, we mention quickly that this proposition allows us to bound $\underline{\lambda}_{\mathcal{M}}( \cG^c)$ in term of the probability of the same event for nice SRW excursions\footnote{We remark that while \cref{prop:quantitative} is located later in the paper, it does not use convergence result of \cref{P:localglobal}. More precisely, it does use the a priori bound of step 3 but not step 4 so there is no circular argument.}. In turn, these are small by Uchiyama's theorem (\cref{lem:scaling_limit_main}) and \cite[Lemma 5.3]{BLR_Riemann1}. %  is bounded above by some $C_r$ times the probability of independent random walk excursions started from the tips of $\bs \eta_r$ achieving $\cG$.
  %Using Uchiyama \note{insert reference}, we see that for all $r,\ve>0$,
  We can choose a large $M =M(\ve)$ so that $\underline \lambda_{\cM}(\cG^c) \le \ve/\|F\|_\infty$, observe that $|\E_{\underline{\lambda}_{\cM}} (F1_{\cG}) - \E_{\underline{\lambda}_{\cM}} (F)| \le \ve$. Thus it is enough to show that $\E_{\underline{\lambda}_{\cM}} (F 1_{\cG})$ converges.
Using \eqref{eq:RN_approx_L} and \eqref{aprioriStep3},
\begin{align}
  \E_{\underline{\lambda}_{\cM}} (F(\bs \eta)1_{\cG}) & = \E_{\underline{\lambda}_{\cM}} \Big(   \E_{\underline{\lambda}_{\cM}}( F(\bs \eta) 1_{\cG} \mid \bs \eta_r) \Big) \nonumber\\
  & \le \E_{\underline{\lambda}_{N}} \Big( (1+\ve)\frac{R}{{\mathsf Z_{M,N}}}  \E_{\underline{\lambda}_{\cM}}( F(\bs \eta)1_{\cG} \mid \bs \eta_r) \Big) \nonumber \\
  & \le \underline{\lambda}_{N}( \cS^c) \|F\|_{\infty}  (2(1+\ve)C_\eps^2) +  \frac{(1+\eps)}{{\mathsf Z_{M,N}}} \E_{\underline{\lambda}_N} \Big( R 1_{\cS}\E_{\underline{\lambda}_{\cM}}( F(\bs \eta)1_{\cG} \mid \bs \eta_r) \Big) \label{limitEsp}
 % & \le \eps + \E_{\P} (1_{\cS} \phi^\d (\bs \eta_r)),
\end{align}

Similarly using the lower bound in \eqref{eq:RN_approx_L},
\begin{align}
 \E_{\underline{\lambda}_{\cM}} (F(\bs \eta)1_{\cG}) \ge \frac{(1-\eps)}{{\mathsf Z_{M,N}}} \E_{\underline{\lambda}_N} \Big( R 1_{\cS}\E_{\underline{\lambda}_{\cM}}( F(\bs \eta)1_{\cG} \mid \bs \eta_r) \Big) \label{limitEsplower}
\end{align}
%where $\phi^\d( \bs \gamma_r^\d) = R( \bs \gamma_r^\d) \E_\Q (F (\bs \eta) \mid \bs \eta_r = \bs \gamma_r^\d)$. Now, as indicated in the notation, both $\bs \gamma_r^\d$ and the function $\phi^\d$ itself depend on the scale $\delta$.

Using \cref{prop:scalinglimit_disc}, we know that $\underline{\lambda}_{N}= (\underline{\lambda}_{N})^\d$ converges weakly, we may assume by Skorokhod's representation theorem that this convergence holds almost surely. Let $\bs \eta^\d_r$ be a sequence of discrete paths satisfying $\cS$ and which converge  as $\delta \to 0$ to a limit path $\bs \eta_r$ which satisfies $\cS$ (or rather, $\cS$ with $\alpha$ replaced by $\alpha/2$, but with an abuse of notation we will not explicitly mention this below).

We now argue that the conditional expectation term in \eqref{limitEsp} converges almost surely. Using an appropriate version of \eqref{eq:start_surface3}, we can write for any ${\bs \eta}_r \in \cS$
\begin{align}
\E_{\underline{\lambda}_{\cM}}( F(\bs \eta)1_{\cG} \mid \bs \eta_r) & \le (1+ \eps) \frac{\E_{\bs x}( F(\bs \eta) 1_{\cG \cap \cA_{\cM}}) }{\P_{\bs x}(\cA_\cM)}\nonumber \\
& \le (1+ \eps) \frac{\E_{\bs x}( F(\bs \eta) 1_{\cG \cap \cA_{\cM}}) }{\P_{\bs x}(\cA_\cM \cap \cG)}\label{eq:RN_upper_bound}\\
& = (1+ \eps) \frac{\E_{\bs x}( F(\bs \eta) 1_{\cG \cap \cA_{\cM}} | \tau_{r(1+\alpha/2M)} < \tau_{\bs \eta_r} )}{\P_{\bs x}(\cA_\cM \cap \cG | \tau_{r(1+\alpha/2M)} < \tau_{\bs \eta_r})}\nonumber
\end{align}
where we interpret $\E_{\bs x}, \P_{\bs x}$ as running Wilson's algorithm from the tips of $\bs \eta_r$, first from $x^1$ and then from $x^2$ (as always stopping when the walk either intersects $\bs \eta_r$, $\partial$ or closes a non-contractible loop). The last equality follows since we are on the good event $\cG$ on which the random walks have to exit $B_{r(1+\alpha/2M)}$ before hitting any of $\bs \eta_r$ in Wilson's algorithm in order to achieve $\cA_{\cM}$. Both the numerator and the denominator are now in the setup of our Markov chain in \cref{sec:discreteMC}, which converges almost surely. 

We now lower bound the conditional expectation in \eqref{limitEsp} by an expression which is close to the ratio in the last line above. Note that for any subcollection of paths $\cS' \subset \cS$,

\begin{equation}
\E_{\underline{\lambda}_{\cM}}( F(\bs \eta)1_{\cG} \mid \bs \eta_r)1_{{\bs \eta_r} \in \cS} \ge \E_{\underline{\lambda}_{\cM}}( F(\bs \eta)1_{\cG} \mid \bs \eta_r) 1_{{\bs \eta}_r \in \cS'} \ge (1-\ve)\E_{\bs x}(F(\bs \eta)1_{\cG} | \cA_\cM)1_{\bs \eta_r \in \cS'}\label{eq:RN_lower1}
\end{equation}
where the latter inequality  has the same interpretation as in the upper bound.
Thus, it is enough to show that there is a good collection of paths $\cS'$ in $\cS$ with
\begin{equation}
 \P_{\bs x}(\cG | \cA_\cM)1_{\bs \eta_r \in \cS'}  >(1-2\sqrt{\ve}) \text{ and }\underline{\lambda}_{\cM}((\cS')^c)  \le \frac{\sqrt{\ve}}{\|F\|_\infty}.\label{eq:RN_Markov}
\end{equation}
The first  inequality above allows us to write the rightmost expression in \eqref{eq:RN_lower1} in terms of the ratio in \eqref{eq:RN_upper_bound} up to a multiplicative factor very close to 1. The second inequality in \eqref{eq:RN_Markov} makes sure that the set of paths on which this approximation does not work has a small probability in terms of $\ve$. We leave these details to the reader and explain how to obtain the desired set $\cS'$.

The existence of $\cS'$ follows from the fact that $\cG$ was chosen with $\underline{\lambda}_\cM (\cG^c) \le \ve/\|F\|$ and  Markov's inequality. Indeed, Markov's inequality yields $\underline{\lambda}_\cM(\underline{\lambda}_\cM(\cG  | \bs \eta_r) \le 1-\sqrt{\ve}) \le \sqrt{\ve}/\|F\|$. Thus we choose $\cS':= \{\bs \eta: \underline{\lambda}_\cM(\cG  | \bs \eta_r) >1-\sqrt{\ve}\}$. Using the same upper bound as in \eqref{eq:RN_upper_bound},  $\underline{\lambda}_\cM(\cG  | \bs \eta_r) $ is well approximated by  $\P_{\bs x}(\cG | \cA_\cM)$ up to a multiplicative factor close to 1.

Combining all this, we obtain that almost surely, for all $\delta, \delta'$ small enough

\begin{equation*}
|\E^\d_{\underline{\lambda}_{\cM} }(F({\bs \eta^\d})1_{\cG} | \bs \eta^\d_r)  -\E^{\d'}_{\underline{\lambda}_{\cM} }(F({\bs \eta^{\d'}})1_{\cG} | \bs \eta^{\d'}_r)|= O(\sqrt{\ve}).
\end{equation*}

%Let us check the convergence of the conditional expectation in the second term in the right hand side above. First of all, it is a consequence of \note{insert right reference} that the conditional $ (\underline{\lambda}_{\cM})^\d$-law of $\bs Y^\d$ given  $\bs Y^\d_r = \bs \eta_r^\d$, converges as $\delta \to 0$ to a certain function of $\bs \eta_r$ only.
%This is because the $\underline{\lambda}_{\cM}^\d$ law of $\bs Y^\d$ is simply the Markov chain described in that Theorem, conditioned on the event $\cA_{\cM}$. More precisely, since we restricted to $\cG$, we can take the first step of the Markov chain for each path to be the event that the walker exits $B(x^i, \alpha r/2M)$ before hitting $\eta_{r,i}$ since none of the other paths intersect this ball. After this step, the conditioning event $\cA_\cM \cap \cG$ has a uniformly positive probability in $\delta$ by uniform crossing. Hence we can conclude that the Markov chain converges under this conditioning.
% %Furthermore Theorem 5.17 in \cite{BLRtorus} is exactly the convergence of the unconditioned chain and $\cA_{\cM}$ is a.s. continuous for the limit of the chain.
%This proves that $\E_{\underline{\lambda}_{\cM}^\d} ( G( \bs \eta^\d) \mid \bs \eta_r^\d)$ (with $\bs \eta_r^\d  \to \bs \eta_r$) converges as $\delta \to 0$.

The convergence of $R= R^\d$ is analogous: we can restrict to $\cG$ in the event of the numerator by losing $\ve$ in the probabilities, again using \cref{prop:quantitative_general}. Note that, the random walk from both the tips hit $\partial B(x^i, \alpha/2M)$ before hitting $\eta_{r,i}$, both on $\cA_\cM \cap \cG$ and $\{\cA_N, \tau_N <\tau_{\bs \eta_r}\}$. Thus we can condition both the numerator and the denominator by this event. These conditional probabilities are uniformly positive by uniform crossing. Thus, using the convergence of a similar Markov chain as in \cref{sec:discreteMC}, we can conclude that $R$ converges.
Taking $F =1$ this gives a lower bound on ${\mathsf Z_{M,N}}$.

Overall, we obtain that for $\delta, \delta'$ small enough,
\[
| \E_{\underline{\lambda}_{\cM}^\d} (F (\bs \eta)) - \E_{\underline{\lambda}_{\cM}^{\d'}}( F( \bs \eta )) |=  O( \sqrt{\eps} )
\]
since both terms are well approximated by the same expression defined in the right hand side of \eqref{limitEsp} with the same value of $r$ depending only on $\eps$. This proves \cref{P:localglobal} in the case $k =1$ as desired. 

For ${\mathsf{k}}>1$, we simply need to run the above argument by replacing $\ul_N$ by independent copies around small neighbourhoods of the punctures. As mentioned at the beginning of the proof, appropriate modifications of the good event $\cG, \cS$, $\cS'$, and the a priori bound \eqref{aprioriStep3'} completes the proof of \cref{P:localglobal}.
\end{proof}
\begin{proof}[Proof of \cref{thm:main_Temp_CRSF}]Recall that \cref{thm:main_Temp_CRSF} contains 4 assertions: the Temperleyan CRSF on $\Gamma^\d$ has a scaling limit (in Schramm topology) under $\Pwils$; the limit does not depend on the sequence; it is conformally invariant; and a uniform bound on the exponential moment of the number of non-contractible loops under $\Pwils$ (implying the convergence of the Temperleyan forest under $\Ptemp$).

The scaling limit under $\Pwils$ is the main point and it was proved in two parts, \cref{P:localglobal,thm:CRSF_universal}:
%This follows from \cref{P:localglobal,thm:CRSF_universal}. 
 (For $\chi=0$ (torus and annulus), we are simply done using \cref{thm:CRSF_universal}. In the remaining general case, one can apply \cref{P:localglobal} to first obtain convergence and conformal invariance of the skeleton. Conditioned on these branches the remaining manifold is divided into annuli. On each annulus, apply \cref{thm:CRSF_universal}. For the uniform bound on exponential moments, note that the skeleton can only create a bounded number of annuli. Therefore it follows directly from the corresponding bound in \Cref{thm:CRSF_universal} (or \cref{lem:non-contractible_cont}). Since $\frac{d\Ptemp}{d \Pwils} \propto 2^{K^\dagger}$, the scaling limit under $\Ptemp$ then follows (again \cref{lem:non-contractible_cont} proves that the number on contractible loops converges).

%The fact that the limit law does not depend on the sequence of approximation is an immediate consequence of the fact that we do not ask for any consistency between the graphs $\Gamma^\d$ for different values of $\delta$. In particular if $(\Gamma_1^\d)_{\delta>0}$ and $(\Gamma_2^\d)_{\delta>0}$ both satisfy our assumptions with respect to the same manifold and with the same marked points, then any combination of the two still satisfy our assumptions and therefore the two scaling limits have to agree. 

The fact that the limit law does not depend on the sequence of approximation is an immediate consequence of the fact that, if we are given two sequences $(\Gamma_1^\d)_{\delta>0}$ and $(\Gamma_2^\d)_{\delta>0}$ satisfying the assumptions of Section \ref{sec:setup}, then the reunion of these two sequences also fulfills the same assumptions, therefore there cannot be two different limits. 

The conformal invariance of the scaling limit then follows from the fact that our assumptions are conformally invariant (\cref{lem:conf_inv_assumption}) and that conformal maps are compatible with the topology of convergence (\cref{rmk:schramm}).
\end{proof}

\subsection{Quantitative estimates for the skeleton}\label{sec:quant_special}

In this section, we show that branches in the skeleton satisfy certain regularity conditions (which are also used in \cite{BLR_Riemann1}).
%{\color{gray}assumptions in \cite{BLRtorus} (Assumption 8.1).} \note{Change assumption to proposition in paper 2.} 
The main result we show in this section is that if an event (such as visiting a given point, or making a lot of winding) is unlikely for a branch of the CRSF, then it is also unlikely for the branches of the skeleton of a Temperleyan CRSF. We are content with proving such statements for portions of branches macroscopically away from the punctures as this is the input we need in \cite{BLR_Riemann1}.
To this end, we prove the following general bound, with the role of the unlikely event played by a generic event $E$. We use the notation $\bs \eta_r := \{\bs \eta_r^{(i)} = (\eta^{(i)}_{r,1},\eta^{(i)}_{r,2})\}_{1 \le i \le {\mathsf{k}}}$ to denote the initial bits of the branches in the skeleton as in \cref{sec:conv_special}. Recall $\mathfrak s =(\mathfrak B_{i1}, \mathfrak B_{i2})_{1\le i \le {\mathsf{k}}}$ denotes the skeleton (here with a small abuse of notation we do not include the path $e_i$).

We will first write the estimate for ${\mathsf{k}}=1$ as it is easier to state. The ${\mathsf{k}}>1$ case will be handled after that. For $r < r'$ and a fixed $\bs \eta_r$, let $\tilde \P_{\cM_{r,r', \bs \eta_r}}$ denote the law of two independent simple random walks started from the tips of $\bs \eta_{r,j}$, conditioned not to hit $\eta_{r,j}$ before the random walk first exits the ball of radius $r'$ (i.e., it is an excursion from the tip till the first exit of the ball of radius $r'$, after which it is allowed to proceed unconditionally). Sometimes we drop the $\bs \eta_r$ from the subscript whenever it is clear from the context. Let $\bs x = (x^{1},x^2)$ denote the tips of $\bs \eta_r$ as before. Let $\tilde \mu_{\cM_{r,r'}} $ be the law induced by the loop erasures of the walks sampled from $\tilde \P_{\cM_{r,r', \bs \eta_r}}$, when we run Wilson's algorithm in $\cM$ first from $\eta_{r,1}$ and then from $\eta_{r,2}$. That is, Let $\gamma^1_r$ denote the portion of the loop erasure from the tip of $\eta_{r,1}$ until it hits either $\bs \eta_r \cup \partial$ or creates a non-contractible loop. Let $\gamma^2_r$ denote the portion of the loop erasure from the tip of $\eta_{r,2}$ until it hits $\bs \eta_r \cup \partial \cup \gamma^1_r$ or creates a non-contractible loop. Note that in order to create some $\bs \eta$ using Wilson's algorithm in $\cM_r$, $\gamma^1_r$ is exactly the portion that has to come from the first walk and $\gamma_r^2$ is the portion from the second walk, i.e,
\begin{equation}
{\tilde \mu}_{r, r'}( \bs \eta) = \P_{\bs x}( LE( X_j) = \gamma^j_r| \tau^j_{\partial B_{r'}} < \tau^j_{\eta_r^j})\label{eq:tildemu}
\end{equation}

Let $E$ be an event measurable with respect to the branches $\mathfrak s \setminus \bs \eta_r$. We warn the reader that it is possible for $(\mathfrak B_{i1} \setminus \eta^{(i)}_{r,1},  \mathfrak B_{i2} \setminus \eta^{(i)}_{r,2})_{1 \le i \le {\mathsf{k}}}$ to have common parts (see \cref{fig:cases_merge}).

%We shall interpret $\tilde \mu_{\cM_{r,r'}} (E)$ as follows. Sample the (conditioned) walks and then loop erase the walks stopped according to Wilson's algorithm in $\cM_r$, first for the walk from the tip of $\eta_{r,1}$ and then from the tip of $\eta_{r,2}$. Then $\tilde \mu_{\cM_{r,r'}} (E)$ denote the probability of the event that $E$ occurs.

%The key point is to transfer estimates from the case of a standard branch to the special ones.
\begin{prop}\label{prop:quantitative}
Suppose we have only one puncture (i.e. ${\mathsf{k}}=1$).
There exists a constant $C$ depending only on $r$ and $\mathcal{M}$ such that for all $\delta>0$ sufficiently small, for all events $E$ measurable with respect to $\mathfrak s \setminus \bs\eta_r$,
\begin{equation*}
	\underline{\lambda}_{\cM}( E) \leq C \left(\sqrt{ \sup_{\bs\eta_{r/2}} \tilde \mu_{r/2,r}(E) + \sup_{\bs\eta_{r/4}} \tilde \mu_{r/4,r/2}(E) } \right)
	\end{equation*}
	where the supremum is over all possible $\bs \eta_{r/2}$ (resp. $\bs \eta_{r/4}$) in $\cA_{r/2}$ (resp. $\cA_{r/4}$).
\end{prop}
In typical applications, one should think of the choice of $r$ to be $O(1)$.
Before proving \cref{prop:quantitative}, we need two lemmas.
 Let $\P_{\bs x}$ denotes the law of independent random walks in $\cM$ started from these tips. Also define the stopping times $\tau_{N} = (\tau^1_{N}, \tau^2_{N})$ as in \cref{sec:conv_special}.

\begin{lemma}\label{lem:ratioP}
For all $r > 0$, there exists a constant $C_r > 0$ such that for all $\delta$ small enough and for all $\bs \eta_r \in \cA_r$
\begin{equation*}
\frac{\P_{\bs x}( \cA_{N}, \tau_{N} < \tau_{\bs{\eta}_r})}{\prod_{j} \P_{x^j}( \tau^j_{N} < \tau_{\eta_{r,j}})} \leq C_r \frac{\P_{\bs x}( \cA_\cM, \text{compatible})}{\prod_{j} \P_{x^j}( \tau_{\partial B_{2r}} < \tau_{\eta_{r,j}})}.
\end{equation*}
\end{lemma}
\begin{proof}
By \cref{cor:separation_chain}, if we do Wilson's algorithm in $N$, conditioned on $\tau_{N} < \tau_{\bs{\eta}_r}$ and $\cA_{N}$, the loop-erasures are separated with positive probability. On that event, $\cA_\cM$ can also happen with positive probability by uniform crossing, so $$\P_{\bs x}( \cA_N, \tau_N < \tau_{\bs{\eta}_r}) \leq C_r \P_{\bs x}( \cA_\cM, \text{ compatible}).$$  The denominator is easier to lower bound, simply note that once the random walk leaves $B_{2r}$, there is a $C_r$ chance of exiting $N$ before hitting $\eta_{r,j}$ by crossing assumption.
\end{proof}

Recall $\mu_N$ is the law of two independent simple random walks started from the punctures, and stopped when they exit $N$.
\begin{lemma}\label{lem:partition_function}
	For all $r > 0$, there exists a constant $C_r > 0$ such that for all $\delta$ small enough %and for all $\bs \eta_r \in \cA_r$,
	\[
	\int \frac{e^{- \mass_{N}( \cI( \bs{\eta}_r))}}{Z(\lambda_N)} d \mu_N (\bs \eta_r) \leq C_r.
	\]
\end{lemma}
\begin{proof}
	We first bound
\[
	\int \frac{e^{- \mass_{N}( \cI( \bs{\eta}_r))}}{Z(\lambda_N)} d \mu_N \leq \int \frac{e^{- \mass_{B_r}( \cI( \bs{\eta}_r))}}{Z(\lambda_N)} d \mu_N = \frac{Z(\lambda_r)}{Z(\lambda_N)}.
\]
The lemma is then a direct consequence of point (c) in \cref{cor:separation}.
\end{proof}
\begin{proof}[Proof of \cref{prop:quantitative}]
Let $B_r$ denotes all points within distance $r$ from the puncture.
Let
$$
p_r := \sup_{\bs\eta_{r/2}} \tilde \mu_{{r/2,r}}(E) + \sup_{\bs\eta_{r/4}} \tilde \mu_{{r/4,r/2}}(E)
$$
\begin{figure}[h]
\centering
\includegraphics[scale = 1]{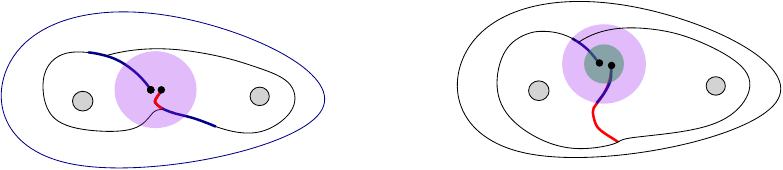}
\caption{The purple disc is the ball of radius $r/2$ and the green disc is the ball of radius $r/4$. The black curve is $\eta_1$ and the red curve is $\eta_2$. The left is the case $\cP_1$ while the right is case $\cP_2$. The dark blue curve is necessarily disjoint in the annulus with inner radius $r/2$ (resp. $r/4$) and outer radius $r$ (resp. $r/2$) in the left (rep. right). }\label{fig:cases_merge}
\end{figure}

We need to divide the set of paths in $\cA_{\cM} $ into several categories as in \cref{sec:conv_special}. We write $\eta_j = \mathfrak B_j$ for the branches of the skeleton in this proof.
Let $\cP_1$ be the set of paths where $\eta_2$ coalesces with $\eta_1$ inside distance $r/2$ from the puncture. Notice that in this case, the portion of the branches between the first hit at distance $r/2 $ to the first hit of distance $r$ are necessarily disjoint. Let $\cP_2$ be the complement of $\cP_1$ in the set $\cA_\cM$. Define
a set of good paths $\bs \eta_s$  in $\cA_s$ to be
$$
\cG_{s,t}:= \Big\{\bs \eta_s : \frac{\prod_{j=1,2}\P_{x^j} (\tau^j_{\partial B_{t}} < \tau_{{\eta}_{s,j}} )}{\P_{\bs x}( \cA_\cM, \text{ compatible}) }\leq \frac{1}{\sqrt{p_r}} \Big\}.	
$$
where $\bs x = (x^1,x^2)$ is the tip of $\bs \eta_s$. The broad idea is to bound $\underline{\lambda}_{\cM}(E)$ on the set of good paths $\cG_{r/2,r}$, show $\cG_{r/2,r}^c$ has small probability, and employ Markov's inequality. Because of similar topological reasons as explained in Step 3 of the proof of \cref{P:localglobal}, we need to carefully do a surgery to make sure certain portions of the branches are disjoint. This explains why we need $\cG_{r/4,r/2}$ as well and there are two terms in the definition of $p_r$.

Recall from \eqref{eq:tildemu}, for any $r'>r$, $\tilde \mu$ can be expressed as 
\[
{\tilde \mu}_{r, r'}( \bs \eta) = \P_{\bs x}( LE( X_j) = \gamma^j_r| \tau^j_{\partial B_{r'}} < \tau^j_{\eta_r^j})
\]
where $LE( X_j)$ denotes the loop-erasure following Wilson's algorithm and $\gamma^j_r$ is as defined just before \eqref{eq:tildemu}.
 Note that $\gamma^j_r$ could be very different from $\eta_j \setminus \eta_{j,r}$ (cf. \cref{fig:cases_merge}).

We will first bound the probability of $E \cap \cP_1 \cap \cG_{r/2,r}$. Observe that, for $\bs \eta \in \cP_1$, the event $\{\tau^j_{\partial B_{r}} < \tau^j_{\eta_{r/2}^j}\}$ is included in then event $\{LE( X_j) = \gamma^j_{r/2}\}$ (as long as $\gamma^j_{r/2}$ satisfies $\cA_\cM$), therefore we have
\begin{align*}
\tilde \mu_{{r/2,r}}( \bs\eta)  &= \frac{\P_{\bs x}( LE( X_j) = \gamma^j_{r/2},j=1,2)}{\prod_{j=1,2} \P(\tau^j_{\partial B_{r}} < \tau^j_{\eta_{r/2}^j})}=\frac{q (\cup \gamma^j_{r/2}) \exp(\mass_{\cM_{r/2}} ((\cup \gamma^j_{r/2}) \cap \cC_{\bs \gamma_{r/2}}))}{\prod_{j=1,2} \P(\tau^j_{\partial B_{r}} < \tau^j_{\eta_{r/2}^j})}.
\end{align*}
Note that outside of $\cP_1$, the formula does not hold because one would also need to include the event $\tau^2_{\partial B_{r}} < \tau^2_{\eta_{r/2}^2}$ in the numerator when $\gamma^2_{r/2}$ does not reach $ \partial B_{r}$.
Using this we have
\begin{align*}
\underline{\lambda}_\cM (E \cap \cP_1 \cap \cG_{r/2,r})& = \E_{\ul_{\cM} } [ 1_{\cG_{r/2,r} } \E_{\ul_{\cM}} ( 1_{E\cap  \cP_1} | \bs \eta_{r/2}) ] \\
& = \E_{ \underline{\lambda}_{\cM}} \left[1_{\cG_{r/2,r}} \int_{\bs \eta \succeq \bs \eta_{r/2}} \frac{\underline{\lambda}_\cM( \bs\eta | \bs\eta_{r/2})}{\tilde \mu_{\cM_{r/2,r}}( \bs\eta)}1_{E \cap \cP_1} d \tilde\mu_{\cM_{r/2,r}}( \bs\eta)\right] \\
&  \asymp
 \E_{\underline{\lambda}_{\cM}} \Big[1_{\cG_{r/2,r}} \int_{\bs \eta \succeq \bs \eta_{r/2}}	\frac{e^{\mass_{\cM_{r/2}} (\cI(\cup_j \gamma^j) \cap \cC_{\bs \eta}) }}{ e^{\mass_{\cM_{r/2}} ( \cI( \cup_j \gamma^j) \cap \cC_{\bs \gamma})}}
 \frac{\prod_{j=1,2}\P_{x^j} (\tau^j_{\partial B_{r}} < \tau_{{\eta}_{r,j}} )}{\P_{\bs x}( \cA_\cM, \text{ compatible}) } 1_{E \cap \cP_1} d \tilde\mu_{{r/2,r}}( \bs\eta) \Big] .
\end{align*}
%where in the last line $\gamma^j = \gamma^j_{r/2}$. In the second equality, since we are integrating over $\cP_1$, the paths $\gamma_{r/2}^j$ \note{notation...fix} are disjoint between $r/2$ and $r$.
%
%
%it must be the case that the random walks generating them from the tips exit the ball of radius $r$ before hitting $\eta_{r/2}^j$.
%
%
%Therefore in the ratio in the second line, the $q$ term cancels, leaving only the loop measure term (in particular there is no $h$ transform term in the denominator). Finally, it is an asymptotic equality as we are only looking at loops in $\cM_{r/2}$, thereby ignoring some non-contractible loops as before.

As usual the ratio of loop measure is bounded by some $C_r$ (bounding by mass of loops of non  contractible support). By definition of $\cG_{r/2,r}$, the ratio of the probabilities is bounded by $1/\sqrt{p_{r}}$. Thus we get
\begin{equation*}
\underline{\lambda}_\cM (E \cap \cP_1 \cap \cG_{r/2,r}) \le \frac1{\sqrt{p_{r}}} \E_{\underline{\lambda}_{\cM}}(1_{\cG_{r/2,r}} \tilde\mu_{r/2,r} (E \cap  \cP_1)) \le \sqrt{p_r} .
\end{equation*}
since $\tilde\mu_{r/2,r} (E \cap  \cP_1)\le p_r$ uniformly over the choice of $\bs \eta_{r/2}$ by definition.

We employ a similar tactic on $\cP_2$ with a slight difference in the surgery step: we cut the paths at $r/4$ instead of $r/2$. Then we restrict to $\cG_{r/4,r/2}$ and we look at the Radon--Nikodym derivative with respect to $\tilde \mu_{r/4,r/2}$. This allows us a similar series of estimates which reduce to
\begin{equation*}
\underline{\lambda}_\cM (E \cap \cP_2 \cap \cG_{r/4,r/2}) \le  \frac1{\sqrt{p_{r}}} \E_{\underline{\lambda}_{\cM}}(1_{\cG_{r/4,r/2}} \tilde\mu_{r/4,r/2} (E \cap  \cP_2)) \le \sqrt{p_r}.\end{equation*}
Now we claim that
\begin{equation}
\underline{\lambda}_\cM (\cG_{r/2,r}^c \cup \cG^c_{r/4,r/2}) \le 2C_r^3 \sqrt{p_r} \label{eq:complementG}
\end{equation}
To see this, we now take the Radon--Nikodym derivative with respect to $\underline{\lambda}_N$ and use \cref{aprioriStep3}. We only show how to upper bound $\cG_{r/2,r}^c$ (as the other term is similar). To that end, note that
	\begin{align*}
	\underline{\lambda}_\cM (\cG_{r/2,r}^c) & = \int \frac{\ul_{\cM}(\bs{\eta}_{r/2})}{\ul_{N}(\bs{\eta}_{r/2})} 1_{\cG_{r/2,r}^c} d \ul_{N}(\bs{\eta}_{r/2}) \\
	& \leq C_r \int 1_{\cG_{r/2,r}^c} d \ul_{N}(\bs{\eta}_r) \\
	& \leq C_r\int 1_{	\big\{\frac{\P_{\bs x}( \cA_{N}, \tau_{N} < \tau_{\bs{\eta}_r)}}{\prod_{j} \P_{x^j}( \tau^j_{N} < \tau_{\eta_{r,j}})} \le C_r\sqrt{p_r}\big\}} \frac{\underline{\lambda}_N(\bs\eta_r)}{\mu_N(\bs\eta_r)} d \mu_N(\bs\eta_r) \\
	 &\leq C_r\int 1_{	\big\{\frac{\P_{\bs x}( \cA_{N}, \tau_{N} < \tau_{\bs{\eta}_r)}}{\prod_{j} \P_{x^j}( \tau^j_{N} < \tau_{\eta_{r,j}})} \le C_r\sqrt{p_r}\big\}} \frac{e^{- \mass_N (\cI(\bs \eta_r))}}{Z(\lambda_N)}  \frac{\P_{\bs x}( \cA_{N}, \tau_{N} < \tau_{\bs{\eta}_r)}}{\prod_{j} \P_{x^j}( \tau^j_{N} < \tau_{\eta_{r,j}})}d \mu_N(\bs\eta_r) \\
	& \leq C_r \int \frac{e^{- \mass_N (\cI(\bs \eta_r))}}{Z(\lambda_N)} C_r \sqrt{p_r} d \mu_N(\bs\eta_r) \\
	& \leq C_r^3 \sqrt{p_r}.
	\end{align*}
In the third line above, we used \cref{lem:ratioP}. The proof is complete for ${\mathsf{k}}=1$ by combining \eqref{eq:complementG} with the previous estimates.
%\note{unfinished}
\end{proof}
We now state \Cref{prop:quantitative} for general ${\mathsf{k}}\ge 1$.
\begin{prop}\label{prop:quantitative_general}
There exists a constant $C$ depending only on $r$ and $\mathcal{M}$ such that for all $\delta>0$ sufficiently small, for all events $E$ measurable with respect to $\mathfrak s \setminus \bs\eta_r$,
\begin{equation*}
	\underline{\lambda}_{\cM}( E) \leq C \left(\sqrt{ \sum_{j=1}^{2{\mathsf{k}}}\sup_{\bs\eta_{r/2^j}} \tilde \mu_{\mathcal{M}_{r/2^j,r/2^{j-1}}}(E)} \right)
	\end{equation*}
	where the supremum is over all possible $\bs \eta_{r/2^j}$.
\end{prop}
\begin{proof}
We sketch the proof as the idea is similar to the proof of \eqref{aprioriStep3'}.  Define $\eta_{11}$ to be the first branch, and iteratively define $\eta_{ij}$ to be the branch until the first hit of a previously sampled branch for $1 \le i \le \mathsf{k}, j=1,2$. 

Let $\Xi$ the set of triple points defined in the  skeleton $\mathfrak{s}$, which we can think of as all the points where one loop-erased path merges into another. Recalling the argument in the ${\mathsf{k}}=1$ case, we need some room where no such coalescence take place. As in the proof of \eqref{aprioriStep3'}, define a set $\mathsf{A}$ of annuli of exponentially increasing radii, finite but sufficiently large that, by the pigeonhole principle, at least one annuli contains no point in $\Xi$. Say that it is $A( r/2^j, r/2^{j-1})$, we can now proceed exactly as for \cref{prop:quantitative} by cutting the paths at $r/2^j$ to bound $\underline{ \lambda}_\cM$ by $\tilde{\mu}_{r/2^j, r/2^{j-1}}$ on that event. We leave the details to the reader.
\end{proof}

\appendix

\section{Random walk on the universal cover}\label{app:RW}

The purpose of this section is to prove the following lemma, saying that the random walk on the universal cover is transient, uniformly over the mesh size (in a specific sense).

Let $\cM$ be a hyperbolic manifold with no boundary and let $p^{-1}$ denote the associated conformal lift from $\cM$ to the universal cover of $\cM$ taken to be the unit disc. Let $\Gamma^\d$ be a sequence of graphs on $\cM$ satisfying the assumption of \cref{sec:setup}. Here we denote by $X$ the random walk on the universal cover, i.e on $p^{-1}( \Gamma^\d) = \tilde \Gamma^\d$. We emphasise here that $\tilde \Gamma^\d$ is not the universal cover of the graph $\Gamma^\d$, but the pre-image of the embedding of $\Gamma^\d$ under the map $p$.

\begin{prop}\label{lem:boundary_convergence}
For all $\eps > 0$, there exists $\eta > 0$ such that for all $\delta$ small enough, for all $v \in \tilde\Gamma^\d $ with $|v| \ge 1- \eta$, 
\[
\P_v(  X[0, \infty] \cap B(0, 1 - \eps) = \emptyset) \geq 1 - \eps.
\]
\end{prop}
{Observe that if the manifold has a boundary, then the set of points in $\partial D$ which is not a lift of $\partial M$ has Lebesgue measure 0. Then by the Beurling type estimate (\cref{lem:Beurling}), the estimate in \cref{lem:boundary_convergence}  is clear. Thus we focus only on the case where $\cM$ has no boundary in the rest of this appendix. Observe that in this case, there is a fundamental domain of $\cM$ which is a compact subset of $\D$.}
	Note that the Invariance principle assumption (assumption (\ref{InvP}) in \cref{sec:setup}) does not directly imply the result since it only covers convergence in compact subsets of the disc. %\note{rewrite the assumption in a way that make it more precise.}

	The first step is to prove some uniform convergence starting from a compact set. Let $B_{\H}(z)$ denote the hyperbolic ball of (hyperbolic) centre $z$ and (hyperbolic) radius $1$ for $z \in \D$ and let $\tau_z$ denote the exit time from this ball. Also let $\phi_z$ be any conformal map sending $\D$ to itself and $z$ to $0$.
\begin{lemma}\label{lem:hitting_centre}
For any compact set $K \subset \D$ and any $\eps>0$, there exists $\delta_0>0$ such that for all $\delta < \delta_0$ and for all $v \in \tilde\Gamma^\d \cap K$, there exists a coupling between $X( \tau_{v})$ and a uniform variable $U$ on $\partial B_\H(0)$ such that
\[
\P_v( | \phi_v( X_{\tau_v}) - U | > \eps ) \leq \eps.
\]
\end{lemma}
\begin{proof}
	By the invariance principle assumption, we know that for any $R$, for any sequence $(v^\d)_{\delta}$ of points converging to some $z$ in the interior of $\D$, the law of $ X_{\tau_z}$ under $\P_{v^\d}$ converges to the law of the exit point position from $B_\H(z)$ by a Brownian motion started at $z$. On the other hand, since the uniform crossing estimate implies a Harnack's inequality, it is clear that the law of the exit position from a ball is equicontinuous in $\delta$ when one changes the starting point. Overall if we fix a compact $K \subset D$ and $\eps > 0$, we can see that for all $\delta$ small enough, for all sequences $(v^\d)_{\delta} \in K \cap \tilde\Gamma^\d$ converging to $z$, the law of $ X_{\tau_{v^\d}}$ is $\eps$-close to the law of the exit point from $B_\H(z)$ for a Brownian motion. It is clear that we can choose to measure this distance by mapping back the centre to $0$ if we wish to do so, which proves the lemma.
\end{proof}

Using the periodicity of $\tilde \Gamma^\d$, we can then `get rid' of the compact set $K$. This is the content of the next corollary.
\begin{corollary}\label{cor:disc_convergence_uniform}
	For any $\eps > 0$, there exists $\delta_0$ such that for all $\delta < \delta_0$ and for all $v \in \tilde\Gamma^\d$, there exists a coupling between $X( \tau_{v})$ and a uniform variable $U$ on $\partial B_\H(0)$ such that
	\[
	\P_v( | \phi_v(  X_{\tau_v}) - U | > \eps ) \leq \eps.
	\]
\end{corollary}
\begin{proof}
	Fix $\eps > 0$ and find $\delta_0$ associated to the previous lemma for any $K$ containing at least one full Dirichlet fundamental domain. Now fix an arbitrary $v \in \tilde \Gamma^\d$ and let $v'$ be a point in $K$ such that $p(v) = p(v')$, i.e., $v$ and $v'$ are two copies of the same vertex from $\Gamma^\d$. By definition, this means that there is an element of the Fuchsian group, $g\in G$, such that $v = g (v')$. Note furthermore that $g$ leaves $\tilde \Gamma^\d$ invariant, since if $e'$ is an edge of $\tilde \Gamma^\d$ and $e = g(e')$ then $p(e) = p(e')$ is an edge of $\Gamma^\d$. Therefore, if $X^{v'}$ denotes a random walk on $\tilde \Gamma^\d$ starting from $v'$ then $g(X^v)$ is a random walk on the same graph starting from $v$.

Let $\tilde \phi_v = \phi_{v'} \circ g^{-1}$, which is a M\"obius map sending $v$ to 0, $\D$ to $\D$, and so is a valid choice for $\phi_v$. Note that $\tilde \phi_v (X^v) = \phi_{v'} ( X^{v'}) $ and hence (since M\"obius maps are isometries) $\tilde \phi_v( X^v_{\tau_v})  = \phi_{v'} ( X^{v'}_{\tau_{v'}})$. Applying Lemma \ref{lem:hitting_centre} we can thus find a coupling such that 
\begin{equation}\label{eq:appendix}
	\P_v( | \tilde \phi_v(  X_{\tau_v}) - U | > \eps ) \leq \eps.
\end{equation}
 Since the law of $U$ is invariant under rotations of the unit disc around zero, we see that \eqref{eq:appendix} does not depend on the choice of $\phi_v$, thereby establishing the result. 
%  Let $X^v$ and $ X^{v'}$ denote respectively random walks started in $v$ and $v'$. 
%  \blue{Recall that M\"obius maps from $\D $ to $\D$ which fix an interior point of the disc are unique up to a rotation of the disc about the origin}.  Thus we can choose the rotations in $\phi_v$ and $\phi_{v'}$ such that $\phi_{v'}^{-1} \circ \phi_{v} (\tilde \Gamma^\d) = \tilde \Gamma^\d$. \blue{Indeed, any M\"obius map which maps $v$ to $v'$ has to be of this form and the Fuchsian group $G$ so that $\cM = \H / G$ is a subgroup of the M\"obius group (see \Cref{SS:universalcover} for references and background)}. Therefore $\phi_{v'}^{-1} \circ \phi_{v}( X^v) \overset{d}{=}  X^{v'}$. Since hyperbolic balls are also transformed isometrically by $\phi_v, \phi_{v'}$ which are M\"obius transforms, the corollary follows.
\end{proof}

\begin{proof}[Proof of \cref{lem:boundary_convergence}]
As mentioned before, we only focus on the case when $\cM$ does not have a boundary. 
	Fix $\eta$ to be chosen small enough later and $v=v_0 \in \D \setminus B( 0, 1-\eta)$. Also assume that $\delta$ is small enough for \cref{cor:disc_convergence_uniform} to hold for some $\eps$ which we will also choose later. The idea is to look at the random walk the sequence of stopping time $\tau_n$ defined by $\tau_{n+1} = \inf\{ t > \tau_n :  X_{t} \not \in B_\H(  X_{\tau_n})\}$ and to consider the sequence of hyperbolic distance $D_n$ between $ X_{\tau_n}$ and $0$.
	
	We start by estimating $\E( D_1 - D_0)$. Since the hyperbolic distance is invariant by M\"obius transform, we will apply the function $\phi_v$ which sends $0$ to a point close to the boundary and $v$ to $0$. To simplify notations and without loss of generality, let us assume that $\phi_v(0) = 1-\eta$ and let us write $\rho = \frac{e-1}{e+1}$ which is the Euclidean radius of the hyperbolic unit ball centred at 0, i.e., the Euclidean radius of $B_\H(0)$. We have
	\begin{align*}
	\E( D_1 - D_0) & =\E( d_{\H} ( \phi_v ( X_{\tau_v}), 1-\eta) - d_{\H}( 0, 1-\eta)) \\
	& = \E( d_{\H}( U, 1- \eta)) -\E(d_{\H}(0, 1-\eta)) + O(\eps)\\
	& = \int_{0}^{2\pi} \arcch \left( 1+ \frac{2|\rho e^{i\theta} - 1 + \eta |^2 }{(1 - (1- \eta)^2)(1 - \rho^2)} \right) - \ln \frac{2 - \eta}{\eta} \frac{d \theta}{2\pi} + O(\eps)
	\end{align*}
	where in the second line we use \cref{cor:disc_convergence_uniform} and the third line is just the exact expression of the hyperbolic distance in terms of usual functions.

	Since we are still free to choose $\eta$ small, let us analyse the asymptotic of this expression as $\eta$ goes to $0$. We note that $\arcch (x) = \ln (2x) + O( \frac{1}{x})$ so we have
\begin{align*}
	\E( D_1 - D_0) & = \int_{0}^{2\pi} \ln(  2\frac{2|\rho e^{i\theta} - 1 |^2 }{2\eta(1 - \rho^2)}(1+ O(\eta)) ) - \ln \big(\frac{2}{\eta}(1+O(\eta)) \big) \frac{d \theta}{2\pi}  +O(\eps)\\
	 & = \int_{0}^{2\pi} \ln( \frac{|\rho e^{i\theta} - 1 |^2}{(1 - \rho^2)} ) \frac{d \theta}{2\pi}  + O(\eta) + O(\eps).
\end{align*}
As a function of $\rho$, the integral above has value $0$ at $0$ and it is straightforward to check that its derivative is positive for $0 < \rho < 1$. Indeed, the derivative in $\rho$ of the integral is $$\int \frac{4\rho - \rho^2(4 \cos \theta + 2 \sin \theta) + 2 \sin \theta}{(1 - \rho^2)(1 + \rho^2 - 2 \rho \cos \theta)}.$$ The $\sin \theta$ terms in the numerator have $0$ integral by symmetry so the numerator becomes $4 \rho - 4 \rho^2 \cos \theta$ which is positive for $0< \rho < 1$. Therefore
%It is not hard to check that the integral above is strictly positive so 
we can find $c > 0$ and a choice of $\eta$ and $\eps$ such that $\E( D_1 - D_0) > c$ uniformly over all choices of $v \in \D \setminus B( 0, (1-\eta))$. Since the result is uniform in $v$, it clearly also applies to $\E( D_{n+1}- D_n | X_{\tau_n})$ as long as $X_{\tau_n} \not \in B( 0, (1-\eta))$.  We can therefore see $D_n$ as a submartingale with increments of conditional expectation at least $c$ up to the first entry into $B(0, 1-\eta)$ with is also the first time where $D_n \leq \ln( \frac{2-\eta}{\eta})$. Since $D_n$ has bounded increments, it is easy to see that by taking $D_0$ to be large enough, we can make $\P( \inf_n D_n \leq K)$ as small as we want, which concludes the proof.
\end{proof}

\bibliographystyle{abbrv}
\bibliography{torus}
\end{document}